\documentclass[reqno,A4paper]{amsart}
\usepackage{amsmath}
\usepackage{amssymb}
\usepackage{amsthm}
\usepackage{enumerate}
\usepackage{mathrsfs} %%\mathcal classic,  \mathscr decorative
\usepackage{eqlist}
\usepackage{array}

\usepackage[english]{babel}                             %kieliasetukset sun muut
\usepackage[latin1]{inputenc}
\usepackage{comment}
\usepackage{ae}
\usepackage[leqno]{amsmath}  %kaavat vasemmalle, int.alue alas

\usepackage[dotinlabels]{titletoc}     %piste sis?llysluetteloon

\usepackage{amssymb,latexsym,verbatim} %vaihtoehtoista s?l??
\usepackage{amstext}
\usepackage{amsfonts}
\usepackage{amsbsy}
\usepackage{amsopn}
\usepackage{enumerate}
\usepackage{txfonts}
\usepackage{amsmath, amsthm, amssymb}

\usepackage{color}
\usepackage[usenames,dvipsnames,svgnames,table]{xcolor}
\usepackage[numbers,sort&compress]{natbib}
\usepackage{bm}
\allowdisplaybreaks
\newtheorem{theorem}{Theorem}
\newtheorem{definition}[theorem]{Definition}
\newtheorem{lemma}[theorem]{Lemma}

\newtheorem{proposition}[theorem]{Proposition}
\newtheorem{remark}[theorem]{Remark}
\numberwithin{equation}{section}
\numberwithin{theorem}{section}

\renewcommand{\epsilon}{\varepsilon}
\renewcommand{\rho}{\varrho}

\DeclareMathOperator{\supp}{supp}
\DeclareMathOperator*{\esssup}{ess\,sup}
\DeclareMathOperator*{\essinf}{ess\,inf}
\DeclareMathOperator*{\essosc}{ess\,osc}
\DeclareMathOperator*{\loc}{loc}

\def\XXint#1#2#3{{\setbox0=\hbox{$#1{#2#3}{\int}$ }
\vcenter{\hbox{$#2#3$ }}\kern-.6\wd0}}

\usepackage{graphicx} % for \rotatebox macro

\def\YYint#1#2#3{{\setbox0=\hbox{$#1{#2#3}{\iint}$}
    \vcenter{\hbox{$#2#3$}}\kern-.51\wd0}}

\numberwithin{equation}{section}
\begin{document}

\title[H\"older regularity for nonlocal doubly degenerate parabolic equations]%
{H\"older regularity of weak solutions to nonlocal doubly degenerate parabolic equations
}
\author[Qifan Li]%
{Qifan Li*}

\newcommand{\acr}{\newline\indent}

\address{\llap{*\,}Department of Mathematics\acr
                   School of Mathematics and Statistics\acr
                   Wuhan University of Technology\acr
                   430070, 122 Luoshi Road,
                   Wuhan, Hubei\acr
                   P. R. China}
\email{qifan\_li@yahoo.com,
qifan\_li@whut.edu.cn}

 %Secondary is optional

\begin{abstract}
We study local regularity for nonlocal doubly degenerate parabolic equations. The model equation is
 \begin{equation*}\begin{split}
\partial_t(|u|^{q-1}u)+\mathrm{P}.\mathrm{V}.\int_{\mathbb{R}^n}\frac{|u(x,t)-u(y,t)|^{p-2}(u(x,t)-u(y,t))}{|x-y|^{n+sp}}\,\mathrm{d}y=0,
\end{split}
\end{equation*}
where $0<s<1$, $p>2$ and $0<q<p-1$.
Under a parabolic tail condition, we show that any locally bounded and sign-changing solution is locally H\"older continuous.
Our proof is based on a nonlocal version of
De Giorgi technique and the method of intrinsic scaling.
\end{abstract}
\maketitle
\section{Introduction}
In recent years, there has been tremendous interest in developing regularity theory of doubly nonlinear parabolic equations
whose prototype is
\begin{equation}\label{localequation}\partial_t (|u|^{q-1}u)-\operatorname{div}|Du|^{p-2}Du=0,\end{equation}
for $u=u(x,t):\mathbb{R}^{n+1}\to\mathbb{R}$.
For an account of the De Giorgi-Nash-Moser theory for this kind of equations, we refer the reader to \cite{BDL, BDLS, BHSS, FS, I, Liao1, LS, PV, V} and references therein. In particular, the H\"older regularity of sign-changing solutions
to \eqref{localequation} has been studied in a series papers \cite{BDL, BDLS, LS}.
According to these papers, the doubly nonlinear equations of the form \eqref{localequation} have been treated separately for three cases, which are
the borderline case, i.e., $p>1$ and $q=p-1$, the doubly degenerate case, i.e., $p>2$ and $0<q<p-1$ and the doubly singular case, i.e., $1<p<2$
and $0<p-1<q$.
 In contrast to \cite{I, PV}, which identified the doubly degenerate range by $0<q<1$ and $p>2$.
Recall that the Barenblatt fundamental solution for the equation \eqref{localequation} can be constructed by
 \begin{equation*}
B_{p,q}(x,t)=	\begin{cases}
	 t^{-\frac{n}{\lambda}}\left[1-b\left(t^{-\frac{1}{\lambda}}|x|\right)^\frac{p}{p-1}\right]_+^\frac{(p-1)q}{p-1-q},\quad t>0\\
	0,\quad t\leq0
	\end{cases}
\end{equation*}
where $\lambda=n\frac{p-1-q}{q}+p$ and $b=\frac{(p-1-q)}{p}\lambda^{-\frac{1}{p-1}}$ (see, for example \cite{GS}); it is more natural to consider
a wider range $p>2$ and $0<q<p-1$ to classify the doubly degenerate case \cite{Liao2}. On the other hand, in the case $p-1=q$,
the equation \eqref{localequation} is called \emph{Trudinger equation}.

It is of interest to know whether the H\"older regularity theory for local equation \eqref{localequation} can be extended to the nonlocal case.
To this end, we consider nonlocal doubly nonlinear parabolic equations of the form
 \begin{equation}\begin{split}\label{model}
\partial_t(|u|^{q-1}u)+\mathrm{P}.\mathrm{V}.\int_{\mathbb{R}^n}\frac{|u(x,t)-u(y,t)|^{p-2}(u(x,t)-u(y,t))}{|x-y|^{n+sp}}\,\mathrm{d}y=0
\end{split}
\end{equation}
in this paper. Here, we assume that $s\in(0,1)$. In the borderline case, i.e., $p-1=q$,
the equation \eqref{model} is also called \emph{nonlocal Trudinger equation}. Moreover, the equation \eqref{model} is said to be
\emph{nonlocal doubly degenerate} if $p>2$ and $0<q<p-1$.
In the past ten years, there have been extensive study of nonlocal integro-differential operators.
The study of regularity theory for the fractional $p$-Laplace equation goes back to
Castro, Kuusi and Palatucci \cite{CKP1, CKP2}.
In \cite{APT, BK, BLS, DZZ, Liao}, the authors proved the H\"older continuity results for the time dependent parabolic case.
The case of parabolic De Giorgi classes was subsequently treated by Nakamura \cite{Nakamura}.
Recently, a great deal of mathematical effort has been devoted to the study of nonlocal Trudinger equation.
Banerjee, Garain and Kinnunen \cite{BGK1} established a local boundedness estimate and a reverse H\"older inequality
for solutions to the nonlocal Trudinger equation. Furthermore, Shang and Zhang \cite{SZ} established a weak Harnack inequality
and subsequently Adimurthi \cite{KA} proved the H\"older regularity result for this kind of equation.
Recently, Ciani and Nakamura \cite{CiNa} studied regularity theory for a nonlocal $p$-homogenous De Giorgi class which related to the nonlocal
Trudinger equation.

However, limited
work has been conducted on nonlocal doubly nonlinear equations for $q\neq p-1$.
The first regularity result for the nonlocal doubly nonlinear equations of general type was established by Banerjee, Garain and Kinnunen \cite{BGK2}.
Recently, Kim, Lee and Prasad proved the H\"older continuity for the nonlocal porous medium equation, which is the case of $q>0$ and $p=2$.
Motivated by this work,
we are interested in the H\"older continuity
for the nonlocal doubly degenerate equations. In this paper, we prove that any bounded weak solution to this kind of
equations is locally H\"older continuous. This also extends \cite[Theorem 1.1]{BDLS} to the nonlocal case.

Our approach is in the spirit of \cite{KA, BGK2, BDLS, BK, DZZ, Liao}, which uses the method of intrinsic scaling.
This method has its origin in \cite{Di93, DGV, U}.
Moreover, the derivation of oscillation decay for solutions of the doubly nonlinear equations is always divided into two parts (see \cite{KA, BDL, BDLS, LS}).
The first part concerns the case when $u$ is near zero, while the second part concerns the case when $u$ is away from zero.
In the first part, the degeneracy plays a major role in the proof. In the second part, the equation behaves like the ordinary
$p$-Laplace type equation. Our
proof follows this approach, but the argument is considerably more delicate in the nonlocal doubly degenerate
setting. Contrary to the local case, it is necessary to address a nonlocal tail term in the study of nonlocal problems.
This motivates us to establish a nonlocal version of the De Giorgi theory. Furthermore, a
difficulty appears in the change of variables in the second part of the proof and the argument becomes much more involved than the one for the
local equation.
In order to overcome this difficulty, we have to establish a
new type of Caccioppoli inequality. In contrast to \cite{KA}, our proof makes no appeal to the expansion of positivity.

This paper is built up as follows. In Section 2, we set up notations and state the main result.
Section 3 establishes the Caccioppoli-type estimate. In section 4, we study the case when $u$ is near zero and derive an oscillation decay estimate
for the weak solution.
Section 5 deals with the case when $u$ is away from zero and we finish the proof of the main theorem in this section.
\section{Statement of the main result}
In this section, we introduce the notations and give the statement of the main
result.
Here and subsequently, $\Omega$ stands for a bounded domain in $\mathbb{R}^n$ with $n\geq2$. For $T>0$, we define a space-time cylinder
$\Omega_T=\Omega\times(0,T)$. Given a point $(x_0,t_0)\in\Omega_T$, we set
$Q_{r,s}(z_0)=B_r(x_0)\times (t_0-s,t_0)$ and $Q_r(z_0)=Q_{r,r^{sp}}(z_0)=B_r(x_0)\times (t_0-r^{sp},t_0)$, where
$B_r(x_0)=\{x\in\mathbb{R}^n:|x-x_0|<r\}$. For $\lambda>0$, we define the scaled cylinder by $Q_r^{(\lambda)}(z_0)=B_r(x_0)\times(t_0-\lambda r^{sp},t_0)$. If the reference point $z_0$ is the origin, then we omit in our notation the point $z_0$ and write $B_r$, $Q_r$ and $Q_r^{(\lambda)}$
for $B_r(0)$, $Q_r(0)$ and $Q_r^{(\lambda)}(0)$.

For a given function $\psi\in L^1(\Omega)$, the fractional Gagliardo seminorm of $\psi$ is given by
\begin{equation*}
[\psi]_{W^{s,p}(\Omega)}=\left(\iint_\Omega\frac{|\psi(x)-\psi(y)|^p}{|x-y|^{n+sp}}\,\mathrm{d}x\mathrm{d}y\right)^\frac{1}{p}.
\end{equation*}
The fractional Sobolev space $W^{s,p}(\Omega)$ is defined by
$W^{s,p}(\Omega)=\{\psi\in L^p(\Omega):[\psi]_{W^{s,p}(\Omega)}<\infty\}$.
Moreover, we define
$W^{s,p}_0(\Omega)=\{\psi\in W^{s,p}(\mathbb{R}^n):\psi=0\ \text{in}\ \mathbb{R}^n\setminus \Omega\}$.
A locally integrable function $f:\Omega\to\mathbb{R}^n$ is in the function space $L_{sp}^{p-1}(\mathbb{R}^n)$, if
\begin{equation*}
\int_{\mathbb{R}^n}\frac{|f(y)|^{p-1}}{(1+|y|)^{n+sp}}\,\mathrm{d}y<\infty.
\end{equation*}
For a fixed measurable function $f:\Omega_T\to\mathbb{R}^{n+1}$ and a cylinder $Q=B_r(x_0)\times I$, we denote the nonlocal parabolic tail by
\begin{equation*}
\text{Tail}_m(f;Q)=\left(\fint_I\left(r^{sp}\int_{\mathbb{R}^n\setminus B_r(x_0)}\frac{|f(y,t)|^{p-1}}{|y-x_0|^{n+sp}}\,\mathrm{d}y\right)^\frac{m}{p-1}\ \mathrm{d}t\right)^\frac{1}{m},
\end{equation*}
where $m>p-1$.
Moreover, we write
\begin{equation*}
\widetilde{\mathrm{Tail}}_m(f;Q)=\left(\int_I\left(\int_{\mathbb{R}^n\setminus B_r(x_0)}\frac{|f(y,t)|^{p-1}}{|y-x_0|^{n+sp}}\,\mathrm{d}y\right)^\frac{m}{p-1}\ \mathrm{d}t\right)^\frac{1}{m}.
\end{equation*}
As pointed out by Byun and Kim \cite{BK}, we know that $f\in L^m(I;L_{sp}^{p-1}(\mathbb{R}^n))$ is equivalent to the tail condition where
$\text{Tail}_m(f;Q)<\infty$ holds for all
$B_r(x_0)\subset\mathbb{R}^n$. The concept of the parabolic tail condition goes back as far as \cite{KW}.

For $T>0$ and a Banach space $X$, we denote by $C(0,T;X)$ the space of continuous functions $g:[0,T]\to X$, equipped with the norm $\|\cdot\|_{L^\infty(0,T;X)}$.
This work considers the doubly nonlinear parabolic equations of the nonlocal type
\begin{equation}\begin{split}\label{LKut}
\partial_t(|u|^{q-1}u)+\mathcal{L}_Ku=0
\end{split}
\end{equation}
in $\Omega_T$. Here, the nonlocal operator $\mathcal{L}_K$
is defined by
\begin{equation}\begin{split}\label{LK}
\mathcal{L}_Ku(x,t)=\mathrm{P}.\mathrm{V}.\int_{\mathbb{R}^n}|u(x,t)-u(y,t)|^{p-2}(u(x,t)-u(y,t))K(x,y,t)\,\mathrm{d}y,
\end{split}
\end{equation}
where $\mathrm{P}.\mathrm{V}.$ stands for the principal value. Throughout the paper, we keep $p>2$ and $0<q<p-1$.
We assume that the kernel $K$ is a measurable symmetric kernel with respect to $x$ and $y$, and satisfies
\begin{equation}\begin{split}\label{kernel}
\frac{\Lambda^{-1}}{|x-y|^{n+sp}}\leq K(x,y,t)\leq\frac{\Lambda}{|x-y|^{n+sp}}
\end{split}
\end{equation}
for almost every $x,y\in\mathbb{R}^n$ uniformly in $t\in(0,T)$ for some $\Lambda\geq1$ and $s\in(0,1)$.
We now give the definition of a weak solution to the nonlocal parabolic equation \eqref{LKut}.
\begin{definition}\label{weak solution}
Let $m>p-1$.
A measurable function $u:\Omega_T\to \mathbb{R}$ is
said to be a local weak solution to \eqref{LKut}-\eqref{kernel} if
\begin{equation*}\begin{split}u\in C_{\loc}(0,T;L_{\loc}^{q+1}(\Omega))\cap L_{\loc}^p(0,T;W_{\loc}^{s,p}(\Omega))
\cap L_{\loc}^m(0,T;L_{sp}^{p-1}(\mathbb{R}^n))
\end{split}
\end{equation*}
and
for every open set
$U\Subset \Omega$ and every subinterval $(t_1,t_2)\subset(-T,0)$
the identity
\begin{equation}\begin{split}\label{weaksolution}
\int_U &|u(\cdot,t)|^{q-1}u(\cdot,t)\varphi(\cdot,t)\,\mathrm{d}x\bigg|_{t=t_1}^{t_2}+\iint_{U\times(t_1,t_2)}-|u|^{q-1}u\partial_t
 \varphi
\,\mathrm {d}x\mathrm {d}t
\\&+\int_{t_1}^{t_2}\iint_{\mathbb{R}^n\times\mathbb{R}^n}K(x,y,t)|u(x,t)-u(y,t)|^{p-2}(u(x,t)-u(y,t))
\\&\qquad\qquad\times(\varphi(x,t)-\varphi(y,t))
\,\mathrm {d}x\mathrm {d}t=0
\end{split}\end{equation}
holds for
any function
$\varphi\in W_{\loc}^{1,q+1}(0,T;L^{q+1}(U))\cap L_{\loc}^p(0,T;W_0^{s,p}(U))$.
 \end{definition}
 The statement that a constant $\gamma$ depends only
on the data means that it can be quantitatively determined a priori only in terms of
$\{n,m,p,q,s,\Lambda\}$. We are now in a position to state our main result.
 \begin{theorem}\label{main1}
Let $p>2$ and $0<q<p-1$.
 Let $u$ be a locally bounded weak solution to \eqref{LKut}-\eqref{kernel}
 in the sense of Definition \ref{weak solution}.
 Then, $u$ is locally H\"older continuous in $\Omega_T$. More precisely, for any $\mathfrak z_0\in\Omega_T$ and let $R>0$ be a fixed number such that $Q_{R}(\mathfrak z_0)\subset\Omega_T$.
 Then, there exist positive constants $\delta<1$, $\alpha<1$, $c<1$ and $\gamma_0>1$, such that the oscillation
estimate
  \begin{equation}\begin{split}\label{theorem1oscillation}
 &\essosc_{Q_{r,cr^{sp}}(z_0)}u\leq \gamma_0\left(\frac{r}{R}\right)^{\alpha}
 \end{split}\end{equation}
 holds for any $z_0\in Q_{\frac{1}{4}R}(\mathfrak z_0)$ and $r<\delta R$. Here, the constant $\alpha$
 depends only upon the data, and the constants $\delta$, $c$ and $\gamma_0$ depend on the data, $R$, $\|u\|_\infty$ and $\mathrm{Tail}_m(|u|;Q_{R}(\mathfrak z_0))$.
 \end{theorem}
 Next, we provide some preliminary lemmas. To start with, we follow the
notation used in \cite{BDL}.
 For $b\in\mathbb{R}$ and $\alpha>0$, we define $\bm{b}^\alpha=|b|^{\alpha-1}b$ if $b\neq0$ and $\bm{b}^\alpha=0$ if $b=0$.
 For $w$, $k\in \mathbb{R}$, we introduce the quantities
\begin{equation}\label{g+-}
g_{\pm}(w,k)=\pm q\int_k^w|s|^{q-1}(s-k)_{\pm}\,\mathrm{d}s.
\end{equation}
We now state the following two lemmas that characterize the quantities $|\bm{b}^\alpha-\bm{a}^\alpha|$ and
$g_{\pm}(w,k)$.
 \begin{lemma}\cite[Lemma 2.1]{BDL}\label{equivalent} There exists a positive constant $\gamma=\gamma(\alpha)$ such that
the inequality
\begin{equation}\label{abequivalent}
\gamma^{-1}|\bm{b}^\alpha-\bm{a}^\alpha|\leq (|a|+|b|)^{\alpha-1}|a-b|\leq\gamma|\bm{b}^\alpha-\bm{a}^\alpha|
\end{equation}
holds for any $a,b\in\mathbb{R}$ and $\alpha>0$.
\end{lemma}
\begin{lemma}\cite[Lemma 2.2]{BDL}\label{equivalent} There exists a positive constant $\gamma=\gamma(q)$ such that
the inequality
\begin{equation}\label{g+-equivalent}
\gamma^{-1}(|w|+|k|)^{q-1}(w-k)_{\pm}^2\leq
g_{\pm}(w,k)\leq\gamma(|w|+|k|)^{q-1}(w-k)_{\pm}^2
\end{equation}
holds for any $k$ and $w\in\mathbb{R}$.
\end{lemma}
The next lemma asserts that the parabolic fractional Sobolev space is embedded in the
mixed-norm Lebesgue space.
\begin{lemma}\cite[Lemma 2.2]{BK}\label{embedding} Let $1<m_1\leq\max\{p,2\}$. Suppose that $\bar q$ and $\bar r$ are the positive constants
satisfying
\begin{equation*}\bar r>p\qquad\text{and}\qquad\frac{1}{\bar q}+\frac{1}{\bar r}\left(\frac{sp}{n}+\frac{p}{m_1}-1\right)=\frac{1}{m_1}.
\end{equation*}
Then, there exists a constant $\gamma>0$, such that
\begin{equation*}\begin{split}
\|v\|_{L^{\bar q,\bar r}(B_R\times I)}^\frac{n\bar q\bar r}{sp\bar q+n\bar r}\leq \gamma\int_I[v(\cdot,t)]_{W^{s,p}(B_R)}^p\,\mathrm{d}t
+\gamma R^{-sp}\|v\|_{L^p(I;L^p(B_R))}^p+\gamma \esssup_{t\in I}\|v(\cdot,t)\|_{L^{m_1}(B_R)}^{m_1}.
\end{split}
\end{equation*}
\end{lemma}
Let $k$ be a fixed real number. For a function $v\in L_{\loc}^1(\Omega_T)$, the truncations are defined by
\begin{equation*}\begin{split}&(v-k)_+=\max\{v-k;0\}\quad\text{and}\quad(v-k)_-=\max\{-(v-k);0\}.\end{split}\end{equation*}
The next lemma provides a comparison for the fractional Gagliardo seminorms of different truncations.
\begin{lemma}\cite[Lemma 2.3]{BK}\label{b<a} Let $f\in W^{s,p}(B_R)$ and $b\leq a$. Then, we have
\begin{equation*}\begin{split}
[(f-b)_-]_{W^{s,p}(B_R)}\leq [(f-a)_-]_{W^{s,p}(B_R)}.
\end{split}
\end{equation*}
\end{lemma}
We remark that the condition for $a$ and $b$ stated in \cite[Lemma 2.3]{BK} is just $0\leq b\leq a$. However, this condition can be relaxed to
$-\infty<b\leq a$. It can be easily seen that
if $f\in W^{s,p}(B_R)$, then $f+|b|\in W^{s,p}(B_R)$. Since
$0\leq b+|b|\leq a+|b|$, we have
\begin{equation*}\begin{split}
&[(f-b)_-]_{W^{s,p}(B_R)}=[(f+|b|-(b+|b|))_-]_{W^{s,p}(B_R)}
\\&\leq[(f+|b|-(a+|b|))_-]_{W^{s,p}(B_R)}
=[(f-a)_-]_{W^{s,p}(B_R)},
\end{split}
\end{equation*}
which proves the desired inequality.
\section{Caccioppoli inequality}
In this section, we establish a Caccioppoli inequality
for the weak solutions.
Let $R>0$, $S>0$ and $z_0=(x_0,t_0)$ be a fixed point. Assume that $Q_{R,S}(z_0)\subset\Omega_T$. We denote by
$\varphi$ a piecewise smooth function in $Q_{R,S}(z_0)$ such that
\begin{equation}\label{def zeta}0\leq\varphi\leq1, \quad|D\varphi|<\infty\quad\text{and}\quad\varphi=0\quad\text{
on}\quad B_R(x_0)\setminus B_{(1-\sigma)R}(x_0),\end{equation}
where $\sigma\in(0,1)$. The next Lemma establishes the desired Caccioppoli-type inequality.
\begin{lemma}\label{caclemma}
Let $u$ be a weak solution to \eqref{LKut}-\eqref{kernel}
 in the sense of Definition \ref{weak solution}.
 There exists a positive constant
$\gamma$ depending only upon the data, such that for every piecewise smooth cutoff function
$\varphi$ satisfying \eqref{def zeta}, there holds
\begin{equation}\begin{split}\label{Cacinequality}
&\esssup_{t_0-S<t<t_0}\int_{B_R(x_0)\times\{t\}}g_{\pm}(u,k)\varphi^p \,\mathrm {d}x
\\&+\iint_{Q_{R,S}(z_0)}(u-k)_\pm(x,t)\varphi(x,t)^p\left(\int_{B_R(x_0)}\frac{(u-k)_\mp(y,t)^{p-1}
}{|x-y|^{n+sp}}\,\mathrm{d}y\right)\,\mathrm {d}x\mathrm {d}t\\
&+\int_{t_0-S}^{t_0}\iint_{B_R(x_0)\times B_R(x_0)}\min\left\{\varphi(x,t),\varphi(y,t)\right\}^p
\\&\qquad\qquad\qquad\qquad\times\frac{|(u-k)_\pm(x,t)-(u-k)_\pm(y,t)|^p}{|x-y|^{n+sp}}\,\mathrm {d}x\mathrm {d}y\mathrm {d}t
\\
&\leq \int_{B_R(x_0)\times\{t_0-S\}}g_{\pm}(u,k)\varphi^p \,\mathrm {d}x+\gamma
\iint_{Q_{R,S}(z_0)} g_{\pm}(u,k)|\partial_t\varphi^p|\,\mathrm {d}x\mathrm {d}t
\\&+\gamma\int_{t_0-S}^{t_0}\iint_{B_R(x_0)\times B_R(x_0)}\max\left\{(u-k)_\pm(x,t),(u-k)_\pm(y,t)\right\}^p
\\&\qquad\qquad\qquad\qquad \times\frac{|\varphi(x,t)-\varphi(y,t)|^p}{|x-y|^{n+sp}}
\,\mathrm {d}x\mathrm {d}y\mathrm {d}t
\\&
+\gamma \sigma^{-n-sp}\int_{t_0-S}^{t_0}\int_{\mathbb{R}^n\setminus B_R(x_0)}\int_{B_R(x_0)}
\frac{(u-k)_\pm(y,t)^{p-1}}{|y-x_0|^{n+sp}}
(u-k)_\pm(x,t)\varphi(x,t)^p\,\mathrm {d}x\mathrm {d}y\mathrm {d}t.
\end{split}\end{equation}
\end{lemma}
\begin{proof}
The proof of \eqref{Cacinequality} follows in a similar manner as the arguments in \cite[Proposition 3.1]{BDL}, \cite[Lemma 3.2]{BGK2},
\cite[Lemma 2.5]{BK} and \cite[Proposition 2.1]{Liao},
and we only focus on the difficulties that really arise for equation \eqref{LKut}-\eqref{kernel}.
To this end, we only treat the case of $(u-k)_+$ as the other case is similar.
First, we introduce the time mollification
\begin{equation}\begin{split}\label{timemollifierexp}
[u]_{\bar h}(x,t)=\tfrac{1}{h}\int_t^Te^\frac{t-\tau}{h}u(x,\tau)\,\mathrm{d}\tau.
\end{split}\end{equation}
For any fixed $t_1\in (t_0-S,t_0)$, we introduce a Lipschitz function $\zeta_\epsilon(t)$ satisfying
$\zeta_\epsilon(t)=1$ in $(t_0-S+\epsilon,t_1-\epsilon)$, $\zeta_\epsilon(t)=0$ in $(-\infty,t_0-S]\cup [t_0,+\infty)$, $\partial_t\zeta_\epsilon(t)
=1/\epsilon$ in $(t_0-S,t_0-S+\epsilon)$ and $\partial_t\zeta_\epsilon(t)
=-1/\epsilon$ in $(t_1-\epsilon,t_1)$. In the weak formulation \eqref{weaksolution}, we choose the testing function
\begin{equation*}\begin{split}
\varphi_h=([u]_{\bar h}-k)_+\varphi^p\zeta_\epsilon.
\end{split}\end{equation*}
The first two terms in \eqref{weaksolution} can be decomposed as
\begin{equation*}\begin{split}
&\int_{t_0-S}^{t_1}\int_{B_R(x_0)}-|u|^{q-1}u\partial_t
 \varphi_h
\,\mathrm {d}x\mathrm {d}t
\\=&\int_{t_0-S}^{t_1}\int_{B_R(x_0)}\left[\phi([u]_{\bar h})-\phi(u)\right]\partial_t
 \varphi_h
\,\mathrm {d}x\mathrm {d}t
\\&+\int_{t_0-S}^{t_1}\int_{B_R(x_0)}-\phi([u]_{\bar h})\partial_t
 \varphi_h\,\mathrm {d}x\mathrm {d}t=:T_1+T_2,
\end{split}\end{equation*}
where we set $\phi(s)=|s|^{q-1}s$.
We first consider the estimate of $T_1$.
Noting that $\phi(s)$ is an increasing function, $\partial_t [u]_{\bar h}=h^{-1}([u]_{\bar h}-u)$
and $|\phi([u]_{\bar h})-\phi(u)|\leq \gamma(q)(|[u]_{\bar h}|+|u|)^{q-1}|[u]_{\bar h}-u|$,
we get
\begin{equation*}\begin{split}
T_1\geq -\gamma\int_{t_0-S}^{t_1}\int_{B_R(x_0)} (|[u]_{\bar h}|+|u|)^{q-1}|[u]_{\bar h}-u|([u]_{\bar h}-k)_+\partial_t
 (\varphi^p\zeta_\epsilon)
\,\mathrm {d}x\mathrm {d}t\to 0,
\end{split}\end{equation*}
as $h\downarrow0$. To estimate $T_2$, we observe that $\partial_tg_+([u]_{\bar h},k)=q(\partial_t [u]_{\bar h})|[u]_{\bar h}|^{q-1}([u]_{\bar h}-k)_+$.
Integrating by parts, we obtain
\begin{equation*}\begin{split}
T_2&=-\int_{t_0-S}^{t_1}\int_{B_R(x_0)}g_+([u]_{\bar h},k)
 \zeta_\epsilon\partial_t\varphi^p\,\mathrm {d}x\mathrm {d}t
 \\&\quad-\int_{t_0-S}^{t_1}\int_{B_R(x_0)}g_+([u]_{\bar h},k)
\varphi^p\partial_t\zeta_\epsilon\,\mathrm {d}x\mathrm {d}t.
\end{split}\end{equation*}
At this point, we pass to the limit $h\downarrow0$ and then pass to the limit $\epsilon\downarrow0$. This leads to
\begin{equation*}\begin{split}
\lim_{\epsilon\to0}\lim_{h\to0}&T_2=-\int_{t_0-S}^{t_1}\int_{B_R(x_0)}g_+(u,k)\partial_t\varphi^p\,\mathrm {d}x\mathrm {d}t
\\&+\int_{B_R(x_0)}[g_+(u,k)\varphi^p](\cdot,t_1)\,\mathrm{d}x-\int_{B_R(x_0)}[g_+(u,k)\varphi^p](\cdot,t_0-S)\,\mathrm{d}x.
\end{split}\end{equation*}
The rest of the proof is similar to that of \cite[Lemma 2.5]{BK} and \cite[Proposition 2.1]{Liao}, and so is omitted.
\end{proof}
We now turn our attention to the proof of Theorem \ref{main1}.
Without loss of generality, we assume that $\mathfrak z_0=(0,0)$. We fix a constant $\xi_0\in(0,1)$, which will be chosen later in a universal way.
For a fixed $R>0$ and a cylinder $Q_R=B_R\times (-R^{sp},0)$, we introduce a quantity
\begin{equation}\begin{split}\label{omegaxi0}
\omega=4\xi_0^{-1}\sup_{Q_R}|u|+4\text{Tail}_m(|u|;Q_R)^\frac{m}{m-(p-q-1)}.
\end{split}\end{equation}
Next, we define $\theta=(\frac{1}{4}\omega)^{q+1-p}$. Let $A>1$ be a constant which will be determined later.
We now fix the point $z_0=(x_0,t_0)\in Q_{\frac{1}{4}R}$ and choose
\begin{equation}\begin{split}\label{firstrho}
\rho_0=\min\left\{\tfrac{1}{4}R,(1-4^{-sp})^\frac{1}{sp}\theta^{-\frac{1}{sp}}
A^{-\frac{1}{sp}}R\right\}.\end{split}\end{equation}
We abbreviate $Q_0$ to $Q_{\rho_0}^{(A\theta)}(z_0)$. It follows from \eqref{firstrho} that $Q_0\subseteq Q_R$.
Moreover, we set $\mu_+=\esssup_{Q_0}u$ and $\mu_-=\essinf_{Q_0}u$. It is easily seen that
$\essosc_{Q_0}u=\mu_+-\mu_-\leq \omega$.
Henceforth, we set
\begin{equation}\begin{split}\label{kappa}
\kappa=\frac{sp}{n}\left(\frac{m-(p-1)}{m}\right).
\end{split}\end{equation}
In order to prove the H\"older estimate \eqref{theorem1oscillation}, we will prove the following assertion. There exist a sequence
of numbers $\{\omega_j\}$ and
a family of nested and shrinking cylinders $Q_j$, with the vertex at $z_0$, such that $\essosc_{Q_j}u\leq\omega_j$.
Here, the sequence $\{\omega_j\}$ converges to zero and $Q_j$ shrink to the point $z_0$, as $j$ tends to infinity.
\section{Reduction of oscillation near zero}
In this section, we proceed to the proof of \eqref{theorem1oscillation}.
From \eqref{omegaxi0}, we see that $\mu_-\leq\xi_0\omega$ and $\mu_+\geq-\xi_0\omega$.
In this case, we find that $|\mu_\pm|\leq 2\omega$, since $\xi_0<1$. Moreover, we assume that $\mu_+-\mu_->\frac{1}{2}\omega$.
This assumption yields that $\mu_+\geq\frac{1}{4}\omega$ or $\mu_-\leq-\frac{1}{4}\omega$.
In this section, we only consider the case $\mu_+\geq\frac{1}{4}\omega$, since the other case can be treated analogously.
Let $\rho<\rho_0$ be a fixed radius. We now consider two complementary cases.
For a fixed constant $\nu_0>0$, we see that either
  \begin{itemize}
 \item[$\bullet$]
 \textbf{The first alternative}. There exists $t_0-(A-1)\theta\rho^{sp}\leq \bar t\leq t_0$ such that
\begin{equation}\label{1st}\left|\left\{(x,t)\in Q_\rho^{(\theta)}(x_0,\bar t)
:u\leq \mu_-+\tfrac{1}{4}\omega\right\}\right|\leq \nu_0|Q_\rho^{(\theta)}(x_0,\bar t)|\end{equation}
\end{itemize}
or this does not hold. More precisely, if \eqref{1st} does not hold, we infer that the
 following second alternative holds.
  \begin{itemize}
 \item[$\bullet$]
 \textbf{The second alternative}. For any $t_0-(A-1)\theta\rho^{sp}\leq \bar t\leq t_0$, there holds
\begin{equation}\label{2nd}\left|\left\{(x,t)\in Q_\rho^{(\theta)}(x_0,\bar t)
:u\leq\mu_-+\tfrac{1}{4}\omega\right\}\right|> \nu_0|Q_\rho^{(\theta)}(x_0,\bar t)|.\end{equation}
\end{itemize}
Here, the constant $\nu_0$ will be fixed in the course of the proof of Lemma \ref{lemmaDeGiorgi1},
while the value of $A$ will be determined during the proof of Proposition \ref{2nd proposition}.
\subsection{The first alternative}
This subsection deals with the first alternative and we establish a decay estimate for essential oscillation of
$u$. A key ingredient in the proof is the following De Giorgi-type lemma.
\begin{lemma}\label{lemmaDeGiorgi1}
Let $u$ be a locally bounded weak solutions to \eqref{LKut}-\eqref{kernel}
 in the sense of Definition \ref{weak solution}. Assume that $Q_{\rho}^{(\theta)}(x_0,\bar t)\subset Q_0$, $|\mu_-|<2\omega$
 and
 \begin{equation}\begin{split}\label{firstTail}
\frac{\rho^\frac{n\kappa}{p-1}}{\theta^\frac{1}{m}}\widetilde{\mathrm{Tail}}_m((u-\mu_-)_-;Q_0)\leq\tfrac{1}{4}\omega.\end{split}\end{equation}
 There exists a constant $\nu_0\in(0,1)$, depending only on the data, such that if
\begin{equation}\label{1st assumption}\large|\large\{(x,t)\in Q_{\rho}^{(\theta)}(x_0,\bar t):u(x,t)<\mu_-+\tfrac{1}{4}\omega\large\}\large
|\leq \nu_0|Q_{\rho}^{(\theta)}(x_0,\bar t)|,\end{equation}
then
\begin{equation}
\label{DeGiorgi1}u(x,t)\geq\mu_-+\tfrac{1}{8}\omega\qquad\text{for}\ \ \text{a.e.}\ \ (x,t)\in Q_{\frac{1}{2}
\rho}^{(\theta)}(x_0,\bar t).\end{equation}
\end{lemma}
\begin{proof} For $j=0,1,2,\cdots$, we define the sequences $k_j=\mu_-+2^{-3}\omega+2^{-3-j}\omega$, $\tilde k_j=\frac{1}{2}(k_j+k_{j+1})$,
$\rho_j=2^{-1}\rho+2^{-j-1}\rho$, $\tilde\rho_j=\frac{1}{2}(\rho_j+\rho_{j+1})$ and $\hat\rho_j=\frac{1}{2}(\tilde\rho_j+\rho_{j+1})$.
For simplicity of notation, we write $B_j=B_{\rho_j}(x_0)$, $\tilde B_j=B_{\tilde\rho_j}(x_0)$,
$\tilde Q_j=Q_{\tilde\rho_j}^{(\theta)}(x_0,\bar t)$ and $\hat Q_j=Q_{\hat\rho_j}^{(\theta)}(x_0,\bar t)$. Moreover, we take a cutoff function
$0\leq\varphi_j\leq1$, such that $\varphi_j=1$ in $Q_{\rho_{j+1}}^{(\theta)}(x_0,\bar t)$, $\varphi_j=0$ on $\partial_P\hat Q_j$,
\begin{equation*}\begin{split}
|D\varphi_j|\leq 2^{j+3}\rho^{-1}\qquad\text{and}\qquad
|\partial_t\varphi_j|\leq c\theta^{-1}2^j\rho^{-sp}.
\end{split}\end{equation*}
With the choice of $\varphi=\varphi_j$ in Lemma \ref{caclemma}, we use the Caccioppoli inequality to the truncated functions $(u-k_j)_-$ over the cylinder $\tilde Q_j$. This leads to
\begin{equation*}\begin{split}
&\esssup_{\bar t-\theta\tilde\rho_j^{sp}<t<\bar t}\int_{B_{j+1}\times\{t\}}g_-(u,k_j) \,\mathrm {d}x
\\
&+\int_{\bar t-\theta\tilde\rho_j^{sp}}^{\bar t}\iint_{B_{j+1}\times B_{j+1}}
\frac{|(u-k_j)_-(x,t)-(u-k_j)_-(y,t)|^p}{|x-y|^{n+sp}}\,\mathrm {d}x\mathrm {d}y\mathrm {d}t
\\
&\leq \gamma 2^j\theta^{-1}\rho^{-sp}
\iint_{\tilde Q_j} g_-(u,k_j)\,\mathrm {d}x\mathrm {d}t
+\gamma\frac{2^{jp}}{\rho^{sp}}\iint_{\tilde Q_j}(u-k_j)_-^p
\,\mathrm {d}x\mathrm {d}t
\\&
+\gamma 2^{(n+sp)j}\int_{\bar t-\theta\tilde\rho_j^{sp}}^{\bar t}\int_{\mathbb{R}^n\setminus \tilde B_j}\int_{\tilde B_j}
\frac{(u-k_j)_-(y,t)^{p-1}}{|y-x_0|^{n+sp}}(u-k_j)_-(x,t)\,\mathrm {d}x\mathrm {d}y\mathrm {d}t.
\end{split}\end{equation*}
At this point, we introduce a new function $v(x,t)=u(x_0+x,\theta t+\bar t)$ and the above inequality can be written as
\begin{equation*}\begin{split}
T_{0,1}+T_{0,2}&:=\theta^{-1}\esssup_{-\tilde\rho_j^{sp}<t<0}\int_{B_{\rho_{j+1}}\times\{t\}}g_-(v,k_j) \,\mathrm {d}x
\\
&+\int_{-\tilde\rho_j^{sp}}^{0}\iint_{B_{\rho_{j+1}}\times B_{\rho_{j+1}}}
\frac{|(v-k_j)_-(x,t)-(v-k_j)_-(y,t)|^p}{|x-y|^{n+sp}}\,\mathrm {d}x\mathrm {d}y\mathrm {d}t
\\
&\leq \gamma 2^j\theta^{-1}\rho^{-sp}
\iint_{Q_{\tilde\rho_j}} g_-(v,k_j)\,\mathrm {d}x\mathrm {d}t
+\gamma\frac{2^{jp}}{\rho^{sp}}\iint_{Q_{\tilde\rho_j}}(v-k_j)_-^p
\,\mathrm {d}x\mathrm {d}t
\\&
+\gamma 2^{(n+sp)j}\int_{-\tilde\rho_j^{sp}}^0\int_{\mathbb{R}^n\setminus B_{\tilde \rho_j}}\int_{B_{\tilde \rho_j}}
\frac{(v-k_j)_-(y,t)^{p-1}}{|y|^{n+sp}}(v-k_j)_-(x,t)\,\mathrm {d}x\mathrm {d}y\mathrm {d}t
\\&=:T_1+T_2+T_3.
\end{split}\end{equation*}
Next, for  $j=0,1,2,\cdots$, we set
\begin{equation*}\begin{split}
Y_j=\frac{1}{|Q_{\rho_j}|}|\{(x,t)\in Q_{\rho_j}:v(x,t)\leq k_j\}|
\end{split}\end{equation*}
and
\begin{equation*}\begin{split}
Z_j=\left(\fint_{-\rho_j^{sp}}^0\left(\frac{1}{|B_j|}|A_-(v;0,\rho_j,k_j)|\right)^\frac{m}{m-(p-1)}\,\mathrm{d}t\right)^\frac{m-(p-1)}{m(1+\kappa)}.
\end{split}\end{equation*}
Noting that $(v-k_j)_-\leq \frac{1}{4}\omega$ in $Q_\rho$, we obtain $T_2\leq\gamma 2^{jp}\rho^{-sp}\omega^p|Q_{\rho_j}|Y_j$.
To estimate $T_1$, we first observe that $\mu_-\leq v\leq k_j\leq \mu_-+\frac{1}{4}\omega$ on the set $Q_{\rho_j}\cap \{v\leq k_j\}$. Since $|\mu_-|\leq2\omega$, we get $|v|+|k_j|\leq |\mu_-|+|\mu_-+\frac{1}{4}\omega|\leq 5\omega$ on the set $Q_{\rho_j}\cap \{v\leq k_j\}$. It follows from
Lemma \ref{equivalent} that
\begin{equation*}\begin{split}
g_-(v,k_j)&\leq\gamma(|v|+|k_j|)^{q-1}(v-k_j)_-^2
\\&\leq \gamma(|v|+|k_j|)^{q}(v-k_j)_-\leq \gamma\omega^{q+1}\chi_{\{v\leq k_j\}},
\end{split}
\end{equation*}
since $(v-k_j)_-\leq |v|+|k_j|$. Recalling that $\theta=\left(\frac{1}{4}\omega\right)^{q+1-p}$, we have
\begin{equation*}\begin{split}
T_1\leq \gamma \frac{2^j}{\omega^{q+1-p}}\omega^{q+1}Y_j\rho^{-sp}|Q_{\rho_j}|\leq \gamma 2^{jp}\rho^{-sp}\omega^p|Q_{\rho_j}|Y_j.
\end{split}
\end{equation*}
The estimate of $T_3$ is quite similar to the estimate of $I_3$ in the proof of \cite[Lemma 4.1]{BK}, and this gives
\begin{equation*}\begin{split}
T_3&\leq \gamma 2^{j(n+sp)}\rho^{-sp}\omega^p|Q_{\rho_j}|Y_j
\\&\quad+2^{j(n+sp)}\theta^{-\frac{p-1}{m}}\widetilde{\mathrm{Tail}}_m((u-\mu_-)_-;Q_0)
^{p-1}\omega \rho^{sp\frac{m-(p-1)}{m}}|B_\rho|Z_j^{1+\kappa}
\\&\leq \gamma 2^{j(n+sp)}\rho^{-sp}\omega^p|Q_{\rho_j}|(Y_j+Z_j^{1+\kappa}),
\end{split}
\end{equation*}
where we used \eqref{firstTail} in the last line.
Combining these inequalities, we conclude that
\begin{equation*}\begin{split}
T_{0,1}+T_{0,2}\leq \gamma 2^{j(n+sp)}\rho^{-sp}\omega^p|Q_{\rho_j}|(Y_j+Z_j^{1+\kappa}).
\end{split}
\end{equation*}
On the other hand, we apply H\"older's inequality and Lemma \ref{embedding} to obtain
\begin{equation*}\begin{split}
Y_{j+1}|Q_{\rho_{j+1}}|&\leq \iint_{Q_{\rho_{j+1}}}\frac{(v-\tilde k_j)_-^p}{(\tilde k_j-k_{j+1})^p}\,\mathrm{d}x\mathrm{d}t
\\&\leq 2^{(j+2)p}\omega^{-p}|Q_{\rho_{j+1}}|^\frac{sp}{n+sp}Y_j^\frac{sp}{n+sp}\left[
\int_{-\rho_{j+1}^{sp}}^0\left[(v-\tilde k_j)_-(\cdot,t)\right]_{W^{s,p}(B_{\rho_{j+1}})}^p\,\mathrm{d}t
\right.
\\&\left.
+\rho_{j+1}^{-sp}\|(v-\tilde k_j)_-\|_{L^p(Q_{\rho_{j+1}})}^p+\esssup_{-\rho_{j+1}^{sp}
<t<0}\int_{B_{\rho_{j+1}}}(v-\tilde k_j)_-^p(x,t)\,\mathrm{d}x
\right]
\\&=:T_4+T_5+T_6.
\end{split}
\end{equation*}
Since $\tilde k_j\leq k_j$, we use Lemma \ref{b<a} to obtain
\begin{equation*}\begin{split}
T_4\leq 2^{(j+2)p}\omega^{-p}|Q_{\rho_{j+1}}|^\frac{sp}{n+sp}Y_j^\frac{sp}{n+sp}T_{0,2}\leq \gamma 2^{(n+2p)j}
|Q_{\rho_j}|Y_j^\frac{sp}{n+sp}(Y_j+Z_j^{1+\kappa}).
\end{split}
\end{equation*}
In view of $(v-\tilde k_j)_-\leq \frac{1}{4}\omega$ in $Q_\rho$, we deduce $T_5\leq \gamma 2^{jp}Y_j^{1+\frac{sp}{n+sp}}|Q_{\rho_{j+1}}|$.
Finally, we come to the estimate of $T_6$. We first observe that the inequality
\begin{equation*}\begin{split}
|v|+|k_j|\geq k_j-v>k_j-\tilde k_j=\frac{1}{2}(k_j-k_{j+1})=2^{-j-4}\omega
\end{split}
\end{equation*}
holds on the set $\{v<\tilde k_j\}\cap Q_{\rho_{j+1}}$. Recalling that $|v|+|k_j|\leq 5\omega$ on the set $Q_{\rho_j}\cap \{v\leq k_j\}$, we infer from Lemma \ref{equivalent} that on the set $\{v<\tilde k_j\}\cap Q_{\rho_{j+1}}$, there holds
\begin{equation*}\begin{split}
g_-(v,k_j)&\geq \gamma(|v|+|k_j|)^{q-1}(v-k_j)_-^2\geq \gamma 2^{-qj}\omega^{q-1}(v-k_j)_-^{2-p}(v-k_j)_-^p
\\&\geq \gamma 2^{-qj}\omega^{1+q-p}(v-k_j)_-^p=\gamma 2^{-qj}\theta(v-k_j)_-^p,
\end{split}
\end{equation*}
since $p>2$, $\theta=\left(\frac{1}{4}\omega\right)^{q+1-p}$ and $(v-k_j)_-\leq \frac{1}{4}\omega$ in $Q_\rho$.
Therefore, we arrive at
\begin{equation*}\begin{split}
T_6&\leq \gamma 2^{jp}\omega^{-p}|Q_{\rho_{j+1}}|^\frac{sp}{n+sp}Y_j^\frac{sp}{n+sp}\esssup_{-\rho_{j+1}^{sp}
<t<0}\int_{B_{\rho_{j+1}}}2^{jq}\theta^{-1}g_-(v,k_j)\,\mathrm{d}x
\\&\leq \gamma 2^{j(p+q)}\omega^{-p}|Q_{\rho_{j+1}}|^\frac{sp}{n+sp}Y_j^\frac{sp}{n+sp} T_{0,1}
\\&\leq \gamma 2^{j(p+q)}|Q_{\rho_j}|Y_j^\frac{sp}{n+sp}(Y_j+Z_j^{1+\kappa}).
\end{split}
\end{equation*}
Consequently, we infer that
\begin{equation}\begin{split}\label{Yj+1}
Y_{j+1}\leq \gamma 2^{j(2p+n+q)}\left(Y_j^{1+\frac{sp}{n+sp}}+Y_j^\frac{sp}{n+sp}Z_j^{1+\kappa}\right).
\end{split}
\end{equation}
Finally, we consider the estimate of $Z_{j+1}$. We use H\"older's inequality and Lemma \ref{embedding} to obtain
\begin{equation*}\begin{split}
\rho_j^nZ_{j+1}&\leq \gamma 2^{jp}\omega^{-p}\|(v-\tilde k_j)_-\|_{L^{\hat m,\frac{m}{m-(p-1)}\hat m}(Q_{{\rho_{j+1}}})}^p
\\&\leq \gamma 2^{jp}\omega^{-p}\left(2^{jq}T_{0,1}+T_{0,2}+\omega^p\rho^{-sp}Y_j|Q_{\rho_{j+1}}|\right)
\\&\leq \gamma\rho^n2^{(n+2p+q)j}\left(Y_j+Z_j^{1+\kappa}\right),
\end{split}
\end{equation*}
where $\hat m=p(1+\kappa)$. This implies that
\begin{equation}\begin{split}\label{Zj+1}
Z_{j+1}&\leq \gamma 2^{(n+2p+q)j}\left(Y_j+Z_j^{1+\kappa}\right).
\end{split}
\end{equation}
At this point, we set
\begin{equation}\begin{split}\label{nu0}
\nu_0=(4\gamma)^{-\frac{m(1+\kappa)}{(m-(p-1))\zeta}}2^{-(n+2p+q)\frac{m(1+\kappa)}{(m-(p-1))\zeta^2}},
\end{split}
\end{equation}
where $\zeta=\min\left\{\kappa,\frac{sp}{n+sp}\right\}$.
Using a lemma on fast geometric convergence of sequences (see for instance \cite[Chapter I, Lemma 4.2]{Di93}),
we infer from \eqref{Yj+1} and \eqref{Zj+1} that $Y_n\to0$ as $n\to\infty$. Transforming
back to the original function $u$, we obtain the desired estimate \eqref{DeGiorgi1}.
This
completes the proof of Lemma \ref{lemmaDeGiorgi1}.
\end{proof}
Next, we establish
a variant of DeGiorgi-type lemma, involving the initial data.
\begin{lemma}\label{lemmaDeGiorgi2}
Let $u$ be a locally bounded weak solutions to \eqref{LKut}-\eqref{kernel}
 in the sense of Definition \ref{weak solution}. Let $\xi\in(0,1)$ be a fixed constant.
 Set
  \begin{equation}\label{hattheta}
  \hat\theta=
	\begin{cases}
	(\xi\omega)^{q+1-p},\qquad \text{if}\quad |\mu_-|\leq \xi\omega\quad\text{or}\quad q\geq1,  \\
	\omega^{q-1}(\xi\omega)^{2-p}, \qquad \text{if}\quad -2\omega\leq \mu_-\leq -\xi\omega \quad\text{and}\quad q<1.
	\end{cases}
\end{equation}
Let $t_1\in(t_0-A\theta\rho^{sp},t_0)$ be a fixed time level. Assume that
\begin{equation}\label{2st assumption}u(\cdot,t_1)\geq \mu_-+\xi\omega\quad\text{a.e.}\quad\text{in}\quad B_\rho(x_0).\end{equation}
 There exists a constant $\nu_1\in(0,1)$, depending only on the data, such that if
  \begin{equation}\begin{split}\label{secondTail}
\frac{\rho^\frac{n\kappa}{p-1}}{(\nu\hat\theta)^\frac{1}{m}}\widetilde{\mathrm{Tail}}_m((u-\mu_-)_-;Q_0)\leq\xi\omega\end{split}\end{equation}
and $B_{\rho}(x_0)\times(t_1,t_1+\nu\hat\theta\rho^{sp})\subset Q_0$ hold for any fixed $\nu<\nu_1$,
then
\begin{equation}
\label{DeGiorgi2}u(x,t)>\mu_-+\tfrac{1}{2}\xi\omega\qquad\text{for}\ \ \text{a.e.}\ \ (x,t)\in B_{\frac{1}{2}\rho}(x_0)\times(t_1,t_1+\nu\hat\theta\rho^{sp}).\end{equation}
\end{lemma}
\begin{proof}For $j=0,1,2,\cdots$, we define the sequences $k_j=\mu_-+2^{-1}\xi\omega+2^{-1-j}\xi\omega$, $\tilde k_j=\frac{1}{2}(k_j+k_{j+1})$,
$\rho_j=2^{-1}\rho+2^{-j-1}\rho$, $\tilde\rho_j=\frac{1}{2}(\rho_j+\rho_{j+1})$ and $\hat\rho_j=\frac{1}{2}(\tilde\rho_j+\rho_{j+1})$.
We take a cutoff function $\varphi_j=\phi_j(x)\in C_0^\infty(B_{\hat \rho_j}(x_0))$,
which is independent of the time variable. Moreover, we assume that $\phi_j=1$ in $B_{\rho_{j+1}}(x_0)$ and $|D\varphi_j|\leq 2^{j+3}\rho^{-1}$.

With the choice of $\varphi=\phi_j$ in Lemma \ref{caclemma}, we first use the Caccioppoli inequality to the truncated functions $(u-k_j)_-$ over the cylinder $B_{\tilde\rho_j}(x_0)\times(t_1,t_1+\nu\hat\theta\rho^{sp})$ and then make a change of variable $v(x,t)=u(x_0+x,\hat\theta t+t_1)$.
This leads to
\begin{equation*}\begin{split}
T_{0,1}+T_{0,2}&:=\hat\theta^{-1}\esssup_{0<t<\nu\rho^{sp}}\int_{B_{\rho_{j+1}}\times\{t\}}g_-(v,k_j) \,\mathrm {d}x
\\
&+\int_{0}^{\nu\rho^{sp}}\iint_{B_{\rho_{j+1}}\times B_{\rho_{j+1}}}
\frac{|(v-k_j)_-(x,t)-(v-k_j)_-(y,t)|^p}{|x-y|^{n+sp}}\,\mathrm {d}x\mathrm {d}y\mathrm {d}t
\\
\leq&
\gamma 2^{(n+sp)j}\int_{0}^{\nu\rho^{sp}}\int_{\mathbb{R}^n\setminus B_{\tilde \rho_j}}\int_{B_{\tilde \rho_j}}
\frac{(v-k_j)_-(y,t)^{p-1}}{|y|^{n+sp}}(v-k_j)_-(x,t)\,\mathrm {d}x\mathrm {d}y\mathrm {d}t
\\&+ \gamma\frac{2^{jp}}{\rho^{sp}}\int_{0}^{\nu\rho^{sp}}\int_{B_{\tilde \rho_j}}(v-k_j)_-^p
\,\mathrm {d}x\mathrm {d}t=:T_1+T_2.
\end{split}\end{equation*}
In view of \eqref{2st assumption}, we find that the first term on the right-hand
side of \eqref{Cacinequality} vanishes.
Next, for  $j=0,1,2,\cdots$, we set $U_j=B_{\rho_j}\times(0,\nu \rho^{sp})$,
\begin{equation*}\begin{split}
Y_j=\frac{1}{|U_j|}|\{(x,t)\in U_j:v(x,t)\leq k_j\}|
\end{split}\end{equation*}
and
\begin{equation*}\begin{split}
Z_j=\left(\fint_{0}^{\nu\rho^{sp}}\left(\frac{1}{|B_{\rho_j}|}
|A_-(v;0,\rho_j,k_j)|\right)^\frac{m}{m-(p-1)}\,\mathrm{d}t\right)^\frac{m-(p-1)}{m(1+\kappa)}.
\end{split}\end{equation*}
Since $(v-k_j)_-\leq \xi\omega$ in $U_j$, we get $T_2\leq \gamma 2^{jp}\rho^{-sp}(\xi\omega)^p|U_j|Y_j$.
The estimate of $T_1$ is quite similar to the estimate of $I_3$ in the proof of \cite[Lemma 4.1]{BK}.
Taking \eqref{secondTail} into account, we deduce
\begin{equation*}\begin{split}
T_1\leq \gamma 2^{(n+sp)j}(\xi\omega)^p\rho^{-sp}|U_j|(Y_j+Z_j^{1+\kappa}).
\end{split}\end{equation*}
On the other hand, we apply H\"older's inequality, Lemma \ref{embedding} and Lemma \ref{b<a} to conclude that
\begin{equation}\begin{split}\label{yj+1Uj+1}
Y_{j+1}|U_{j+1}|&\leq 2^{jp}(\xi\omega)^{-p}|U_{j+1}|^\frac{sp}{n+sp}Y_j^\frac{sp}{n+sp}
\\&\times\left[
T_{0,2}
+\rho^{-sp}(\xi\omega)^pY_j|U_{j+1}|+\esssup_{0
<t<\nu\rho^{sp}}\int_{B_{\rho_{j+1}}}(v-\tilde k_j)_-^p(x,t)\,\mathrm{d}x
\right].
\end{split}
\end{equation}
It remains to treat the third term on the right-hand side of \eqref{yj+1Uj+1}. To this end, we claim that
\begin{equation}\begin{split}\label{claim1}
g_-(v,k_j)&\geq \gamma 2^{-qj}\hat\theta(v-k_j)_-^p\qquad\text{on}\quad U_{j+1}\cap\{v\leq\tilde k_j\}.
\end{split}
\end{equation}
First, we note that on the set
$U_{j+1}\cap\{v\leq\tilde k_j\}$, there holds
\begin{equation}\begin{split}\label{|v|+|kj|}
|v|+|k_j|\geq k_j-v\geq k_j-\tilde k_j=2^{-j-1}\xi\omega.
\end{split}
\end{equation}
In the case $|\mu_-|\leq \xi\omega$,
we have $\hat\theta=(\xi\omega)^{q+1-p}$.
Noting that $\mu_-\leq v\leq \mu_-+\xi\omega$ on
$U_{j+1}\cap\{v\leq\tilde k_j\}$, we get $|v|+|k_j|\leq 2|\mu_-|+\xi\omega\leq 3\xi\omega$. Since $(v-k_j)_-\leq\xi\omega$ in $U_{j+1}$,
we infer from Lemma \ref{equivalent} that
\begin{equation*}\begin{split}
g_-(v,k_j)&\geq \gamma(|v|+|k_j|)^{q-1}(v-k_j)_-^2\geq \gamma 2^{-qj}(\xi\omega)^{q-1}(v-k_j)_-^{2-p}(v-k_j)_-^p
\\&\geq \gamma 2^{-qj}(\xi\omega)^{1+q-p}(v-k_j)_-^p=\gamma 2^{-qj}\hat\theta(v-k_j)_-^p,
\end{split}
\end{equation*}
which proves the claim \eqref{claim1}. In the case $q\geq1$, we have $\hat\theta=(\xi\omega)^{q+1-p}$. According to \eqref{|v|+|kj|},
the claim \eqref{claim1} follows from Lemma \ref{equivalent} and the fact that $(v-k_j)_-\leq\xi\omega$ in $U_{j+1}$.

Finally, we consider the case $-2\omega\leq \mu_-\leq -\xi\omega$ and $q<1$. In this case, we have $\hat \theta=\omega^{q-1}(\xi\omega)^{2-p}$.
In view of $-2\omega\leq \mu_-\leq -\xi\omega$, we get $|\mu_-|\leq 2\omega$; hence $|v|\leq 2|\mu_-|+\xi\omega<5\omega$. It follows that
$|v|+|k_j|\leq 5\omega+|\mu_-|+\xi\omega<8\omega$. At this point, we deduce from Lemma \ref{equivalent} that
\begin{equation*}\begin{split}
g_-(v,k_j)&\geq \gamma(|v|+|k_j|)^{q-1}(v-k_j)_-^2\geq \gamma\omega^{q-1}(v-k_j)_-^{2-p}(v-k_j)_-^p
\\&\geq \gamma \omega^{q-1}(\xi\omega)^{2-p}(v-k_j)_-^p=\gamma \hat\theta(v-k_j)_-^p,
\end{split}
\end{equation*}
since $q<1$ and $(v-k_j)_-\leq\xi\omega$ in $U_{j+1}$. This completes the proof of the
claim \eqref{claim1}.

Combining \eqref{yj+1Uj+1}-\eqref{|v|+|kj|}, we conclude that
\begin{equation}\begin{split}\label{Yj+1xi}
Y_{j+1}\leq \gamma 2^{j(2p+n+q)}\nu^\frac{sp}{n+sp}\left(Y_j^{1+\frac{sp}{n+sp}}+Y_j^\frac{sp}{n+sp}Z_j^{1+\kappa}\right).
\end{split}
\end{equation}
Analysis similar to that in the proofs of \cite[Lemma 4.1, 4.2]{BK} shows that
\begin{equation}\begin{split}\label{Zj+1xi}
Z_{j+1}&\leq \gamma 2^{(n+2p+q)j}\nu^{1-\frac{m-(p-1)}{m(1+\kappa)}}\left(Y_j+Z_j^{1+\kappa}\right).
\end{split}
\end{equation}
At this stage, we choose
\begin{equation}\begin{split}\label{nu1}
\nu_1=(2\gamma)^{-\frac{1}{l}}2^{-\frac{\zeta}{(1+\kappa)l}-\frac{n+2p+q}{\zeta l}},
\end{split}
\end{equation}
where $\zeta=\min\left\{\kappa,\frac{sp}{n+sp}\right\}$ and $l=\min\left\{\frac{sp}{n+sp},1-\frac{m-(p-1)}{m(1+\kappa)}\right\}$.
Since $Y_0\leq1$ and $Z_0\leq1$, we find that the inequality
\begin{equation*}Y_0+Z_0^{1+\kappa}\leq 2\leq (2\gamma \nu^l)^{-\frac{1+\kappa}{\zeta}}2^{-\frac{(n+2p+q)(1+\kappa)}{\zeta^2}}\end{equation*}
holds for any $\nu\in(0,\nu_1]$.
Using a lemma on fast geometric convergence of sequences (see for instance \cite[Chapter I, Lemma 4.2]{Di93}),
we conclude from \eqref{Yj+1xi} and \eqref{Zj+1xi} that the sequence $\{Y_n\}$ tends to zero, as $n\to\infty$. Transforming
back to the original function $u$, we obtain the desired estimate \eqref{DeGiorgi2}.
This
completes the proof of Lemma \ref{lemmaDeGiorgi2}.
\end{proof}
Before proceeding further, we denote
\begin{equation}\begin{split}\label{xi0}
\xi_0=\min\left\{\frac{1}{8},\frac{1}{4}\left(\frac{\nu_1}{2^{sp}A}\right)^\frac{1}{p-q-1},4^\frac{q+1-p}{p-2}\left(\frac{\nu_1}{2^{sp}A}\right)
^\frac{1}{p-2}\right\},
\end{split}
\end{equation}
where $\nu_1$ is the constant in \eqref{nu1}. The next proposition is our main result in this subsection.
\begin{proposition}\label{1st proposition}
Let $u$ be a bounded weak solution to \eqref{LKut}-\eqref{kernel}
 in the sense of Definition \ref{weak solution}.
Assume that there exists a time level $t_0-(A-1)\theta\rho^{sp}\leq \bar t\leq t_0$ such that \eqref{1st} holds.
Moreover, suppose that
 \begin{equation}\begin{split}\label{Alt1Tail}
\frac{\rho^\frac{n\kappa}{p-1}}{\theta^\frac{1}{m}}\widetilde{\mathrm{Tail}}_m((u-\mu_-)_-;Q_0)\leq\xi_0\omega.\end{split}\end{equation}
Then, we have
\begin{equation}\begin{split}\label{osc1}
\essosc_{\tilde Q_0} u\leq (1-\tfrac{1}{2}\xi_0)\omega,
 \end{split}\end{equation}
 where $\tilde Q_0=Q_{\frac{1}{4}\rho}^{(\theta)}(z_0)$.
\end{proposition}
\begin{proof} First, we infer from \eqref{Alt1Tail} that \eqref{firstTail} holds. This enables us to use Lemma \ref{lemmaDeGiorgi1}, since $|\mu_-|\leq 2\omega$.
According to \eqref{DeGiorgi1}, we infer that
\begin{equation}
\label{initial1}u(x,\tilde t)\geq\mu_-+\xi_0\omega\qquad\text{for}\ \ \text{a.e.}\ \ x\in B_{\frac{1}{2}\rho}(x_0),\end{equation}
where $\tilde t=\bar t-\theta (\frac{1}{2}\rho)^{sp}$. Next, we will apply Lemma \ref{lemmaDeGiorgi2} with $\xi=\xi_0$.
To this end, we distinguish two cases.

In the case $|\mu_-|\leq\xi_0\omega$ or $q\geq1$, we infer from \eqref{hattheta} that $\hat\theta=(\xi_0\omega)^{q+1-p}$.
In view of \eqref{xi0}, we find that
\begin{equation}\label{tildet1}\tilde t+\nu_1(\xi_0\omega)^{q+1-p}\left(\tfrac{1}{2}\rho\right)^{sp}\geq t_0,\end{equation}
since $\tilde t\geq t_0-A\theta\rho^{sp}=t_0-A(\frac{1}{4}\omega)^{q+1-p}\rho^{sp}$ and $\xi_0\leq \frac{1}{4}\left(\frac{1}{2^{sp}A}\nu_1\right)^\frac{1}{p-q-1}$.
On the other hand, we set $\nu_1^\prime=4^{p-q-1}\xi_0^{p-1-q}$. We observe that $\nu_1^\prime\leq 2^{-sp}A^{-1}\nu_1<\nu_1$ and there holds
\begin{equation}\label{tildet11}\tilde t+\nu_1^\prime(\xi_0\omega)^{q+1-p}\left(\tfrac{1}{2}\rho\right)^{sp}\leq t_0,\end{equation}
since $\tilde t=\bar t-\theta (\frac{1}{2}\rho)^{sp}\leq t_0-(\frac{1}{4}\omega)^{q+1-p}(\frac{1}{2}\rho)^{sp}$. From \eqref{tildet1}
and \eqref{tildet11}, we infer that there exists a constant $\nu\in [\nu_1^\prime,\nu_1]$, such that
\begin{equation*}
\tilde t+\nu\hat\theta\left(\tfrac{1}{2}\rho\right)^{sp}=
\tilde t+\nu(\xi_0\omega)^{q+1-p}\left(\tfrac{1}{2}\rho\right)^{sp}= t_0.\end{equation*}
Moreover, we infer from \eqref{Alt1Tail} that
\begin{equation*}\begin{split}
\frac{(\frac{1}{2}\rho)^\frac{n\kappa}{p-1}}{\omega^\frac{q+1-p}{m}}&\widetilde{\mathrm{Tail}}_m((u-\mu_-)_-;Q_0)<4^\frac{p-1-q}{m}\xi_0\omega
\\&= (\nu_1^\prime)^\frac{1}{m}\xi_0^{1-\frac{p-1-q}{m}}
\omega
\leq\nu^\frac{1}{m}\xi_0^{1-\frac{p-1-q}{m}}
\omega,
\end{split}\end{equation*}
which in turn yields that \eqref{secondTail} holds with $\xi=\xi_0$ and $\hat\theta=(\xi_0\omega)^{q+1-p}$.
We now apply Lemma \ref{lemmaDeGiorgi2} with $\rho$ replaced by $\rho/2$ to conclude that
\begin{equation}
\label{DeGiorgipro}
u(x,t)>\mu_-+\tfrac{1}{2}\xi_0\omega\qquad\text{for}\ \ \text{a.e.}\ \ (x,t)\in B_{\frac{1}{4}\rho}(x_0)\times(\tilde t,t_0),\end{equation}
which proves the desired estimate \eqref{osc1}.

In the case $|\mu_-|>\xi_0\omega$ and $q<1$, we infer from $\mu_-\leq\xi_0\omega$ and $|\mu_-|\leq 2\omega$ that
$-2\omega\leq\mu_-\leq-\xi_0\omega$. It follows from \eqref{hattheta} that
$\hat\theta=\omega^{q-1}(\xi_0\omega)^{2-p}$. Since $\tilde t\geq t_0-A\theta\rho^{sp}$ and $\xi_0\leq 4^\frac{q+1-p}{p-2}\left(\frac{1}{2^{sp}A}\nu_1\right)
^\frac{1}{p-2}$, we have
\begin{equation}\label{tildet2}\tilde t+\nu_1\omega^{q-1}(\xi_0\omega)^{2-p}\left(\tfrac{1}{2}\rho\right)^{sp}\geq t_0.\end{equation}
Furthermore, we set $\nu_1^{\prime\prime}=4^{p-q-1}\xi_0^{p-2}$. It can be easily seen that $\nu_1^{\prime\prime}<\nu_1$ and there holds
\begin{equation}\label{tildet22}\tilde t+\nu_1^{\prime\prime}\omega^{q-1}
(\xi_0\omega)^{2-p}\left(\tfrac{1}{2}\rho\right)^{sp}\leq t_0,\end{equation}
since $\tilde t\leq t_0-(\frac{1}{4}\omega)^{q+1-p}(\frac{1}{2}\rho)^{sp}$. In view of \eqref{tildet2}
and \eqref{tildet22}, we find that there exists a constant $\nu\in [\nu_1^{\prime\prime},\nu_1]$, such that
\begin{equation*}
\tilde t+\nu\hat\theta\left(\tfrac{1}{2}\rho\right)^{sp}=
\tilde t+\nu\omega^{q-1}(\xi_0\omega)^{2-p}\left(\tfrac{1}{2}\rho\right)^{sp}= t_0.\end{equation*}
Next, we conclude from \eqref{Alt1Tail} and $\nu_1^{\prime\prime}=4^{p-q-1}\xi_0^{p-2}$ that
\begin{equation*}\begin{split}
\frac{(\frac{1}{2}\rho)^\frac{n\kappa}{p-1}}{\omega^\frac{q+1-p}{m}}&\widetilde{\mathrm{Tail}}_m((u-\mu_-)_-;Q_0)<4^\frac{p-1-q}{m}\xi_0\omega
\\&= (\nu_1^{\prime\prime})^\frac{1}{m}\xi_0^{1-\frac{p-2}{m}}
\omega
\leq\nu^\frac{1}{m}\xi_0^{1-\frac{p-2}{m}}
\omega,
\end{split}\end{equation*}
which implies that \eqref{secondTail} holds with $\xi=\xi_0$ and $\hat\theta=\omega^{q-1}(\xi_0\omega)^{2-p}$.
At this point, we use Lemma \ref{lemmaDeGiorgi2} with $\rho$ replaced by $\rho/2$ to find that \eqref{DeGiorgipro} holds.
This completes the proof
of the proposition.
\end{proof}
\subsection{The second alternative}
In this subsection, we assume that \eqref{2nd} holds for any $t_0-(A-1)\theta\rho^{sp}\leq \bar t\leq t_0$, where $\theta=(\frac{1}{4}\omega)^{q+1-p}$.
It can be easily seen that there exists a time level $t_*\in[\bar t-\theta\rho^{sp},\bar t-\frac{1}{2}\nu_0\theta\rho^{sp}]$,
such that
\begin{equation}\label{t*}\large|\large\{u(\cdot,t_*)\leq\mu_+-\tfrac{1}{4}\omega\large
\}\cap B_\rho(x_0)\large|>\tfrac{1}{2}\nu_0|B_\rho|.\end{equation}
On the other hand, recalling that $|\mu_+|\leq 2\omega$ and $\mu_+\geq \frac{1}{4}\omega$, we get $\frac{1}{4}\omega\leq \mu_+\leq 2\omega$.
In order to proceed further, we first establish a result regarding the propagation of measure theoretical information.
\begin{lemma}\label{measure}
Let $u$ be a bounded weak solution to \eqref{LKut}-\eqref{kernel}
 in the sense of Definition \ref{weak solution}.
Assume that \eqref{t*} holds.
Then, there exists a constant $\xi_1\in (0,\frac{1}{8})$ depending only upon the data, such that if
 \begin{equation}\begin{split}\label{measureTail}
\frac{\rho^\frac{n\kappa}{p-1}}{\theta^\frac{1}{m}}\widetilde{\mathrm{Tail}}_m((u-\mu_+)_+;Q_0)\leq\xi_1\omega,\end{split}\end{equation}
then
\begin{equation}\label{t*bart}\large|\large\{u(\cdot,t)\leq\mu_+-\xi_1\omega\large
\}\cap B_\rho(x_0)\large|\geq\tfrac{1}{4}\nu_0|B_\rho|\end{equation}
holds for all $t\in [t_*,\bar t]$.
\end{lemma}
\begin{proof} Let $\xi\in(0,\frac{1}{8})$ be a fixed constant which will be determined in the course of the proof.
For any fixed $t\in (t_*,\bar t)$, we define $Q=B_\rho(x_0)\times (t_*,t)$. Moreover, we set $k=\mu_+-\xi\omega$ and observe that $k\geq \frac{1}{8}\omega>0$. Next, we take a testing function $\varphi(x,t)=\phi(x)$, defined in $B_\rho$, satisfying
$\phi\equiv 1$ in $B_{(1-\sigma)\rho}(x_0)$ and $|D\phi|\leq (\sigma \rho)^{-1}$, where $\sigma\in (0,1)$ is to be determined.
With these choices, we apply the Caccioppoli inequality \eqref{Cacinequality} to deduce
\begin{equation*}\begin{split}
&\int_{B_\rho(x_0)\times\{t\}}g_{+}(u,k)\phi^p \,\mathrm {d}x
\\
&\leq \int_{B_\rho(x_0)\times\{t_*\}}g_{+}(u,k)\phi^p \,\mathrm {d}x
+\gamma\sigma^{-p}\rho^{-sp}\iint_{Q}(u-k)_+^p\,\mathrm {d}x\mathrm {d}t
\\&
+\gamma \sigma^{-n-sp}\int_{t_*}^{\bar t}\int_{\mathbb{R}^n\setminus B_\rho(x_0)}\int_{B_\rho(x_0)}
\frac{(u-k)_+(y,t)^{p-1}}{|y-x_0|^{n+sp}}
(u-k)_+(x,t)\phi(x)^p\,\mathrm {d}x\mathrm {d}y\mathrm {d}t
\\&=:T_1+T_2+T_3
\end{split}\end{equation*}
holds for all $t\in (t_*,\bar t)$. To estimate below the integral
on the left-hand side, we set $\tilde k=\mu_+-\tilde\xi\xi\omega$ and $\xi_1=\tilde \xi\xi$, where $\tilde\xi\in(0,\frac{1}{2})$ is to
be determined later. From \eqref{g+-}, we get
\begin{equation*}\begin{split}
\int_{B_\rho(x_0)\times\{t\}}&g_{+}(u,k)\phi^p \,\mathrm {d}x
=q\int_{B_\rho(x_0)\times\{t\}}\int_k^u\tau^{q-1}(\tau-k)_+\,\mathrm{d}\tau\phi^p \,\mathrm {d}x
\\&\geq q\large|\large\{u(\cdot,t)>\tilde k\large
\}\cap B_{(1-\sigma)\rho}(x_0)\large|\int_k^{\tilde k}\tau^{q-1}(\tau-k)_+\,\mathrm{d}\tau.
\end{split}\end{equation*}
Noting that $\frac{1}{4}\omega\leq\mu_+\leq2\omega$, we obtain
\begin{equation*}\begin{split}
\int_k^{\tilde k}\tau^{q-1}(\tau-k)_+\,\mathrm{d}\tau\geq \gamma\omega^{q-1}(\xi\omega)^2= \gamma \xi^2\omega^{q+1}.
\end{split}\end{equation*}
Recalling that $u\leq\mu_+$ on $B_{\rho}(x_0)\times\{t_*\}$, we infer from \eqref{g+-} and \eqref{t*} that
\begin{equation*}\begin{split}
T_1&\leq \large|\large\{u(\cdot,t_*)\geq k\large
\}\cap B_{\rho}(x_0)\large|\int_k^{\mu_+}\tau^{q-1}(\tau-k)_+\,\mathrm{d}\tau
\\&<\left(1-\tfrac{1}{2}\nu_0
\right)|B_\rho|\int_k^{\mu_+}\tau^{q-1}(\tau-k)_+\,\mathrm{d}\tau.
\end{split}\end{equation*}
Next, we consider the estimate for $T_2$. Noting that $(u-k)_+\leq\xi\omega$ in $Q$ and $\bar t-t_*\leq \theta \rho^{sp}$, we have
$T_2\leq \sigma^{-p}\xi^p\omega^{1+q}|B_\rho|$,
since $\theta=(\frac{1}{4}\omega)^{q+1-p}$. To estimate $T_3$, we observe that $(u-k)_+\leq (u-\mu_+)_++\xi\omega$. It follows that
\begin{equation*}\begin{split}
T_3&\leq \gamma \sigma^{-n-sp}\int_{t_*}^{\bar t}\int_{\mathbb{R}^n\setminus B_\rho(x_0)}\int_{B_\rho(x_0)}
\frac{(\xi\omega)^{p-1}}{|y-x_0|^{n+sp}}
(u-k)_+(x,t)\phi(x)^p\,\mathrm {d}x\mathrm {d}y\mathrm {d}t
\\&+\gamma \sigma^{-n-sp}\int_{t_*}^{\bar t}\int_{\mathbb{R}^n\setminus B_{\rho_0}(x_0)}\int_{B_\rho(x_0)}
\frac{(u-\mu_+)_+(y,t)^{p-1}}{|y-x_0|^{n+sp}}
(u-k)_+(x,t)\phi(x)^p\,\mathrm {d}x\mathrm {d}y\mathrm {d}t
\\&=:T_{3,1}+T_{3,2}.
\end{split}\end{equation*}
We first observe that
$T_{3,1}\leq \gamma\sigma^{-n-sp}\xi^p\omega^{q+1}|B_\rho|$. To estimate $T_{3,2}$, we apply H\"older's inequality and \eqref{measureTail} to obtain
\begin{equation*}\begin{split}
T_{3,2}&\leq \gamma \sigma^{-n-sp}\xi\omega|B_\rho|(\bar t-t_*)^\frac{m-(p-1)}{m}\left[\int_{t_*}^{\bar t}\left(\int_{\mathbb{R}^n\setminus B_{\rho_0}(x_0)}
\frac{(u-\mu_+)_+(y,t)^{p-1}}{|y-x_0|^{n+sp}}
\,\mathrm {d}y\right)^\frac{m}{p-1}\,\mathrm {d}t\right]^\frac{p-1}{m}
\\&\leq \gamma \sigma^{-n-sp}\xi\omega (\theta\rho^{sp})^\frac{m-(p-1)}{m}|B_\rho|\widetilde{\mathrm{Tail}}_m((u-\mu_+)_+;Q_R)^{p-1}
\\&\leq \gamma\sigma^{-n-sp}\xi^p\omega^{q+1}|B_\rho|,
\end{split}\end{equation*}
since $\xi_1<\xi$.
Combining the above estimates, we arrive at
\begin{equation*}\begin{split}
&\large|\large\{u(\cdot,t)>\tilde k\large
\}\cap B_{(1-\sigma)\rho}(x_0)\large|
\\&\leq \left(1-\tfrac{1}{2}\nu_0\right)|B_\rho|\frac{\int_k^{\mu_+}\tau^{q-1}(\tau-k)_+\,\mathrm{d}\tau}
{\int_k^{\tilde k}\tau^{q-1}(\tau-k)_+\,\mathrm{d}\tau}+\gamma\sigma^{-n-sp}\xi^{p-2}|B_\rho|.
\end{split}\end{equation*}
Recalling that $\frac{1}{4}\omega\leq\mu_+\leq2\omega$ and $k\geq\frac{1}{8}\omega$, we find that $\frac{1}{8}\omega\leq \tilde k\leq 3\omega$.
According to the proof of \cite[Lemma 5.1]{BDLS}, we infer that there exists a constant $C_0=C_0(\text{data})>0$, such that\begin{equation*}\begin{split}
\frac{\int_k^{\mu_+}\tau^{q-1}(\tau-k)_+\,\mathrm{d}\tau}
{\int_k^{\tilde k}\tau^{q-1}(\tau-k)_+\,\mathrm{d}\tau}\leq 1+C_0\tilde\xi.
\end{split}\end{equation*}
Consequently, we infer that
\begin{equation*}\begin{split}
&\large|\large\{u(\cdot,t)>\tilde k\large
\}\cap B_{\rho}(x_0)\large|
\\&\leq \left(1-\tfrac{1}{2}\nu_0\right)(1+C_0\tilde\xi)|B_\rho|+\gamma\sigma^{-n-sp}\xi^{p-2}|B_\rho|+n\sigma|B_\rho|.
\end{split}\end{equation*}
At this point, we choose $\tilde\xi=\tilde\xi(\text{data},\nu_0)<\frac{1}{2}$ such that $\left(1-\frac{1}{2}\nu_0\right)(1+C_0\tilde\xi)\leq 1-\frac{3}{8}\nu_0$. Next, we set $\sigma=(16n)^{-1}\nu_0$ and then choose $\xi<\frac{1}{8}$ such that $\gamma\sigma^{-n-sp}\xi^{p-2}<\frac{1}{16}\nu_0$.
With the choice of $\xi_1=\tilde\xi\xi$, we
conclude that the lemma holds.
\end{proof}
\begin{remark}\label{largetime}
Noting that $\bar t$ is an arbitrary time level in $[t_0-(A-1)\theta\rho^{sp},t_0]$, then we conclude that if \eqref{measureTail} holds then
the estimate \eqref{t*bart} holds for all $t\in[t_0-(A-1)\theta\rho^{sp},t_0]$.
\end{remark}
Let $\sigma_*<1$ be a constant such that $A=1+\sigma_*^{-(p-2)}\xi_1^{2-p}$, where $\xi_1$ is the constant from Lemma \ref{measure}.
Next, we set $\tilde\theta=\theta\xi_1^{2-p}
=\left(\frac{1}{4}\omega\right)^{q+1-p}\xi_1^{2-p}$.
For simplicity of notation, we abbreviate $Q_A=Q_\rho^{(A-1)\tilde\theta}(z_0)=Q_\rho^{\sigma_*^{-(p-2)}\tilde\theta}(z_0)$.
Our next goal is to establish the following measure shrinking lemma.
\begin{lemma}\label{measureshrinkinglemma} Let $u$ be a bounded weak solution to \eqref{LKut}-\eqref{kernel}
 in the sense of Definition \ref{weak solution}.
Assume that \eqref{t*} holds.
Moreover, suppose that
 \begin{equation}\begin{split}\label{measureTail1}
\frac{\rho^\frac{n\kappa}{p-1}}{(\sigma_*^{2-p}
\tilde\theta)^\frac{1}{m}}\widetilde{\mathrm{Tail}}_m((u-\mu_+)_+;Q_0)\leq \sigma_*\xi_1\omega.\end{split}\end{equation}
There exists a constant $\gamma$ that can be determined a priori only in terms of the data such that
\begin{equation}\label{measureQ}\large|\large\{u\geq\mu_+-\tfrac{1}{4}\xi_1\sigma_*\omega\large
\}\cap Q_A\large|\leq\gamma \sigma_*^{p-1}\nu_0^{-1}|Q_A|.\end{equation}
\end{lemma}
\begin{proof} First, we set $k=\mu_+-\frac{1}{2}\sigma_*\xi_1\omega$ and $\tilde Q_A=B_{2\rho}(x_0)\times (t_0-\sigma_*^{-(p-2)}\tilde\theta\rho^{sp},t_0)$. Next, we take a smooth cutoff function $0\leq\phi(x)\leq1$, supported in $B_{\frac{3}{2}\rho}(x_0)$, satisfying
$\phi\equiv 1$ in $B_{\rho}(x_0)$ and $|D\phi|\leq 2\rho^{-1}$.
Write the energy estimate \eqref{Cacinequality} over
the cylinder $\tilde Q_A$ for the truncated function $(u-k)_+$. This leads to
\begin{equation*}\begin{split}
&T_0:=\iint_{Q_A}(u-k)_+(x,t)\left(\int_{B_\rho(x_0)}\frac{(u-k)_-(y,t)^{p-1}
}{|x-y|^{n+sp}}\,\mathrm{d}y\right)\,\mathrm {d}x\mathrm {d}t
\\
&\leq \int_{B_{2\rho}(x_0)\times\{t_0-\sigma_*^{-(p-2)}\tilde\theta\rho^{sp}\}}g_+(u,k)\phi^p \,\mathrm {d}x
+\gamma\rho^{-sp}\iint_{\tilde Q_A}(u-k)_+^p\,\mathrm {d}x\mathrm {d}t
\\&
+\gamma \int_{t_0-\sigma_*^{-(p-2)}\tilde\theta\rho^{sp}}^{t_0}\int_{\mathbb{R}^n\setminus B_{2\rho}(x_0)}\int_{B_{2\rho}(x_0)}
\frac{(u-k)_+(y,t)^{p-1}}{|y-x_0|^{n+sp}}
(u-k)_+(x,t)\phi(x)^p\,\mathrm {d}x\mathrm {d}y\mathrm {d}t
\\&=:T_1+T_2+T_3.
\end{split}\end{equation*}
To estimate below the integral
on the left-hand side, we infer from \eqref{measureTail1} and $\sigma_*<1$ that \eqref{measureTail} holds.
According to
Remark \ref{largetime}, we find that
\begin{equation*}\begin{split}
&\int_{B_\rho(x_0)}\frac{(u-k)_-(y,t)^{p-1}
}{|x-y|^{n+sp}}\,\mathrm{d}y\geq \int_{B_\rho(x_0)\cap \{u\leq \mu_+-\xi_1\omega\}}\frac{(u-k)_-(y,t)^{p-1}
}{|x-y|^{n+sp}}\,\mathrm{d}y
\\&\geq \large|\large\{u(\cdot,t)\leq\mu_+-\xi_1\omega\large
\}\cap B_\rho(x_0)\large|\frac{[(1-\tfrac{1}{2}\sigma_*)\xi_1\omega]^{p-1}}{\rho^{n+sp}}
\geq \gamma \nu_0\frac{(\xi_1\omega)^{p-1}}{\rho^{sp}}
\end{split}\end{equation*}
holds for all $(x,t)\in Q_A$. It follows that
\begin{equation*}\begin{split}
T_0
&\geq \gamma \nu_0\frac{(\xi_1\omega)^{p-1}}{\rho^{sp}}\iint_{Q_A\cap\{u\geq\mu_+-\frac{1}{4}\sigma_*\xi_1\omega
\}}(u-k)_+\,\mathrm {d}x\mathrm {d}t
\\&\geq \gamma\nu_0\sigma_*(\xi_1\omega)^p\rho^{-sp}\large|\large\{u\geq\mu_+-\tfrac{1}{4}\xi_1\sigma_*\omega\large
\}\cap Q_A\large|.
\end{split}\end{equation*}
Next, we consider the estimate for $T_2$. Since $(u-k)_+\leq \frac{1}{2}\sigma_*\xi_1\omega$ in $\tilde Q_A$, we have $T_2\leq \gamma \rho^{-sp} (\sigma_*\xi_1\omega)^p|Q_A|$. To estimate $T_1$, we note that $\frac{1}{4}\omega\leq\mu_+\leq2\omega$ and hence $\frac{1}{16}\omega\leq
k\leq3\omega$, since $k=\mu_+-\frac{1}{2}\sigma_*\xi_1\omega$. We apply Lemma \ref{equivalent} to obtain
\begin{equation*}\begin{split}
T_1&\leq \gamma\int_{B_{2\rho}(x_0)\times\{t_0-\sigma_*^{-(p-2)}\tilde\theta\rho^{sp}\}}(|u|+|k|)^{q-1}(u-k)_+^2 \,\mathrm {d}x
\\&\leq \gamma\omega^{q-1}(\sigma_*\xi_1\omega)^2|B_\rho|=\gamma\omega^{q-1}(\sigma_*\xi_1\omega)^2\frac{|Q_A|}{\sigma_*^{-(p-2)}\tilde\theta\rho^{sp}}
\\&=\gamma\omega^{q-1}(\sigma_*\xi_1\omega)^2\frac{|Q_A|}{\sigma_*^{-(p-2)}\xi_1^{2-p}\omega^{q+1-p}\rho^{sp}}=\gamma\frac{(\sigma_*\xi_1\omega)^p}
{\rho^{sp}}|Q_A|,
\end{split}\end{equation*}
since $\tilde\theta
=\left(\frac{1}{4}\omega\right)^{q+1-p}\xi_1^{2-p}$.
Finally, we turn our attention to the estimate of $T_3$. To this end, we observe that $(u-k)_+\leq (u-\mu_+)_++\frac{1}{2}\sigma_*
\xi_1\omega$ and there holds
\begin{equation*}\begin{split}
T_3&\leq \gamma\frac{(\sigma_*\xi_1\omega)^p}
{\rho^{sp}}|Q_A|
\\&+\gamma \int_{t_0-\sigma_*^{-(p-2)}\tilde\theta\rho^{sp}}^{t_0}\int_{\mathbb{R}^n\setminus B_{\rho_0}(x_0)}\int_{B_\rho(x_0)}
\frac{(u-\mu_+)_+(y,t)^{p-1}}{|y|^{n+sp}}
(u-k)_+(x,t)\,\mathrm {d}x\mathrm {d}y\mathrm {d}t.
\end{split}\end{equation*}
To proceed further, we denote by $T_{3,1}$ the second term of the right-hand side of the above inequality.
We use H\"older's inequality and \eqref{measureTail1} to infer that
\begin{equation*}\begin{split}
T_{3,2}&\leq \gamma \sigma_*\xi_1\omega|B_\rho|\left(\frac{\tilde\theta\rho^{sp}}{\sigma_*^{p-2}}\right)^\frac{m-(p-1)}{m}
\left[\int_{t_0-\sigma_*^{-(p-2)}\tilde\theta\rho^{sp}}^{ t_0}\left(\int_{\mathbb{R}^n\setminus B_{\rho_0}(x_0)}
\frac{(u-\mu_+)_+(y,t)^{p-1}}{|y-x_0|^{n+sp}}
\,\mathrm {d}y\right)^\frac{m}{p-1}\,\mathrm {d}t\right]^\frac{p-1}{m}
\\&\leq \gamma \sigma_*\xi_1\omega|B_\rho|(\sigma_*^{-(p-2)}\tilde\theta\rho^{sp})^\frac{m-(p-1)}{m}\widetilde{\mathrm{Tail}}_m((u-\mu_+)_+;Q_0)^{p-1}
\\&\leq \gamma\sigma_*^{2-p}\tilde\theta(\sigma_*\xi_1\omega)^p|B_\rho|=\gamma \rho^{-sp} (\sigma_*\xi_1\omega)^p|Q_A|,
\end{split}\end{equation*}
since $(t_0-\sigma_*^{-(p-2)}\tilde\theta\rho^{sp},t_0)=(t_0-(A-1)\theta\rho^{sp},t_0)\subset (t_0-A\theta\rho_0^{sp},t_0)$.
This implies that $T_3\leq \gamma \rho^{-sp} (\sigma_*\xi_1\omega)^p|Q_A|$. Combining the estimates for $T_0$-$T_3$, we obtain the desired estimate
\eqref{measureQ}. This completes the proof.
\end{proof}
At this point, we introduce a De Giorgi-type lemma for the second alternative.
This lemma can be proved in a similar way as Lemma \ref{lemmaDeGiorgi1},
and we only focus on the
difficulties that arise for the second alternative.
\begin{lemma}\label{lemmaDeGiorgi3}
Let $u$ be a locally bounded weak solutions to \eqref{LKut}-\eqref{kernel}
 in the sense of Definition \ref{weak solution}. Assume that \eqref{t*} and \eqref{measureTail1} hold.
 There exists a constant $\nu_2\in(0,1)$, depending only on the data, such that if
\begin{equation}\label{2nd Degiorgi assumption}\left|Q_A\cap\left\{u\geq\mu_+-\tfrac{1}{4}\sigma_*\xi_1\omega\right\}\right
|\leq \nu_2|Q_A|,\end{equation}
then
\begin{equation}
\label{DeGiorgi3}u(x,t)\leq\mu_+-\tfrac{1}{8}\sigma_*\xi_1\omega\qquad\text{for}\ \ \text{a.e.}\ \ (x,t)\in Q_A^\prime,\end{equation}
where  $Q_A^\prime=B_{\frac{1}{2}\rho}(x_0)\times (t_0-\sigma_*^{-(p-2)}\tilde\theta(\frac{1}{2}\rho)^{sp},t_0)$.
\end{lemma}
\begin{proof}
Let $\hat \theta=\sigma_*\tilde\theta=(\sigma_*\xi_1)^{2-p}\theta$.
For $j=0,1,2,\cdots$, we define the sequences $k_j=\mu_+-2^{-3}\sigma_*\xi_1\omega-2^{-j-4}\sigma_*\xi_1\omega$, $\tilde k_j=\frac{1}{2}(k_j+k_{j+1})$,
$\rho_j=2^{-1}\rho+2^{-j-1}\rho$, $\tilde\rho_j=\frac{1}{2}(\rho_j+\rho_{j+1})$ and $\hat\rho_j=\frac{1}{2}(\tilde\rho_j+\rho_{j+1})$.
Moreover, we take a cutoff function
$0\leq\varphi_j\leq1$, such that $\varphi_j=1$ in $Q_{\rho_{j+1}}^{(\hat\theta)}(z_0)$, $\varphi_j=0$ on $\partial_PQ_{\hat\rho_j}^{(\hat\theta)}(z_0)$,
\begin{equation*}\begin{split}
|D\varphi_j|\leq 2^{j+3}\rho^{-1}\qquad\text{and}\qquad
|\partial_t\varphi_j|\leq \gamma\hat\theta^{-1}2^j\rho^{-sp}.
\end{split}\end{equation*}
With the choice of $\varphi=\varphi_j$ in Lemma \ref{caclemma}, we first apply
 the Caccioppoli inequality to the truncated functions $(u-k_j)_+$ over the cylinder $Q_{\tilde\rho_j}^{(\hat\theta)}(z_0)$
and subsequently make a
change of variable $v(x,t)=u(x_0+x,t_0+\hat\theta t)$.
 This leads to
\begin{equation*}\begin{split}
T_{0,1}&+T_{0,2}:=\hat\theta^{-1}\esssup_{-\tilde\rho_j^{sp}<t<0}\int_{B_{\rho_{j+1}}\times\{t\}}g_+(v,k_j) \,\mathrm {d}x
\\
&+\int_{-\tilde\rho_j^{sp}}^{0}\iint_{B_{\rho_{j+1}}\times B_{\rho_{j+1}}}
\frac{|(v-k_j)_+(x,t)-(v-k_j)_+(y,t)|^p}{|x-y|^{n+sp}}\,\mathrm {d}x\mathrm {d}y\mathrm {d}t
\\
&\leq \gamma 2^j\hat\theta^{-1}\rho^{-sp}
\iint_{Q_{\tilde\rho_j}} g_+(v,k_j)\,\mathrm {d}x\mathrm {d}t
+\gamma\frac{2^{jp}}{\rho^{sp}}\iint_{Q_{\tilde\rho_j}}(v-k_j)_+^p
\,\mathrm {d}x\mathrm {d}t
\\&
+\gamma 2^{(n+sp)j}\int_{-\tilde\rho_j^{sp}}^0\int_{\mathbb{R}^n\setminus B_{\tilde \rho_j}}\int_{B_{\tilde \rho_j}}
\frac{(v-k_j)_+(y,t)^{p-1}}{|y|^{n+sp}}(v-k_j)_+(x,t)\,\mathrm {d}x\mathrm {d}y\mathrm {d}t
\\&=:T_1+T_2+T_3.
\end{split}\end{equation*}
At this point, for  $j=0,1,2,\cdots$, we set
\begin{equation*}\begin{split}
Y_j=\frac{1}{|Q_{\rho_j}|}|\{(x,t)\in Q_{\rho_j}:v(x,t)\geq k_j\}|
\end{split}\end{equation*}
and
\begin{equation*}\begin{split}
Z_j=\left(\fint_{-\rho_j^{sp}}^0\left(\frac{1}{|B_j|}|A_+(v;0,\rho_j,k_j)|\right)^\frac{m}{m-(p-1)}\,\mathrm{d}t\right)^\frac{m-(p-1)}{m(1+\kappa)}.
\end{split}\end{equation*}
Since $(v-k_j)_+\leq \frac{1}{4}\sigma_*\xi_1\omega$ in $Q_\rho$, we have $T_2\leq\gamma 2^{jp}\rho^{-sp}(\sigma_*\xi_1\omega)^p|Q_{\rho_j}|Y_j$.
To estimate $T_1$, we first note that
$\mu_+-\sigma_*\xi_1\omega\leq v\leq \mu_+$
on the set $Q_{\rho_j}\cap \{v\geq k_j\}$. Noting that $\frac{1}{4}\omega\leq\mu_+\leq2\omega$, we get
\begin{equation}\label{vandkj}\tfrac{1}{32}\omega\leq|k_j|\leq |v|+|k_j|\leq 2|\mu_+|+2\sigma_*\xi_1\omega\leq 6\omega\end{equation}
on the set $Q_{\rho_j}\cap \{v\geq k_j\}$.
We now apply
Lemma \ref{equivalent} to obtain
\begin{equation*}\begin{split}
g_+(v,k_j)&\leq\gamma(|v|+|k_j|)^{q-1}(v-k_j)_+^2
\leq \gamma \omega^{q-1}(\sigma_*\xi_1\omega)^2.
\end{split}
\end{equation*}
Recalling that $\hat \theta=(\sigma_*\xi_1)^{2-p}\theta$ and $\theta=\left(\frac{1}{4}\omega\right)^{q+1-p}$, we have
\begin{equation*}\begin{split}
T_1\leq \gamma \rho^{-sp}\frac{2^j}{(\sigma_*\xi_1)^{2-p}\omega^{q+1-p}}\omega^{q-1}(\sigma_*\xi_1\omega)^2
Y_j|Q_{\rho_j}|=\gamma 2^{jp}\rho^{-sp}(\sigma_*\xi_1\omega)^p|Q_{\rho_j}|Y_j.
\end{split}
\end{equation*}
Finally, we consider the estimate of $T_3$. To this end, we apply H\"older's inequality and \eqref{measureTail1} to obtain
\begin{equation*}\begin{split}
T_3&\leq \gamma 2^{j(n+sp)}\rho^{-sp}(\sigma_*\xi_1\omega)^p|Q_{\rho_j}|Y_j
\\&\quad+2^{j(n+sp)}\hat\theta^{-\frac{p-1}{m}}\widetilde{\mathrm{Tail}}_m((u-\mu_+)_+;Q_0)
^{p-1}(\sigma_*\xi_1\omega) \rho^{sp\frac{m-(p-1)}{m}}|B_\rho|Z_j^{1+\kappa}
\\&\leq \gamma 2^{j(n+sp)}\rho^{-sp}(\sigma_*\xi_1\omega)^p|Q_{\rho_j}|(Y_j+Z_j^{1+\kappa}),
\end{split}
\end{equation*}
since $\hat \theta=\sigma_*\tilde\theta$.
Combining these inequalities, we arrive at
\begin{equation*}\begin{split}
T_{0,1}+T_{0,2}\leq \gamma 2^{j(n+sp)}\rho^{-sp}(\sigma_*\xi_1\omega)^p|Q_{\rho_j}|(Y_j+Z_j^{1+\kappa}).
\end{split}
\end{equation*}
On the other hand, we use H\"older's inequality and Lemma \ref{embedding} to deduce
\begin{equation*}\begin{split}
Y_{j+1}|Q_{\rho_{j+1}}|&\leq \iint_{Q_{\rho_{j+1}}}\frac{(v-\tilde k_j)_+^p}{(k_{j+1}-\tilde k_j)^p}\,\mathrm{d}x\mathrm{d}t
\\&\leq \gamma 2^{jp}(\sigma_*\xi_1\omega)^{-p}|Q_{\rho_{j+1}}|^\frac{sp}{n+sp}Y_j^\frac{sp}{n+sp}\left[
\int_{-\rho_{j+1}^{sp}}^0\left[(v-\tilde k_j)_+(\cdot,t)\right]_{W^{s,p}(B_{\rho_{j+1}})}^p\,\mathrm{d}t
\right.
\\&\left.
+\rho_{j+1}^{-sp}\|(v-\tilde k_j)_+\|_{L^p(Q_{\rho_{j+1}})}^p+\esssup_{-\rho_{j+1}^{sp}
<t<0}\int_{B_{\rho_{j+1}}}(v-\tilde k_j)_+^p(x,t)\,\mathrm{d}x
\right]
\\&=:T_4+T_5+T_6.
\end{split}
\end{equation*}
To estimate $T_4$, we note that
$-\tilde k_j\leq -k_j$, $(v-\tilde k_j)_+=(\tilde k_j-v)_-=(-v-(-\tilde k_j))_-$ and $(v-k_j)_+=(-v-(-k_j))_-$.
We use Lemma \ref{b<a} to obtain
\begin{equation*}\begin{split}
&\left[(v-\tilde k_j)_+(\cdot,t)\right]_{W^{s,p}(B_{\rho_{j+1}})}
=\left[(-v-(-\tilde k_j))_-(\cdot,t)\right]_{W^{s,p}(B_{\rho_{j+1}})}
\\&\leq \left[(-v-(- k_j))_-(\cdot,t)\right]_{W^{s,p}(B_{\rho_{j+1}})}=\left[(v- k_j)_+(\cdot,t)\right]_{W^{s,p}(B_{\rho_{j+1}})}.
\end{split}
\end{equation*}
It follows that
\begin{equation*}\begin{split}
T_4\leq \gamma 2^{jp}(\sigma_*\xi_1\omega)^{-p}|Q_{\rho_{j+1}}|^\frac{sp}{n+sp}Y_j^\frac{sp}{n+sp}T_{0,2}\leq \gamma 2^{(n+2p)j}
|Q_{\rho_j}|Y_j^\frac{sp}{n+sp}(Y_j+Z_j^{1+\kappa}).
\end{split}
\end{equation*}
Recalling that $(v-k_j)_+\leq \frac{1}{4}\sigma_*\xi_1\omega$ in $Q_\rho$, we get $T_5\leq \gamma 2^{jp}Y_j^{1+\frac{sp}{n+sp}}|Q_{\rho_{j+1}}|$.
Finally, we consider the estimate for $T_6$. To this end, we infer from \eqref{vandkj} and Lemma \ref{equivalent} that
\begin{equation*}\begin{split}
g_+(v,k_j)&\geq \gamma(|v|+|k_j|)^{q-1}(v-k_j)_+^2\geq \gamma\omega^{q-1}(v-k_j)_+^{2-p}(v-k_j)_+^p
\\&\geq \gamma \omega^{q-1}(\sigma_*\xi_1\omega)^{2-p}(v-k_j)_+^p
\\&=\gamma (\sigma_*\xi_1)^{2-p}\theta(v-k_j)_+^p=\gamma \hat\theta (v-k_j)_+^p,
\end{split}
\end{equation*}
since $p>2$, $\theta=\left(\frac{1}{4}\omega\right)^{q+1-p}$ and $\hat \theta=(\sigma_*\xi_1)^{2-p}\theta$.
It follows that
\begin{equation*}\begin{split}
T_6&\leq \gamma \gamma 2^{jp}(\sigma_*\xi_1\omega)^{-p}|Q_{\rho_{j+1}}|^\frac{sp}{n+sp}Y_j^\frac{sp}{n+sp} T_{0,1}
\\&\leq \gamma 2^{j(n+2p)}|Q_{\rho_j}|Y_j^\frac{sp}{n+sp}(Y_j+Z_j^{1+\kappa}).
\end{split}
\end{equation*}
Consequently, we infer that
\begin{equation}\begin{split}\label{2ndYj+1}
Y_{j+1}\leq \gamma 2^{j(2p+n)}\left(Y_j^{1+\frac{sp}{n+sp}}+Y_j^\frac{sp}{n+sp}Z_j^{1+\kappa}\right).
\end{split}
\end{equation}
To estimate $Z_{j+1}$,
we proceed similarly as in the proof of \cite[Lemma 4.1]{BK}. This leads to
\begin{equation}\begin{split}\label{2ndZj+1}
Z_{j+1}&\leq \gamma 2^{(n+2p)j}\left(Y_j+Z_j^{1+\kappa}\right).
\end{split}
\end{equation}
At this stage, we choose
\begin{equation}\begin{split}\label{nu2}
\nu_2=(4\gamma)^{-\frac{m(1+\kappa)}{(m-(p-1))\zeta}}2^{-(n+2p)\frac{m(1+\kappa)}{(m-(p-1))\zeta^2}},
\end{split}
\end{equation}
where $\zeta=\min\left\{\kappa,\frac{sp}{n+sp}\right\}$.
Using a lemma on fast geometric convergence of sequences (see for instance \cite[Chapter I, Lemma 4.2]{Di93}),
we conclude from \eqref{2ndYj+1} and \eqref{2ndZj+1} that $Y_n\to0$ as $n\to\infty$. This proves the desired estimate \eqref{DeGiorgi3} and
we have thus proved the lemma.
\end{proof}
We are now in a position to establish a decay estimate of $u$ for the second alternative and the following proposition is our main result in this subsection.
\begin{proposition}\label{2nd proposition}
Let $u$ be a bounded weak solution to \eqref{LKut}-\eqref{kernel}
 in the sense of Definition \ref{weak solution}.
Assume that \eqref{2nd} holds for any $t_0-(A-1)\theta\rho^{sp}\leq \bar t\leq t_0$.
Then, there exists a constant $\sigma_*<1$ depending only upon the data, such that if
\begin{equation}\begin{split}\label{Alt2Tail}
\frac{\rho^\frac{n\kappa}{p-1}}{
\theta^\frac{1}{m}}\widetilde{\mathrm{Tail}}_m((u-\mu_+)_+;Q_0)\leq (\sigma_*\xi_1)^{1-\frac{p-2}{m}}\omega,\end{split}\end{equation}
then
\begin{equation}\begin{split}\label{osc2}
\essosc_{\tilde Q_0} u\leq (1-\tfrac{1}{8}\sigma_*\xi_1)\omega,
 \end{split}\end{equation}
 where $\tilde Q_0=Q_{\frac{1}{4}\rho}^{(\theta)}(z_0)$.
\end{proposition}
\begin{proof} First, we note that \eqref{Alt2Tail} implies \eqref{measureTail}, since $\sigma_*<1$ and $\xi_1<1$. So, Lemma \ref{measure}
and Lemma \ref{measureshrinkinglemma} are at our disposal. Let $\nu_2<1$ be the constant defined in \eqref{nu2}. We now choose
$\sigma_*=(\gamma^{-1}\nu_0\nu_2)^\frac{1}{p-1}$, where $\gamma$ is the constant claimed by Lemma \ref{measureshrinkinglemma}. According to
Lemma \ref{measureshrinkinglemma}, we obtain \eqref{2nd Degiorgi assumption}. Moreover, we infer from Lemma \ref{lemmaDeGiorgi3} that
\eqref{DeGiorgi3} holds for such a choice of $\sigma_*$. This proves the desired estimate \eqref{osc2}, since $\tilde Q_0\subset Q_A^\prime$.
\end{proof}
From the proof of Proposition \ref{2nd proposition}, we remark that the choice of $\sigma_*=(\gamma^{-1}\nu_0\nu_2)^\frac{1}{p-1}$ also
determines the value of $A$ by $A=1+\sigma_*^{-(p-2)}\xi_1^{2-p}$. This also fixes the value of $\xi_0$ via \eqref{xi0}.
\subsection{The proof of Theorem \ref{main1}}\label{nearzeroproof}
This subsection is devoted to the proof of Theorem \ref{main1} under an assumption that "$u$ is near zero".
The precise definition of such assumption will be described in \eqref{nearzeroallj}.
Our proof is divided into five steps.

Step 1:
First, we observe that \eqref{osc1} or \eqref{osc2} still holds if $\mu_+-\mu_-\leq\frac{1}{2}\omega$.
Let $\nu_*=\min\{\xi_0,(\sigma_*\xi_1)^{1-\frac{p-2}{m}}\}$ and $\eta=\min\{\frac{1}{2}\xi_0,\frac{1}{8}\sigma_*\xi_1\}$.
Propositions \ref{1st proposition} and \ref{2nd proposition} may be summarized by saying that
if
 \begin{equation}\begin{split}\label{AltTailall}
\frac{\rho^\frac{n\kappa}{p-1}}{\theta^\frac{1}{m}}\widetilde{\mathrm{Tail}}_m((u-\mu_\pm)_\pm;Q_0)\leq\nu_*\omega,\end{split}\end{equation}
then
\begin{equation}\begin{split}\label{oscall}
\essosc_{\tilde Q_0} u\leq (1-\eta)\omega,
 \end{split}\end{equation}
 where $\tilde Q_0=Q_{\frac{1}{4}\rho}^{(\theta)}(z_0)$.

Step 2: \emph{We claim that there exists a constant $\hat\sigma=\hat\sigma(\text{data},\nu_*,A)<1$, such that if $\rho=\hat \sigma \rho_0$,
then \eqref{AltTailall} holds}. To prove the claim, we only deal with the case of the truncated function $(u-\mu_-)_-$ as the other case is similar.
First, we decompose
\begin{equation*}\begin{split}
&\frac{\rho^\frac{n\kappa}{p-1}}{\theta^\frac{1}{m}}\widetilde{\mathrm{Tail}}_m((u-\mu_-)_-;Q_0)
\\&\leq \gamma\frac{\rho^\frac{n\kappa}{p-1}}{\theta^\frac{1}{m}}\left[\int_{t_0-A\theta\rho_0^{sp}}^{t_0}\left(\int_{B_R
\setminus B_{\rho_0}(x_0)}
\frac{(u-\mu_-)_-(y,t)^{p-1}}{|y-x_0|^{n+sp}}
\,\mathrm {d}y\right)^\frac{m}{p-1}\,\mathrm {d}t\right]^\frac{1}{m}
\\&+\gamma\frac{\rho^\frac{n\kappa}{p-1}}{\theta^\frac{1}{m}}\left[\int_{t_0-A\theta\rho_0^{sp}}^{t_0}\left(\int_{\mathbb{R}^n\setminus B_R}
\frac{(u-\mu_-)_-(y,t)^{p-1}}{|y-x_0|^{n+sp}}
\,\mathrm {d}y\right)^\frac{m}{p-1}\,\mathrm {d}t\right]^\frac{1}{m}
\\&=:T_1+T_2.
\end{split}\end{equation*}
To estimate $T_1$, we note that $(u-\mu_-)_-\leq |u|+|\mu_-|\leq 2\sup_{Q_R}|u|$ in $Q_R$. It follows from \eqref{omegaxi0} that
\begin{equation*}\begin{split}
T_1&\leq \gamma\frac{\rho^\frac{n\kappa}{p-1}}{\theta^\frac{1}{m}}\sup_{Q_R}|u|\cdot\left[\int_{t_0-A\theta\rho_0^{sp}}^{t_0}\left(\int_{\mathbb{R}^n
\setminus B_{\rho_0}(x_0)}
\frac{1}{|y-x_0|^{n+sp}}
\,\mathrm {d}y\right)^\frac{m}{p-1}\,\mathrm {d}t\right]^\frac{1}{m}
\\&\leq \gamma \omega\frac{\rho^\frac{n\kappa}{p-1}}{\theta^\frac{1}{m}}(A\theta\rho_0^{sp})^\frac{1}{m}\rho_0^{-\frac{sp}{p-1}}
=\gamma A^\frac{1}{m}\left(\frac{\rho}{\rho_0}\right)^\frac{n\kappa}{p-1}\omega.
\end{split}\end{equation*}
Next, we consider the estimate for $T_2$. To this end, we observe that $|y-x_0|\geq |y|-|x_0|\geq\frac{3}{4}|y|$ holds for all
$|y|\geq R$, since $|x_0|\leq \frac{1}{4}R\leq \frac{1}{4}|y|$. In view of $(u-\mu_-)_-\leq |u|+|\mu_-|\leq |u|+\omega$
and $(t_0-A\theta\rho_0^{sp},t_0)\subset (-R^{sp},0)$
, we infer from
\eqref{omegaxi0} that
\begin{equation*}\begin{split}\text{Tail}_m(|u|;Q_R)<\omega^\frac{m-(p-q-1)}{m}=4^\frac{q+1-p}{m}\theta^\frac{1}{m}\omega
\end{split}\end{equation*}
and hence
\begin{equation*}\begin{split}
T_2&\leq \gamma\frac{\rho^\frac{n\kappa}{p-1}}{\theta^\frac{1}{m}}\left[\int_{-R^{sp}}^{0}\left(\int_{\mathbb{R}^n\setminus B_R}
\frac{|u(y,t)|^{p-1}}{|y|^{n+sp}}
\,\mathrm {d}y\right)^\frac{m}{p-1}\,\mathrm {d}t\right]^\frac{1}{m}
\\&+\gamma\omega\frac{\rho^\frac{n\kappa}{p-1}}{\theta^\frac{1}{m}}\left[\int_{t_0-A\theta\rho_0^{sp}}^{t_0}\left(\int_
{\mathbb{R}^n\setminus B_R}
\frac{1}{|y|^{n+sp}}
\,\mathrm {d}y\right)^\frac{m}{p-1}\,\mathrm {d}t\right]^\frac{1}{m}
\\&\leq \gamma\frac{\rho^\frac{n\kappa}{p-1}}{\theta^\frac{1}{m}}R^{\frac{sp}{m}-\frac{sp}{p-1}}\text{Tail}_m(|u|;Q_R)
+\gamma\omega\frac{\rho^\frac{n\kappa}{p-1}}{\theta^\frac{1}{m}}(A\theta\rho_0^{sp})^\frac{1}{m}R^{-\frac{sp}{p-1}}
\\&\leq \gamma (1+A^\frac{1}{m})\left(\frac{\rho}{\rho_0}\right)^\frac{n\kappa}{p-1}\omega,
\end{split}\end{equation*}
since $\rho_0<R$. Therefore, we conclude that there exists a constant $\gamma=\gamma(\text{data})$, such that
\begin{equation*}\begin{split}
\frac{\rho^\frac{n\kappa}{p-1}}{\theta^\frac{1}{m}}\widetilde{\mathrm{Tail}}_m((u-\mu_-)_-;Q_0)
\leq\gamma (1+A^\frac{1}{m})\hat\sigma^\frac{n\kappa}{p-1}\omega.
\end{split}\end{equation*}
At this point, we choose $\hat\sigma>0$ so small that $\gamma (1+A^\frac{1}{m})\hat\sigma^\frac{n\kappa}{p-1}<\nu_*$.
This proves the inequality \eqref{AltTailall}, which is our claim.

Step 3: Let us denote by $\tilde\rho_1$ the radius $\rho$ in Step 1.
Here and subsequently, $\omega_0$ and $\theta_0$ stand for $\omega$ and $\theta$ in Step 1, respectively. Moreover, we set
$\omega_1=(1-\eta)\omega_0$ and $\theta_1=\left(\frac{1}{4}\omega_1\right)^{q+1-p}$. Our task now is to construct a smaller
radius $\rho_1<\tilde\rho_1$ and a scaled parabolic
cylinder $Q_1=Q_{\rho_1}^{(A\theta_1)}(z_0)$ such that $Q_1\subseteq \tilde Q_0$. To this end, we define
$\lambda=\frac{1}{4}(1-\eta)^\frac{p-(q+1)}{sp}A^{-\frac{1}{sp}}$,
$\hat\delta=\lambda\hat
\sigma$
and
$\rho_1=\lambda\tilde\rho_1=\lambda\hat
\sigma\rho_0=\hat\delta \rho_0$.
At this point, we set $Q_1=Q_{\rho_1}^{(A\theta_1)}(z_0)$, $\mu_1^-=\essinf_{Q_1}u$ and $\mu_1^+=\esssup_{Q_1}u$.
In view of $\rho_1=\lambda\tilde\rho_1$, we find that $Q_1\subseteq \tilde Q_0$; hence
\begin{equation*}\begin{split}\essosc_{Q_1}u=\mu_1^+-\mu_1^-\leq\essosc_{\tilde Q_0}u\leq \omega_1.\end{split}\end{equation*}
Let $\sigma<1$ be a constant that will be determined later.
Next, we define $\delta=\lambda\sigma$, where $\lambda=\frac{1}{4}(1-\eta)^\frac{p-(q+1)}{sp}A^{-\frac{1}{sp}}$.
For $j\geq2$, we set $\tilde \rho_j=\sigma\rho_{j-1}$
and $\rho_j=\lambda\tilde\rho_{j}=\delta \rho_{j-1}$. Set $\omega_j=(1-\eta)\omega_{j-1}=(1-\eta)^j\omega_0$, $\theta_j=\left(\frac{1}{4}\omega_j\right)^{q+1-p}$,
$\tilde Q_{j-1}=Q_{\frac{1}{4}\tilde\rho_j}^{(\theta_{j-1})}(z_0)$, $Q_j=Q_{\rho_j}^{(A\theta_j)}(z_0)$, $\mu_j^-=\essinf_{Q_j}u$ and $\mu_j^+
=\esssup_{Q_j}u$. Since $\rho_j=\lambda\tilde\rho_j$ and $\tilde \rho_j<\rho_{j-1}$, we find that $Q_j\subseteq \tilde Q_{j-1}\subseteq Q_{j-1}$
holds for any $j\geq1$. In the rest of this section, we assume that
\begin{equation}\begin{split}\label{nearzeroallj}\mu_j^-\leq\xi_0\omega_j\qquad \text{and}\qquad \mu_j^+\geq-\xi_0\omega_j\quad \text{hold}\ \text{for}\ \text{all}\quad
j\geq1,
\end{split}\end{equation}
where $\xi_0$ is the constant defined in \eqref{xi0}. Recalling that $\omega_0=\omega$ is defined in \eqref{omegaxi0}, we find that
\eqref{nearzeroallj} holds for $j=0$.

Step 4: \emph{Let $j_0\geq1$ be a fixed integer. Assume that for $j=1,\cdots,j_0$, there holds
\begin{equation}\begin{split}\label{jAltTailall}
\frac{\tilde\rho_j^\frac{n\kappa}{p-1}}{\theta_{j-1}
^\frac{1}{m}}\widetilde{\mathrm{Tail}}_m((u-\mu_{j-1}^\pm)_\pm;Q_{j-1})\leq\nu_*\omega_{j-1}.\end{split}\end{equation}
Moreover, repeating the arguments as in Step 1, we conclude that
\begin{equation}\begin{split}\label{assumptionforosc}
\essosc_{\tilde Q_{j-1}}u\leq\omega_j\end{split}\end{equation}
holds for any $1\leq j\leq j_0$. Then, we claim that there exists a constant $\sigma<1$ depending only upon the data, such that if $\tilde \rho_j=\sigma\rho_{j-1}$ holds for any $j\geq2$,
then
\begin{equation}\begin{split}\label{j0AltTailall}
\frac{\tilde\rho_{j_0+1}^\frac{n\kappa}{p-1}}{\theta_{j_0}
^\frac{1}{m}}\widetilde{\mathrm{Tail}}_m((u-\mu_{j_0}^\pm)_\pm;Q_{j_0})\leq\nu_*\omega_{j_0}.\end{split}\end{equation}}
To prove the claim, we only treat the case of the truncated function $(u-\mu_{j_0}^-)_-$ as the other case is similar.
We first consider the case when $j_0\geq2$.
Since $p>2$, we have $\frac{1}{p-1}<1$ and $m>p-1>1$.
To simplify the presentation, we write $U(y,t)=(u-\mu_{j_0}^-)_-(y,t)^{p-1}|y-x_0|^{-n-sp}$ and obtain
\begin{equation*}\begin{split}
&\frac{\tilde\rho_{j_0+1}^\frac{n\kappa}{p-1}}{\theta_{j_0}
^\frac{1}{m}}\widetilde{\mathrm{Tail}}_m((u-\mu_{j_0}^-)_-;Q_{j_0})
\\&\leq \frac{\tilde\rho_{j_0+1}^\frac{n\kappa}{p-1}}{\theta_{j_0}
^\frac{1}{m}}\left[\int_{t_0-A\theta_{j_0}\rho_{j_0}^{sp}}^{t_0}\left(\sum_{l=2}^{j_0}\left(\int_{B_{\frac{1}{4}\tilde\rho_{l-1}
}(x_0)\setminus B_{\frac{1}{4}\tilde\rho_l}
(x_0)}U(y,t)\,\mathrm{d}y
\right)^\frac{1}{p-1}
\right.\right.
\\&+
\left(\int_{B_{\frac{1}{4}\tilde\rho_{j_0}
}(x_0)\setminus B_{\rho_{j_0}}
(x_0)}U(y,t)\,\mathrm{d}y
\right)^\frac{1}{p-1}
+\left(\int_{B_{\rho_0}(x_0)\setminus B_{\frac{1}{4}\tilde\rho_1}
(x_0)}U(y,t)\,\mathrm{d}y
\right)^\frac{1}{p-1}
\\&\left.\left.+\left(\int_{\mathbb{R}^n\setminus B_{\rho_0}(x_0)}U(y,t)\,\mathrm{d}y
\right)^\frac{1}{p-1}\right)^m\,\mathrm{d}t\right]^\frac{1}{m},
\end{split}\end{equation*}
since $\frac{1}{p-1}<1$.
By Minkowski's inequality, we get
\begin{equation*}\begin{split}
&\frac{\tilde\rho_{j_0+1}^\frac{n\kappa}{p-1}}{\theta_{j_0}
^\frac{1}{m}}\widetilde{\mathrm{Tail}}_m((u-\mu_{j_0}^-)_-;Q_{j_0})
\\&\leq \frac{\tilde\rho_{j_0+1}^\frac{n\kappa}{p-1}}{\theta_{j_0}
^\frac{1}{m}}\sum_{l=2}^{j_0}\left(\int_{t_0-A\theta_{j_0}\rho_{j_0}^{sp}}^{t_0}
\left(\int_{B_{\frac{1}{4}\tilde\rho_{l-1}}(x_0)\setminus B_{\frac{1}{4}\tilde\rho_l}
(x_0)}U(y,t)\,\mathrm{d}y
\right)^\frac{m}{p-1}\,\mathrm{d}t\right)^\frac{1}{m}
\\&+\frac{\tilde\rho_{j_0+1}^\frac{n\kappa}{p-1}}{\theta_{j_0}
^\frac{1}{m}}\left(\int_{t_0-A\theta_{j_0}\rho_{j_0}^{sp}}^{t_0}
\left(\int_{B_{\rho_0}(x_0)\setminus B_{\frac{1}{4}\tilde\rho_1}
(x_0)}U(y,t)\,\mathrm{d}y
\right)^\frac{m}{p-1}\,\mathrm{d}t\right)^\frac{1}{m}
\\&+\frac{\tilde\rho_{j_0+1}^\frac{n\kappa}{p-1}}{\theta_{j_0}
^\frac{1}{m}}\left(\int_{t_0-A\theta_{j_0}\rho_{j_0}^{sp}}^{t_0}
\left(\int_{B_{\frac{1}{4}\tilde\rho_{j_0}
}(x_0)\setminus B_{\rho_{j_0}}
(x_0)}U(y,t)\,\mathrm{d}y
\right)^\frac{m}{p-1}\,\mathrm{d}t\right)^\frac{1}{m}
\\&+ \frac{\tilde\rho_{j_0+1}^\frac{n\kappa}{p-1}}{\theta_{j_0}
^\frac{1}{m}}\left(\int_{t_0-A\theta_{j_0}\rho_{j_0}^{sp}}^{t_0}
\left(\int_{\mathbb{R}^n\setminus B_{\rho_0}(x_0)}U(y,t)\,\mathrm{d}y
\right)^\frac{m}{p-1}\,\mathrm{d}t\right)^\frac{1}{m}
\\&=:\sum_{l=2}^{j_0}T_l+T+T^\prime+T^{\prime\prime}.
\end{split}\end{equation*}
For $l=2,\cdots, j_0$, we observe that $B_{\frac{1}{4}\tilde\rho_{l-1}}(x_0)\times (t_0-A\theta_{j_0}\rho_{j_0}^{sp},t_0)\subseteq \tilde
Q_{l-2}$. It follows from
\eqref{assumptionforosc} that
$(u-\mu_{j_0}^-)_-\leq \essosc_{\tilde Q_{l-2}}u\leq \omega_{l-1}$ on the set $\tilde Q_{l-2}$. Noting that $\tilde\rho_{j_0+1}=\sigma\rho_{j_0}$,
$\rho_{j_0}=\delta^{j_0-l}\rho_l=(\lambda\sigma)^{j_0-l}\rho_l$, $\tilde\rho_l>\rho_l$ and $\omega_{j_0}=(1-\eta)^{j_0-l}\omega_l$,
we obtain
\begin{equation*}\begin{split}
T_l&\leq \gamma \frac{\tilde\rho_{j_0+1}^\frac{n\kappa}{p-1}}{\theta_{j_0}
^\frac{1}{m}}\omega_{l-1}\tilde\rho_l^{-\frac{sp}{p-1}}(A\theta_{j_0}\rho_{j_0}^{sp})^\frac{1}{m}
\leq\gamma(1-\eta)^{-1}A^\frac{1}{m}\sigma^\frac{n\kappa}{p-1}\left(\frac{\rho_{j_0}}{\rho_l}\right)^\frac{sp}{p-1}
\left(\frac{\omega_l}{\omega_{j_0}}\right)\omega_{j_0}
\\&\leq \gamma (1-\eta)^{-1}
A^\frac{1}{m}\sigma^\frac{n\kappa}{p-1}\left[(\lambda\sigma)^\frac{sp}{p-1}(1-\eta)^{-1}\right]^{j_0-l}\omega_{j_0}.
\end{split}\end{equation*}
At this point, we set
\begin{equation}\begin{split}\label{sigmadef}
\sigma=\left[\tfrac{1}{100}\gamma^{-1}A^{-\frac{1}{m}}\nu_*(1-\eta)\right]^\frac{p-1}{n\kappa}.
\end{split}\end{equation}
With this choice of $\sigma$, we find that $(\lambda\sigma)^\frac{sp}{p-1}<\frac{1}{2}(1-\eta)$; hence
\begin{equation*}\begin{split}
\sum_{l=2}^{j_0}T_l&\leq \gamma (1-\eta)^{-1}A^\frac{1}{m}\sigma^\frac{n\kappa}{p-1}\omega_{j_0}\sum_{l=0}^\infty2^{-l}\leq\tfrac{1}{5}\nu_*\omega_{j_0}.
\end{split}\end{equation*}
To estimate $T$, we observe that
$B_{\rho_0}(x_0)\times (t_0-A\theta_{j_0}\rho_{j_0}^{sp},t_0)\subseteq Q_R$; hence $(u-\mu_{j_0}^-)_-\leq 2\omega_0$. It follows that
\begin{equation*}\begin{split}
T&\leq \gamma \frac{\tilde\rho_{j_0+1}^\frac{n\kappa}{p-1}}{\theta_{j_0}
^\frac{1}{m}}\omega_0\tilde\rho_1^{-\frac{sp}{p-1}}(A\theta_{j_0}\rho_{j_0}^{sp})^\frac{1}{m}\leq \gamma A^\frac{1}{m}\sigma^\frac{n\kappa}{p-1}
(1-\eta)^{-j_0}\left(\frac{\rho_{j_0}}{\rho_1}\right)^\frac{sp}{p-1}\omega_{j_0}.
\end{split}\end{equation*}
Noting that $j_0\geq 2>1$, $\rho_{j_0}=(\lambda\sigma)^{j_0-1}\rho_1$ and $(\lambda\sigma)^\frac{sp}{p-1}<\frac{1}{2}(1-\eta)$, we get
\begin{equation*}\begin{split}
T&\leq \gamma A^\frac{1}{m}\sigma^\frac{n\kappa}{p-1}(\lambda\sigma)^{-\frac{sp}{p-1}}
\left(\frac{(\lambda\sigma)^\frac{sp}{p-1}}{1-\eta}\right)^{j_0}\omega_{j_0}
\leq \gamma A^\frac{1}{m}\sigma^\frac{n\kappa}{p-1}(1-\eta)^{-1}\omega_{j_0}
\leq \tfrac{1}{5}\nu_*\omega_{j_0}.
\end{split}\end{equation*}
Next, we consider the estimate for $T^\prime$. In view of $B_{\frac{1}{4}\tilde\rho_{j_0}
}(x_0)\times (t_0-A\theta_{j_0}\rho_{j_0}^{sp},t_0)\subseteq \tilde Q_{j_0-1}$, we conclude from \eqref{assumptionforosc} that
$(u-\mu_{j_0}^-)_-\leq \essosc_{\tilde Q_{j_0-1}}u\leq \omega_{j_0}$ on the set $\tilde Q_{j_0-1}$; hence
\begin{equation*}\begin{split}
T^\prime&\leq \gamma \frac{\tilde\rho_{j_0+1}^\frac{n\kappa}{p-1}}{\theta_{j_0}
^\frac{1}{m}}\omega_{j_0}\rho_{j_0}^{-\frac{sp}{p-1}}(A\theta_{j_0}\rho_{j_0}
^{sp})^\frac{1}{m}= \gamma A^\frac{1}{m}\sigma^\frac{n\kappa}{p-1}
\omega_{j_0}\leq \tfrac{1}{5}\nu_*\omega_{j_0}.
\end{split}\end{equation*}
Finally, we come to the estimate of $T^{\prime\prime}$.
To this end, we observe that $(u-\mu_{j_0}^-)_-(y,t)\leq (u-\mu_{0}^-)_-(y,t)+\mu_{j_0}^--\mu_0^-\leq (u-\mu_{0}^-)_-(y,t)+\omega_0$.
According to Step 2, we get
\begin{equation*}\begin{split}
T^{\prime\prime}&\leq
\sigma^{j_0\frac{n\kappa}{p-1}}\frac{\tilde\rho_1^\frac{n\kappa}{p-1}}{\theta_0
^\frac{1}{m}}\widetilde{\mathrm{Tail}}_m((u-\mu_0^-)_-;Q_0)
+\gamma \frac{\tilde\rho_{j_0+1}^\frac{n\kappa}{p-1}}{\theta_{j_0}
^\frac{1}{m}}\omega_0\rho_0^{-\frac{sp}{p-1}}(A\theta_{j_0}\rho_{j_0}^{sp})^\frac{1}{m}
\\&\leq \sigma^{j_0\frac{n\kappa}{p-1}} \nu_*\omega_0+\gamma A^\frac{1}{m}\sigma^\frac{n\kappa}{p-1}
(1-\eta)^{-j_0}\left(\frac{\rho_{j_0}}{\rho_1}\right)^\frac{sp}{p-1}\omega_{j_0}
,\end{split}\end{equation*}
since $\theta_{j_0}\geq \theta_0$ and $\rho_0>\rho_1$. Similar to the estimate of $T$, we infer from \eqref{sigmadef} that
\begin{equation*}\begin{split}
T^{\prime\prime}&\leq  \gamma A^\frac{1}{m}\sigma^\frac{n\kappa}{p-1}(1-\eta)^{-1}\omega_{j_0}+\tfrac{1}{5}\nu_*\omega_{j_0}
<\tfrac{2}{5}\nu_*\omega_{j_0}.
\end{split}\end{equation*}
This proves the desired inequality \eqref{j0AltTailall} for $j_0\geq2$. Finally, we consider the case $j_0=1$. This is an easy case and we have
\begin{equation*}\begin{split}
&\frac{\tilde\rho_{2}^\frac{n\kappa}{p-1}}{\theta_{1}
^\frac{1}{m}}\widetilde{\mathrm{Tail}}_m((u-\mu_{1}^-)_-;Q_{1})
\leq T+T^\prime+T^{\prime\prime}<\nu_*\omega_{j_0}.
\end{split}\end{equation*}
This completes the proof of the claim.

Step 5:
\emph{Proof of the inequality \eqref{theorem1oscillation}.}
According to the assumption \eqref{nearzeroallj}, we conclude from \eqref{j0AltTailall} and Step 1 that
\begin{equation*}\begin{split}
\essosc_{Q_{j_0+1}}u\leq
\essosc_{\tilde Q_{j_0}}u\leq\omega_{j_0+1}.\end{split}\end{equation*}
This also implies that \eqref{assumptionforosc} holds for $j=j_{0}+1$.
We can repeat the arguments from Step 3, and there holds
\begin{equation}\begin{split}\label{assumptionforoscforallj}
\essosc_{Q_{\rho_j}^{(\theta_0)}(z_0)} u\leq
\essosc_{Q_j}u\leq\omega_j=(1-\eta)^j\omega_0\end{split}\end{equation}
for all $j\geq0$. Here, $\rho_1=\hat\delta\rho_0$ and $\rho_{j+1}=\delta\rho_j$ for $j\geq1$. For any fixed $r<\rho_1$,
there exists an integer $\bar j\geq0$, such that $\delta^{\bar j+1}\rho_1\leq r\leq \delta^{\bar j}\rho_1$. Moreover, we apply
\eqref{assumptionforoscforallj} to obtain
\begin{equation*}\begin{split}
\essosc_{Q_r^{(\theta_0)}(z_0)} u&\leq
\essosc_{Q_{\bar j+1}}u\leq\omega_{\bar j+1}
\\&\leq \gamma\left(\frac{r}{\rho_1}\right)^{\alpha_0}\omega_0\leq \gamma_0\left(\frac{r}{R}\right)^{\alpha_0},\end{split}\end{equation*}
where $\alpha_0=\frac{\ln(1-\eta)}{\ln\delta}$
and the constant $\gamma_0$ depends on the data, $R$, $\|u\|_\infty$ and $\mathrm{Tail}_m(|u|;Q_{R})$.
This proves the inequality \eqref{theorem1oscillation} under the assumption \eqref{nearzeroallj}.
\section{Reduction of oscillation away from zero}
In this section we assume that there exists an integer $j\geq1$, such that \eqref{nearzeroallj} does not hold.
Suppose that $i_0\geq1$ is the first index satisfying $\mu_{i_0}^->\xi_0\omega_{i_0}$ or $\mu_{i_0}^+<-\xi_0\omega_{i_0}$.
We only consider the case of $\mu_{i_0}^->\xi_0\omega_{i_0}$. This is because $-u$ is also the weak solution to \eqref{LKut}-\eqref{kernel},
and hence the case $\mu_{i_0}^+<-\xi_0\omega_{i_0}$ can be addressed analogously.

Since $i_0$ is the first index such that \eqref{nearzeroallj} does not hold, we find that
$\mu_{i_0-1}^+\leq\mu_{i_0-1}^-+\omega_{i_0-1}\leq (1+\xi_0)\omega_{i_0-1}$. Recalling that $Q_{i_0}\subset Q_{i_0-1}$ and $\omega_{i_0}=(1-\eta)\omega_{i_0-1}$, we see that
$\mu_{i_0}^-\leq \mu_{i_0-1}^+\leq(1+\xi_0)\omega_{i_0-1}\leq\frac{1+\xi_0}{1-\eta}\omega_{i_0}$. This implies that \begin{equation}\label{ximuequivalent}\xi_0\omega_{i_0}<\mu_{i_0}^-\leq\mu_{i_0}^+\leq \frac{1+\xi_0}{1-\eta}\omega_{i_0}.\end{equation}
On the other hand, we  conclude from Step 2 and Step 4 of subsection \ref{nearzeroproof} that
\begin{equation}\begin{split}\label{i0AltTailall}
\frac{\tilde\rho_{i_0}^\frac{n\kappa}{p-1}}{\theta_{i_0-1}
^\frac{1}{m}}\widetilde{\mathrm{Tail}}_m((u-\mu_{i_0-1}^\pm)_\pm;Q_{i_0-1})\leq\nu_*\omega_{i_0-1}.\end{split}\end{equation}
In the following we omit for simplification the index $i_0$ in our notation until subsection \ref{finalsubsection}.
We introduce a
new function $\tilde v=u/\mu_-$. In view of \eqref{ximuequivalent}, we infer that
\begin{equation}\label{tildev}1\leq\tilde v\leq \frac{\mu^+}{\mu^-}\leq \frac{\mu^-+\omega}{\mu^-}\leq\frac{1+\xi_0}{\xi_0}
\qquad\text{a.e.}\quad\text{in}\quad Q.\end{equation}
Here, we abbreviate $Q=Q_{i_0}$. Moreover, it follows from \eqref{tildev} that $\tilde v>0$ in $Q$, and there holds
\begin{equation}\begin{split}\label{infsuptildev}
\mu_{\tilde v}^-:=
\essinf_Q\tilde v=\frac{1}{\mu^-}\essinf_Qu=1
\quad\text{and}\quad
\mu_{\tilde v}^+:=
\esssup_Q\tilde v=\frac{1}{\mu^-}\esssup_Qu\leq\frac{1+\xi_0}{\xi_0}.
\end{split}\end{equation}
To proceed further, we set $\omega_{\tilde v}=\mu_{\tilde v}^+-\mu_{\tilde v}^-$. We infer from \eqref{ximuequivalent} and \eqref{infsuptildev}
that $\omega_{\tilde v}=(\mu^-)^{-1}\omega$ and
$\frac{1-\eta}{1+\xi_0}\leq\omega_{\tilde v}<\xi_0^{-1}$.
From the definition of $\tilde v$, we have
\begin{equation*}\begin{split}
\tilde v\in C_{\loc}(0,T;L_{\loc}^{q+1}(\Omega))\cap L_{\loc}^p(0,T;W_{\loc}^{s,p}(\Omega))
\cap L_{\loc}^m(0,T;L_{sp}^{p-1}(\mathbb{R}^n)).
\end{split}
\end{equation*}
Let
$U\Subset B_\rho(x_0)$
be an open set. Next, we define $\hat v=\tilde v^q$ on the set $Q$.
For every subinterval $(t_1,t_2)\subset(t_0-A\theta\rho^{sp},t_0)$, we infer from \eqref{weaksolution} that
the identity
\begin{equation}\begin{split}\label{weaksolutionfortildev}
\int_U &\hat v(\cdot,t)\varphi(\cdot,t)\,\mathrm{d}x\bigg|_{t=t_1}^{t_2}+\iint_{U\times(t_1,t_2)}-\hat v\partial_t
 \varphi
\,\mathrm {d}x\mathrm {d}t
\\&+(\mu^-)^{p-q-1}\int_{t_1}^{t_2}\iint_{B_\rho(x_0)\times B_\rho(x_0)}K(x,y,t)|\hat v(x,t)^\frac{1}{q}-\hat v(y,t)^\frac{1}{q}|^{p-2}(
\hat v(x,t)^\frac{1}{q}-\hat v(y,t)^\frac{1}{q})
\\&\qquad\qquad\times(\varphi(x,t)-\varphi(y,t))
\,\mathrm {d}x\mathrm {d}y\mathrm {d}t
\\&+2(\mu^-)^{p-q-1}\int_{t_1}^{t_2}\int_{\mathbb{R}^n\setminus B_\rho(x_0)}\int_{B_\rho(x_0)}K(x,y,t)|\hat v(x,t)^\frac{1}{q}-\tilde
 v(y,t)|^{p-2}
\\&\qquad\qquad\times (
\hat v(x,t)^\frac{1}{q}-\tilde v(y,t))\varphi(x,t)
\,\mathrm {d}x\mathrm {d}y\mathrm {d}t=0
\end{split}\end{equation}
holds for
any testing function
\begin{equation*}\varphi\in W_{\loc}^{1,q+1}(t_0-A\theta\rho^{sp},t_0;L^{q+1}(U))\cap L_{\loc}^p(t_0-A\theta\rho^{sp},t_0;W_0^{s,p}(U)).
\end{equation*}
At this point, we set $T_0=(\mu^-)^{p-q-1}A\theta\rho^{sp}$ and $\hat Q=B_\rho(x_0)\times(-T_0,0)$. Moreover, we introduce the new functions
\begin{equation*}
v(x,t)=\hat v(x,t_0+(\mu^-)^{q-p+1}t)\qquad\text{on}\ \text{the}\ \text{set}\quad \hat Q
\end{equation*}
and $\bar v(x,t)=\tilde v(x,t_0+(\mu^-)^{q-p+1}t)$.
For any subinterval $(t_1,t_2)\subseteq(-T_0,0)$ and open set $U\Subset B_\rho(x_0)$, we conclude from \eqref{weaksolutionfortildev} that
\begin{equation}\begin{split}\label{weaksolutionforv}
\int_U & v(\cdot,t)\varphi(\cdot,t)\,\mathrm{d}x\bigg|_{t=t_1}^{t_2}+\iint_{U\times(t_1,t_2)}-v\partial_t
 \varphi
\,\mathrm {d}x\mathrm {d}t
\\&+\int_{t_1}^{t_2}\iint_{B_\rho(x_0)\times B_\rho(x_0)}\tilde K(x,y,t)|v(x,t)^\frac{1}{q}-v(y,t)^\frac{1}{q}|^{p-2}(
v(x,t)^\frac{1}{q}-v(y,t)^\frac{1}{q})
\\&\qquad\qquad\times(\varphi(x,t)-\varphi(y,t))
\,\mathrm {d}x\mathrm {d}y\mathrm {d}t
\\&+2\int_{t_1}^{t_2}\int_{\mathbb{R}^n\setminus B_\rho(x_0)}\int_{B_\rho(x_0)}\tilde K(x,y,t)| v(x,t)^\frac{1}{q}-\bar
 v(y,t)|^{p-2}
\\&\qquad\qquad\times (
v(x,t)^\frac{1}{q}-\bar v(y,t))\varphi(x,t)
\,\mathrm {d}x\mathrm {d}y\mathrm {d}t=0
\end{split}\end{equation}
holds for
any testing function
$\varphi\in W_{\loc}^{1,q+1}(-T_0,0;L^{q+1}(U))\cap L_{\loc}^p(-T_0,0;W_0^{s,p}(U))$.
Here, $\tilde K(x,y,t)=K(x,y,t_0+(\mu^-)^{q-p+1}t)$ and satisfies \eqref{kernel}. Moreover, we infer from \eqref{tildev} that
\begin{equation}\label{barvv}
1\leq\bar v\leq \frac{1+\xi_0}{\xi_0}
\qquad\text{a.e.}\quad\text{in}\quad \hat Q.\end{equation}
According to \eqref{infsuptildev}, we find that
\begin{equation}\begin{split}\label{infsupbarv}
\mu_{\bar v}^-:=
\essinf_{\hat Q}\bar v=1
\qquad\text{and}\qquad
\mu_{\bar v}^+:=
\esssup_{\hat Q} \bar v\leq\frac{1+\xi_0}{\xi_0}.
\end{split}\end{equation}
Furthermore, we denote $\omega_{\bar v}=\mu_{\bar v}^+-\mu_{\bar v}^-$. It follows that
$\omega_{\bar v}=\omega_{\tilde v}=(\mu^-)^{-1}\omega$; hence
\begin{equation}\begin{split}\label{barvomega}\frac{1-\eta}{1+\xi_0}\leq\omega_{\bar v}<\frac{1}{\xi_0}.\end{split}\end{equation}
At this stage, we define
\begin{equation*}\begin{split}
\mu_{ v}^-:=
\essinf_{\hat Q} v=(\mu_{\bar v}^-)^q=1,\quad
\mu_{ v}^+:=
\esssup_{\hat Q} v=(\mu_{\bar v}^+)^q
\end{split}\end{equation*}
\begin{equation*}\begin{split}
\quad\text{and}\quad
\omega_v=\essosc_{\hat Q}v=\mu_{ v}^+-\mu_{ v}^-.
\end{split}\end{equation*}
By the mean value theorem, we conclude from \eqref{infsupbarv} that
\begin{equation}\label{gamma0}
\omega_v=(\mu_{\bar v}^+)^q-(\mu_{\bar v}^-)^q\leq \gamma_0(\mu_{\bar v}^+-\mu_{\bar v}^-)= \gamma_0 \omega_{\bar v},\end{equation}
where the constant $\gamma_0$
depends only upon $q$ and $\xi_0$. Similarly, we find that
there exists a constant $\gamma_1=\gamma_1(q,\xi_0)$, such that
\begin{equation}\begin{split}\label{gamma1}\omega_{\bar v}&=\mu_{\bar v}^+-\mu_{\bar v}^-=[(\mu_{\bar v}^+)^q]^{1/q}-[(\mu_{\bar v}^-)^q]^{1/q}
\\&\leq\gamma_1 [(\mu_{\bar v}^+)^q-(\mu_{\bar v}^-)^q]=\gamma_1 \omega_v.\end{split}\end{equation}
Consequently, we infer from \eqref{barvomega} that
there exist positive constants $C_1$ and $C_2$ depending only upon the data, such that
\begin{equation}\begin{split}\label{omegavc1c2}
C_1\leq\omega_v\leq C_2.
\end{split}\end{equation}
At this point, we fix a constant $\tilde \omega\in (\omega_v,2\omega_v)$.
The inequality \eqref{omegavc1c2} means that $\omega_v$ or $\tilde\omega$ can be considered as a constant.
Finally, we claim that $v\in C(-T_0,0;L^2(B_\rho(x_0)))$. To prove this claim, we observe that
$\bar v\in C(-T_0,0;L^{q+1}(B_\rho(x_0)))$. Recalling that $v=\bar v^q$ in
 $\hat Q$, we infer from \eqref{barvv} that for any
$s,t\in(-T_0,0)$, there holds
\begin{equation*}\begin{split}
\|v(\cdot,t)-v(\cdot,s)\|_{L^2(B_\rho(x_0))}^2&\leq\gamma\|v\|_{L^\infty(\hat Q)}^{2-\hat q}\|\bar v(\cdot,t)^q-\bar v(\cdot,s)^q\|_{L^{\hat q}(B_\rho(x_0))}^{\hat q}
\\&\leq\gamma \|\bar v(\cdot,t)-\bar v(\cdot,s)\|_{L^{\hat q}(B_\rho(x_0))}^{\hat q}\to 0\qquad\text{as}\quad s\to t,
\end{split}\end{equation*}
where $\hat q=\min\{2,q+1\}$. This proves the claim.
\subsection{Caccioppoli inequality}
The aim of this section is to establish a Caccioppoli-type inequality for the function $v$.
Let $z_1=(x_0,t_1)$ be a fixed point. Next, let $r<\rho$ and $s>0$ be positive numbers
 such that $Q_{r,s}(z_1)\subseteq\hat Q$. We denote by
$\varphi$ a piecewise smooth function in $Q_{r,s}(z_1)$ such that
\begin{equation}\label{def zeta1}0\leq\varphi\leq1, \quad|D\varphi|<\infty\quad\text{and}\quad\varphi=0\quad\text{
on}\quad B_r(x_0)\setminus B_{(1-\sigma)r}(x_0),\end{equation}
where $\sigma\in(0,1)$. We now study the Caccioppoli inequality for $v$. Before stating this result, we remark that the time mollification defined in \eqref{timemollifierexp} cannot be used in the proof of the Caccioppoli inequality in this case, since the function $v$ is defined locally on the set $\hat Q$. Instead, we have to use a family of
convolutions $v*\xi_\epsilon$, where the smooth functions $\xi_\epsilon$ have compact supports in $(-\frac{1}{2},\frac{1}{2})$.

More precisely, the functions $\xi_\epsilon$ are defined as follows.
Let $\xi(t)$ be a nonnegative, even smooth function with compact support in $(-\frac{1}{2},\frac{1}{2})$.
Moreover, we assume that $\int_{-\frac{1}{2}}^\frac{1}{2}\xi(t)\,\mathrm{d}t=1$.
For any fixed $\epsilon\in(0,1)$ and an integrable function $f\in L^1(\mathbb{R}^{n+1})$, we define
$\xi_\epsilon=\epsilon^{-1}\xi(t/\epsilon)$ and
\begin{equation*}\begin{split}f^\epsilon=f*\xi_\epsilon=\int_{\mathbb{R}}\xi_\epsilon(\tau)
f(x,t-\tau)\,\mathrm{d}\tau.\end{split}\end{equation*}
Let us denote by $w$ the extension function of $v$, i.e., $w=v$ in $\hat Q$ and $w=0$ in $\mathbb{R}^{n+1}\setminus \hat Q$.
The next lemma is our main result in this subsection.
\begin{lemma}\label{caclemmatype2}
Let $k\geq0$ be a fixed positive number.
 There exists a positive constant
$\gamma$ depending only upon the data, such that for every piecewise smooth cutoff function
$\varphi$ and satisfying \eqref{def zeta1}, there holds
\begin{equation}\begin{split}\label{Cacinequality2}
&\esssup_{t_1-s<t<t_1}\int_{B_r(x_0)\times\{t\}}(v-k)_\pm^2\varphi^p \,\mathrm {d}x
\\&+\iint_{Q_{r,s}(z_1)}(v-k)_\pm(x,t)\varphi(x,t)^p\left(\int_{B_r(x_0)}\frac{(v-k)_\mp(y,t)^{p-1}
}{|x-y|^{n+sp}}\,\mathrm{d}y\right)\,\mathrm {d}x\mathrm {d}t\\
&+\int_{t_1-s}^{t_1}\iint_{B_r(x_0)\times B_r(x_0)}\min\left\{\varphi(x,t),\varphi(y,t)\right\}^p
\\&\qquad\qquad\qquad\qquad\times\frac{|(v-k)_\pm(x,t)-(v-k)_\pm(y,t)|^p}{|x-y|^{n+sp}}\,\mathrm {d}x\mathrm {d}y\mathrm {d}t
\\
&\leq \int_{B_r(x_0)\times\{t_1-s\}}(v-k)_\pm^2\varphi^p \,\mathrm {d}x+\gamma
\iint_{Q_{r,s}(z_1)} (v-k)_\pm^2|\partial_t\varphi^p|\,\mathrm {d}x\mathrm {d}t
\\&+\gamma\int_{t_1-s}^{t_1}\iint_{B_r(x_0)\times B_r(x_0)}\max\left\{(v-k)_\pm(x,t),(v-k)_\pm(y,t)\right\}^p
\\&\qquad\qquad\qquad\qquad \times\frac{|\varphi(x,t)-\varphi(y,t)|^p}{|x-y|^{n+sp}}
\,\mathrm {d}x\mathrm {d}y\mathrm {d}t
\\&
+\gamma \sigma^{-n-sp}\int_{t_1-s}^{t_1}\int_{\mathbb{R}^n\setminus B_r(x_0)}\int_{B_r(x_0)}
\frac{(\bar v-k^\frac{1}{q})_\pm(y,t)^{p-1}}{|y-x_0|^{n+sp}}
(v-k)_\pm(x,t)\varphi(x,t)^p\,\mathrm {d}x\mathrm {d}y\mathrm {d}t.
\end{split}\end{equation}
Here, we assume that $-T_0<t_1-s<t_1\leq0$.
\end{lemma}
\begin{proof}
We only treat the case of the truncated function $(v-k)_+$ as the other case is similar.
For a fixed $t_2\in (t_1-s,t_1)$, we set $\epsilon_0=\frac{1}{1000}\min\{-t_2,t_1-s+T_0\}$.
For $\epsilon<\epsilon_0$, we find that for $t_1-s<t<t_2$ and $x\in B_r(x_0)$, there holds \begin{equation*}w^\epsilon(x,t)=\int_{\mathbb{R}}\xi_\epsilon(t-\tau)w(x,\tau)\,\mathrm{d}\tau=\int_{\mathbb{R}}\xi_\epsilon(t-\tau)v(x,\tau)
\,\mathrm{d}\tau.\end{equation*}
This is because $\supp\xi_\epsilon\subset (-\frac{\epsilon}{2},\frac{\epsilon}{2})$ and the domain of the integration is $t-\frac{\epsilon}{2}<\tau<t+\frac{\epsilon}{2}$. This implies that $-T_0<\tau<0$; hence $w(x,\tau)=v(x,\tau)$ on
$\supp\xi_\epsilon (t-\cdot)$. For simplicity of notation, we write $v^\epsilon$ instead of $w^\epsilon$.

To proceed further, we set $\tau_1=t_1-s-\epsilon_0$ and $\tau_2=t_2+\epsilon_0$.
Next, we take a testing function $\tilde\varphi$ satisfying
\begin{equation}\label{conditionforphi}\tilde\varphi \in L^p(\tau_1,\tau_2;W_0^{s,p}(B_r(x_0)))\cap W^{1,\infty}(\tau_1,\tau_2;L^\infty(B_r(x_0))).\end{equation}
Note that for $t_1-s<t<t_2$ and $\epsilon<\epsilon_0$, there holds $t-\frac{\epsilon}{2}>t_1-s-\frac{\epsilon}{2}>\tau_1$ and
$t+\frac{\epsilon}{2}<t_2+\epsilon_0=\tau_2$. By a similar argument, we introduce a mollifier $\tilde \varphi^\epsilon$ by
$\tilde \varphi^\epsilon=\int_{\mathbb{R}}\xi_\epsilon(t-\tau)\tilde \varphi(x,\tau)
\,\mathrm{d}\tau$.
In the weak formulation \eqref{weaksolutionforv}, take the test function $\tilde \varphi^\epsilon$. This leads to
\begin{equation}\begin{split}\label{weaksolutionforvproof}
\int_{B_r(x_0)} & v(x,t_2)\tilde \varphi^\epsilon(x,t_2)\,\mathrm{d}x+\int_{t_1-s}^{t_2}\int_{B_r(x_0)}-v\partial_t
\tilde \varphi^\epsilon
\,\mathrm {d}x\mathrm {d}t
\\& -\int_{B_r(x_0)} v(x,t_1-s)\tilde \varphi^\epsilon(x,t_1-s)\,\mathrm{d}x
\\&+\int_{t_1-s}^{t_2}\iint_{B_r(x_0)\times B_r(x_0)}\tilde K(x,y,t)|v(x,t)^\frac{1}{q}-v(y,t)^\frac{1}{q}|^{p-2}(
v(x,t)^\frac{1}{q}-v(y,t)^\frac{1}{q})
\\&\qquad\qquad\times(\tilde \varphi^\epsilon(x,t)-\tilde \varphi^\epsilon(y,t))
\,\mathrm {d}x\mathrm {d}y\mathrm {d}t
\\&+2\int_{t_1-s}^{t_2}\int_{\mathbb{R}^n\setminus B_r(x_0)}\int_{B_r(x_0)}\tilde K(x,y,t)| v(x,t)^\frac{1}{q}-\bar
 v(y,t)|^{p-2}
\\&\qquad\qquad\times (
v(x,t)^\frac{1}{q}-\bar v(y,t))\tilde \varphi^\epsilon(x,t)
\,\mathrm {d}x\mathrm {d}y\mathrm {d}t=0.
\end{split}\end{equation}
By Fubini's theorem, we obtain
\begin{equation*}\begin{split}
&\int_{t_1-s}^{t_2}\int_{B_r(x_0)}-v\partial_t
\tilde \varphi^\epsilon
\,\mathrm {d}x\mathrm {d}t
\\&=\int_{t_1-s+\frac{\epsilon}{2}}^{t_2-\frac{\epsilon}{2}}\int_{B_r(x_0)}\tilde \varphi\partial_tv^\epsilon
\,\mathrm {d}x\mathrm {d}t-\int_{B_r(x_0)}v^\epsilon(x,t_2-\tfrac{\epsilon}{2})
\tilde \varphi(x,t_2-\tfrac{\epsilon}{2})
\,\mathrm {d}x
\\&+\int_{B_r(x_0)}v^\epsilon(x,t_1-s+\tfrac{\epsilon}{2})
\tilde \varphi(x,t_1-s+\tfrac{\epsilon}{2})
\,\mathrm {d}x+\Sigma(\epsilon),
\end{split}\end{equation*}
where
\begin{equation*}\begin{split}
\Sigma(\epsilon)=&-\int_{B_r(x_0)}\int_{t_1-s-\frac{\epsilon}{2}}^{t_1-s+\frac{\epsilon}{2}}\left(
\frac{1}{\epsilon}\int_{t_1-s}^{\tau+\frac{\epsilon}{2}}v(x,t)\xi\left(\frac{\tau-t}{\epsilon}\right)\,\mathrm{d}t
\right)\partial_\tau\tilde\varphi(x,\tau)\,
\mathrm{d}\tau\mathrm{d}x
\\&-\int_{B_r(x_0)}\int_{t_2-\frac{\epsilon}{2}}^{t_2+\frac{\epsilon}{2}}
\left(\frac{1}{\epsilon}\int_{\tau-\frac{\epsilon}{2}}^{t_2}v(x,t)\xi\left(\frac{\tau-t}{\epsilon}\right)\,\mathrm{d}t\right)
\partial_\tau\tilde\varphi(x,\tau)\,
\mathrm{d}\tau\mathrm{d}x.
\end{split}\end{equation*}
At this point, we take $\tilde\varphi=v_k^\epsilon\varphi^p$, where $v_k^\epsilon=(v^\epsilon-k)_+$. We assert that this choice of $\tilde\varphi$
satisfies \eqref{conditionforphi}. To prove this assertion, we only need to show that
\begin{equation}\begin{split}\label{fractionalvk}
\int_{\tau_1}^{\tau_2}\iint_{B_r(x_0)\times B_r(x_0)}\frac{|v_k^\epsilon(x,t)-v_k^\epsilon(y,t)|^p}{|x-y|^{n+sp}}\,
\mathrm{d}x\mathrm{d}y\mathrm{d}t\leq \gamma \|\bar v\|_{L^p(-T_0,0;W^{s,p}(B_r(x_0)))}^p
\end{split}\end{equation}
and
\begin{equation}\begin{split}\label{timevk}
\|\partial_tv_k^\epsilon(x,t)\|_{L^\infty(\tau_1,\tau_2;L^\infty(B_r(x_0)))}\leq \gamma(\epsilon,\text{data}).
\end{split}\end{equation}
We first consider the proof of \eqref{fractionalvk}. Noting that $|v_k^\epsilon(x,t)-v_k^\epsilon(y,t)|\leq |v^\epsilon(x,t)-v^\epsilon(y,t)|$,
we use \cite[Lemma 3.2 (iv)]{DZZ}, \eqref{barvv} and the mean value theorem to obtain
\begin{equation*}\begin{split}
&\int_{\tau_1}^{\tau_2}\iint_{B_r(x_0)\times B_r(x_0)}\frac{|v_k^\epsilon(x,t)-v_k^\epsilon(y,t)|^p}{|x-y|^{n+sp}}\,
\mathrm{d}x\mathrm{d}y\mathrm{d}t
\\&\leq \int_{-T_0}^{0}\iint_{B_r(x_0)\times B_r(x_0)}\frac{|\bar v(x,t)^q-\bar v(y,t)^q|^p}{|x-y|^{n+sp}}\,
\mathrm{d}x\mathrm{d}y\mathrm{d}t
\\&\leq\gamma \int_{-T_0}^{0}\iint_{B_r(x_0)\times B_r(x_0)}\frac{|\bar v(x,t)-\bar v(y,t)|^p}{|x-y|^{n+sp}}\,
\mathrm{d}x\mathrm{d}y\mathrm{d}t\leq \gamma \|\bar v\|_{L^p(-T_0,0;W^{s,p}(B_r(x_0)))}^p,
\end{split}\end{equation*}
which proves the inequality \eqref{fractionalvk}. To prove \eqref{timevk}, we observe that $\partial_tv_k^\epsilon=\int_{\mathbb{R}}(\partial_t\xi_\epsilon)(t-\tau)v(x,\tau)\,\mathrm{d}\tau$
and hence
\begin{equation*}\begin{split}
\|\partial_tv_k^\epsilon(x,t)\|_{L^\infty(\tau_1,\tau_2;L^\infty(B_r(x_0)))}\leq \|\partial_t\xi_\epsilon
\|_{L^1(\mathbb{R})}\|\bar v\|_{L^\infty(\hat Q)}^q
\leq \gamma(\epsilon,\text{data}),
\end{split}\end{equation*}
where we used \eqref{barvv} for the last inequality.
This proves the inequality \eqref{timevk}. Recalling that $v\in C(-T_0,0;L^2(B_\rho(x_0)))$, we have
$v^\epsilon\in C(\tau_1,\tau_2;L^2(B_r(x_0)))$; hence $\tilde\varphi\in C(\tau_1,\tau_2;L^2(B_r(x_0)))$.
According to the proof of \cite[Lemma 3.3]{BLS}, we find that $\Sigma(\epsilon)\to 0$ as $\epsilon\downarrow 0$.
Moreover, we infer that
\begin{equation*}\begin{split}
&-\int_{B_r(x_0)}v^\epsilon(x,t_2-\tfrac{\epsilon}{2})
\tilde \varphi(x,t_2-\tfrac{\epsilon}{2})
\,\mathrm {d}x
+\int_{B_r(x_0)}v^\epsilon(x,t_1-s+\tfrac{\epsilon}{2})
\tilde \varphi(x,t_1-s+\tfrac{\epsilon}{2})
\,\mathrm {d}x
\\&+\int_{B_r(x_0)}v^\epsilon(x,t_2)
\tilde \varphi(x,t_2)
\,\mathrm {d}x
-\int_{B_r(x_0)}v^\epsilon(x,t_1-s)
\tilde \varphi(x,t_1-s)
\,\mathrm {d}x\to0
\end{split}\end{equation*}
as $\epsilon\downarrow 0$. Next, we consider the fourth and fifth terms on the left hand side of \eqref{weaksolutionforvproof}.
First, we set $A=\tilde K(x,y,t)|\bar v(x,t)-\bar v(y,t)|^{p-2}(
\bar v(x,t)-\bar v(y,t))$ and introduce two quantities
\begin{equation*}\begin{split}
T_1=\int_{t_1-s}^{t_2}\iint_{B_r(x_0)\times B_r(x_0)} A\cdot([\tilde\varphi^\epsilon(x,t)-\tilde\varphi^\epsilon(y,t)]-[
\tilde\varphi(x,t)-\tilde\varphi(y,t)])\,\mathrm {d}x\mathrm {d}y\mathrm {d}t
\end{split}\end{equation*}
and
\begin{equation*}\begin{split}
T_2=\int_{t_1-s}^{t_2}\int_{\mathbb{R}^n\setminus B_r(x_0)}\int_{B_r(x_0)} A\cdot(\tilde\varphi^\epsilon(x,t)-
\tilde\varphi(x,t))\,\mathrm {d}x\mathrm {d}y\mathrm {d}t.
\end{split}\end{equation*}
Our task now is to show that $T_1$ and $T_2$ converge to zero after sending $\epsilon\downarrow 0$. In view of \eqref{fractionalvk}, we dedcue
\begin{equation*}\begin{split}
\int_{\tau_1}^{\tau_2}\iint_{B_r(x_0)\times B_r(x_0)}\frac{|\tilde\varphi^\epsilon(x,t)-\tilde\varphi^\epsilon(y,t)|^p}{|x-y|^{n+sp}}\,
\mathrm{d}x\mathrm{d}y\mathrm{d}t\leq \gamma \|\bar v\|_{L^p(-T_0,0;W^{s,p}(B_r(x_0)))}^p.
\end{split}\end{equation*}
We now proceed along the lines of the proof of \cite[Lemma 3.3]{DZZ} to conclude that $T_1\to0$ as $\epsilon\downarrow 0$.
We now turn our attention to the estimate of $T_2$. For $x\in B_{(1-\sigma)r}(x_0)$ and $y\in\mathbb{R}^n\setminus B_r(x_0)$,
we have $|y-x|\geq |y-x_0|-|x-x_0|\geq\sigma|y-x_0|$. It follows that
\begin{equation*}\begin{split}
T_2&\leq \gamma\sigma^{-n-sp}\int_{t_1-s}^{t_2}\int_{\mathbb{R}^n\setminus B_r(x_0)}\int_{B_r(x_0)} \frac{|\bar v(x,t)|^{p-1}+|\bar v(y,t)|^{p-1}}{|y-x_0|^{n+sp}}|\tilde\varphi^\epsilon(x,t)-
\tilde\varphi(x,t)|\,\mathrm {d}x\mathrm {d}y\mathrm {d}t
\\&\leq  \gamma\sigma^{-n-sp}r^{-sp}\int_{t_1-s}^{t_2}\int_{B_r(x_0)}|\bar v(x,t)|^{p-1}|\tilde\varphi^\epsilon(x,t)-
\tilde\varphi(x,t)|\,\mathrm {d}x\mathrm {d}t
\\&\quad+\gamma\sigma^{-n-sp}\int_{t_1-s}^{t_2}\left(\int_{\mathbb{R}^n\setminus B_r(x_0)} \frac{|\bar v(y,t)|^{p-1}}{|y-x_0|^{n+sp}}
\mathrm {d}y\right)
\left(
\int_{B_r(x_0)}|\tilde\varphi^\epsilon(x,t)-
\tilde\varphi(x,t)|\,\mathrm {d}x\right)\mathrm {d}t
\\&=:T_{2,1}+T_{2,2}.
\end{split}\end{equation*}
In view of \eqref{barvv}, we obtain
\begin{equation*}\begin{split}
T_{2,1}\leq \gamma\sigma^{-n-sp}r^{-sp}\|\bar v\|_{\hat Q}^{p-1}\int_{t_1-s}^{t_2}\int_{B_r(x_0)}|\tilde\varphi^\epsilon(x,t)-
\tilde\varphi(x,t)|\,\mathrm {d}x\mathrm {d}t\to 0,
\end{split}\end{equation*}
as $\epsilon\downarrow 0$. To estimate $T_{2,2}$,
we apply H\"older's inequality to obtain
\begin{equation*}\begin{split}
T_{2,2}&\leq \gamma\sigma^{-n-sp}\left[\int_{t_1-s}^{t_2}\left(\int_{\mathbb{R}^n\setminus B_r(x_0)} \frac{|\bar v(y,t)|^{p-1}}{|y-x_0|^{n+sp}}
\mathrm {d}y\right)^\frac{m}{p-1}
\mathrm {d}t\right]^\frac{p-1}{m}
\\&\times \left[\int_{t_1-s}^{t_2}\left(
\int_{B_r(x_0)}|\tilde\varphi^\epsilon(x,t)-
\tilde\varphi(x,t)|\,\mathrm {d}x\right)^\frac{m}{m-(p-1)}\mathrm {d}t\right]^\frac{m-(p-1)}{m}
\\&=:\gamma\sigma^{-n-sp}T_{2,2,1}\cdot T_{2,2,2}.
\end{split}\end{equation*}
To estimate $T_{2,2,1}$, we note that $\bar v(x,t)=(\mu^-)^{-1}u(x,t_0+(\mu^-)^{q-p+1}t)$ and this yields that
\begin{equation*}\begin{split}
T_{2,2,1}&=(\mu^-)^{(p-1)\frac{p-1-q-m}{m}}
\left[\int_{t_0+(\mu^-)^{q-p+1}(t_1-s)}^{t_0+(\mu^-)^{q-p+1}t_2}
\left(\int_{\mathbb{R}^n\setminus B_r(x_0)} \frac{|u(y,t)|^{p-1}}{|y-x_0|^{n+sp}}
\mathrm {d}y\right)^\frac{m}{p-1}
\mathrm {d}t\right]^\frac{p-1}{m}
\\&\leq (\mu^-)^{(p-1)\frac{p-1-q-m}{m}}\widetilde{\mathrm{Tail}}_m(|u|;Q)^{p-1}<+\infty.
\end{split}\end{equation*}
Finally, we consider the estimate of $T_{2,2,2}$. We use H\"older's inequality to obtain
\begin{equation*}\begin{split}
 T_{2,2,2}&\leq c(r) \left[
\int_{B_r(x_0)}\int_{t_1-s}^{t_2}|\tilde\varphi^\epsilon(x,t)-
\tilde\varphi(x,t)|^\frac{m}{m-(p-1)}\,\mathrm {d}t\mathrm {d}x\right]^\frac{m-(p-1)}{m}
\\&\leq c(r) (1+\|\xi\|_{L^1})\left[
\int_{B_r(x_0)}\int_{t_1-s}^{t_2}|v^\epsilon(x,t)-
v(x,t)|^\frac{m}{m-(p-1)}\,\mathrm {d}t\mathrm {d}x\right]^\frac{m-(p-1)}{m}.
\end{split}\end{equation*}
Since $v\in L^\infty(\hat Q)$, we have $T_{2,2,2}\to 0$ as $\epsilon\downarrow 0$. Consequently, we infer that $T_2\to0$
as $\epsilon\downarrow 0$.
At this point, we pass to the limit $\epsilon\downarrow 0$ in \eqref{weaksolutionforvproof} and
we are thus led to the following identity
\begin{equation*}\begin{split}
&\tfrac{1}{2}\int_{B_r(x_0)\times\{t_2\}}(v-k)_+^2\varphi^p \,\mathrm {d}x
 -\tfrac{1}{2}\int_{B_r(x_0)\times\{t_1-s\}}(v-k)_+^2\varphi^p \,\mathrm {d}x
\\& -\tfrac{1}{2}\int_{t_1-s}^{t_2}\int_{B_r(x_0)}(v-k)_+^2\partial_t
\varphi^p
\,\mathrm {d}x\mathrm {d}t
\\&+\int_{t_1-s}^{t_2}\iint_{B_r(x_0)\times B_r(x_0)}\tilde K(x,y,t)|v(x,t)^\frac{1}{q}-v(y,t)^\frac{1}{q}|^{p-2}(
v(x,t)^\frac{1}{q}-v(y,t)^\frac{1}{q})
\\&\qquad\qquad\times([(v-k)_+\varphi^p](x,t)-[(v-k)_+\varphi^p](y,t))
\,\mathrm {d}x\mathrm {d}y\mathrm {d}t
\\&+2\int_{t_1-s}^{t_2}\int_{\mathbb{R}^n\setminus B_r(x_0)}\int_{B_r(x_0)}\tilde K(x,y,t)| v(x,t)^\frac{1}{q}-\bar
 v(y,t)|^{p-2}
\\&\qquad\qquad\times (
v(x,t)^\frac{1}{q}-\bar v(y,t))[(v-k)_+\varphi^p](x,t)
\,\mathrm {d}x\mathrm {d}y\mathrm {d}t=0.
\end{split}\end{equation*}
At this point, the desired estimate \eqref{Cacinequality2} follows from
the proof of \cite[Lemma 7.6]{KA}
and we omit the details. The proof of the lemma is now
complete.
\end{proof}
\subsection{Auxiliary lemmas}
In this subsection, we provide some auxiliary lemmas and establish a decay estimate for the oscillation of the weak
solution.
First, we define $\Theta=\left(\frac{1}{4}\tilde \omega\right)^{2-p}$ and fix a constant $B>1$. Next, we set
\begin{equation}\label{tau0}\tau_0=\min\left\{1,\ \left(\tfrac{\xi_0^{p-q-1}A}{4^{q-1}B}\right)^\frac{1}{sp}
\left(\gamma_1^{-1}\tfrac{1-\eta}{1+\xi_0}\right)^\frac{p-2}{sp}
\right\},\end{equation}
where $\gamma_1$ is the constant in \eqref{gamma1}.
Let
$r<\tau_0\rho$ be a fixed radius.
According to \eqref{ximuequivalent}, \eqref{barvomega} and \eqref{gamma1}, we find that
$Q_r^{(B\Theta)}(x_0,0)\subset \hat Q$. The value of $B$ will be determined during the proof of Proposition \ref{oscvproposition}.
The choice of $B$ also fixes the value of $\tau_0$ via \eqref{tau0}.
We now establish a De Giorgi-type lemma for the function $v$.
\begin{lemma}\label{DeGiorgilemmaforv}
Let $r>0$,
$\delta_0\in(0,1)$,
$\xi\in(0,\frac{1}{4})$ and $\hat \Theta=(\xi\tilde\omega)^{2-p}$. Assume that $Q_r^{(\delta_0\hat\Theta)}(x_0,\bar t)\subset \hat Q$
 and
 \begin{equation}\begin{split}\label{firstTailv}
\frac{r^\frac{n\kappa}{p-1}}{(\delta_0\hat\Theta)^\frac{1}{m}}\widetilde{\mathrm{Tail}}_m((\bar v-\mu_{\bar v}^\pm)_\pm;\hat Q)\leq\xi\tilde\omega.\end{split}\end{equation}
 There exists a constant $\nu_0\in(0,1)$, depending only on the data and $\delta_0$, such that if
\begin{equation}\label{1st assumptionv}\large|\large\{(x,t)\in Q_r^{(\delta_0\hat\Theta)}(x_0,\bar t):\pm(\mu_v^\pm-v(x,t))\leq\xi\tilde\omega\large\}\large
|\leq \nu_0|Q_r^{(\delta_0\hat\Theta)}(x_0,\bar t)|,\end{equation}
then
\begin{equation}
\label{DeGiorgi1v}\pm(\mu_v^\pm-v(x,t))\geq\tfrac{1}{2}\xi\tilde\omega\qquad\text{for}\ \ \text{a.e.}\ \ (x,t)\in Q_{\frac{1}{2}r}
^{(\delta_0\hat\Theta)}(x_0,\bar t).\end{equation}
\end{lemma}
\begin{proof}
We give the proof only for the function $-(\mu_v^--v(x,t))$; the other case is left to the
reader.
For $j=0,1,2,\cdots$, we define the sequences $k_j=\mu_v^-+2^{-1}\xi\tilde\omega+2^{-1-j}\xi\tilde\omega$, $\tilde k_j=\frac{1}{2}(k_j+k_{j+1})$,
$r_j=2^{-1}r+2^{-j-1}r$, $\tilde r_j=\frac{1}{2}(r_j+r_{j+1})$ and $\hat r_j=\frac{1}{2}(\tilde r_j+r_{j+1})$.
Moreover, we take a cutoff function
$0\leq\varphi_j\leq1$, such that $\varphi_j=1$ in $Q_{r_{j+1}}^{(\delta_0\hat\Theta)}(x_0,\bar t)$, $\varphi_j=0$ on $\partial_PQ_{\hat r_j}^{(\delta_0\hat\Theta)}(x_0,\bar t)$,
\begin{equation*}\begin{split}
|D\varphi_j|\leq 2^{j+3}r^{-1}\qquad\text{and}\qquad
|\partial_t\varphi_j|\leq c\delta_0^{-1}\hat\Theta^{-1}2^jr^{-sp}.
\end{split}\end{equation*}
With the choice of $\varphi=\varphi_j$ in Lemma \ref{caclemmatype2}, we first use the Caccioppoli inequality
\eqref{Cacinequality2}
to the truncated functions $(v-k_j)_-$ over the cylinder $Q_{\tilde r_j}^{(\hat\Theta)}(x_0,\bar t)$ and then rewrite
 the Caccioppoli inequality in terms of the
new functions $v_1(x,t)=v(x_0+x,\delta_0\hat\Theta t+\bar t)$ and $\bar v_1(x,t)=\bar v(x_0+x,\delta_0\hat\Theta t+\bar t)$. This leads to
\begin{equation*}\begin{split}
T_{0,1}+T_{0,2}&:=\hat\Theta^{-1}\esssup_{-r_{j+1}^{sp}<t<0}\int_{B_{r_{j+1}}\times\{t\}}(v_1-k_j)_-^2 \,\mathrm {d}x
\\
&+\int_{-r_{j+1}^{sp}}^{0}\iint_{B_{r_{j+1}}\times B_{r_{j+1}}}
\frac{|(v_1-k_j)_-(x,t)-(v_1-k_j)_-(y,t)|^p}{|x-y|^{n+sp}}\,\mathrm {d}x\mathrm {d}y\mathrm {d}t
\\
&\leq \gamma 2^j\hat\Theta^{-1}r^{-sp}
\iint_{Q_{\tilde r_j}} (v_1-k_j)_-^2 \,\mathrm {d}x\mathrm {d}t
+\gamma\frac{2^{jp}}{r^{sp}}\iint_{Q_{\tilde  r_j}}(v_1-k_j)_-^p
\,\mathrm {d}x\mathrm {d}t
\\&
+\gamma 2^{(n+sp)j}\int_{-\tilde r_j^{sp}}^0\int_{\mathbb{R}^n\setminus B_{\tilde r_j}}\int_{B_{\tilde r_j}}
\frac{(\bar v_1-k_j^\frac{1}{q})_-(y,t)^{p-1}}{|y|^{n+sp}}(v_1-k_j)_-(x,t)\,\mathrm {d}x\mathrm {d}y\mathrm {d}t
\\&=:T_1+T_2+T_3,
\end{split}\end{equation*}
since $k_j\geq0$.
Next, for  $j=0,1,2,\cdots$, we define
\begin{equation*}\begin{split}
Y_j=\frac{1}{|Q_{r_j}|}|\{(x,t)\in Q_{r_j}:v_1(x,t)\leq k_j\}|
\end{split}\end{equation*}
and
\begin{equation*}\begin{split}
Z_j=\left(\fint_{-r_j^{sp}}^0\left(\frac{1}{|B_j|}|A_-(v_1;0,r_j,k_j)|\right)^\frac{m}{m-(p-1)}\,\mathrm{d}t\right)^\frac{m-(p-1)}{m(1+\kappa)}.
\end{split}\end{equation*}
In view of $(v_1-k_j)_-\leq \xi\tilde\omega$ in $Q_r$, we obtain $T_1+T_2\leq\gamma 2^{jp}r^{-sp}(\xi\tilde\omega)^p|Q_{r_j}|Y_j$.
Next, we consider the estimate for $T_3$. To this end, we observe that
\begin{equation*}\begin{split}
T_3&\leq \gamma(\xi\tilde\omega) 2^{(n+sp)j}\int_{-\tilde r_j^{sp}}^0\int_{B_\rho\setminus B_{\tilde r_j}}\int_{B_{\tilde r_j}}
\frac{(\bar v_1-k_j^\frac{1}{q})_-(y,t)^{p-1}}{|y|^{n+sp}}\chi_{\{v_1\leq k_j\}}(x,t)\,\mathrm {d}x\mathrm {d}y\mathrm {d}t
\\&\quad+\gamma(\xi\tilde\omega) 2^{(n+sp)j}\int_{-\tilde r_j^{sp}}^0\int_{\mathbb{R}^n\setminus B_\rho}\int_{B_{\tilde r_j}}
\frac{(\bar v_1-k_j^\frac{1}{q})_-(y,t)^{p-1}}{|y|^{n+sp}}\chi_{\{v_1\leq k_j\}}(x,t)\,\mathrm {d}x\mathrm {d}y\mathrm {d}t
\\&=:T_{3,1}+T_{3,2}.
\end{split}\end{equation*}
To estimate $T_{3,1}$, we observe that $\bar v_1=v_1^\frac{1}{q}$ in $B_\rho\times(-r^{sp},0)$. Moreover, we infer from \eqref{infsupbarv}
and \eqref{barvomega} that $k_j\geq \mu_v^-=(\mu_{\bar v}^-)^q=1$ and $k_j\leq \mu_v^-+\xi\tilde\omega<(\mu_{\bar v}^-)^q +2\omega_v<1+2\gamma\omega_{\bar v}\leq 1+2\gamma \xi_0^{-1}$.
For $(y,t)\in B_\rho\times(-r^{sp},0)$, we deduce from \eqref{barvv} that $1\leq |k_j|+|v_1|\leq 1+2\gamma \xi_0^{-1}+\left(\frac{1+\xi_0}{\xi_0}\right)^q$; hence
\begin{equation*}\begin{split}
(\bar v_1-k_j^\frac{1}{q})_-(y,t)=(v_1(y,t)^\frac{1}{q}-k_j^\frac{1}{q})_-\leq \gamma (|k_j|+|v_1|)^{\frac{1}{q}-1}(v_1-k_j)_-\leq\gamma\xi\tilde\omega.
\end{split}\end{equation*}
It follows that $T_{3,1}\leq\gamma 2^{j(n+sp)}r^{-sp}(\xi\tilde\omega)^p|Q_{r_j}|Y_j$. We now turn our attention to the estimate of $T_{3,2}$.
To this end, we use H\"older's inequality to obtain
\begin{equation*}\begin{split}
T_{3,2}&\leq \gamma(\xi\tilde\omega) 2^{(n+sp)j}r^{n(1+\kappa)}Z_j^{1+\kappa}\left[\int_{-\tilde r_j^{sp}}^0\left(\int_{\mathbb{R}^n\setminus B_\rho}
\frac{(\bar v_1-k_j^\frac{1}{q})_-(y,t)^{p-1}}{|y|^{n+sp}}\,\mathrm {d}y\right)^\frac{m}{p-1}\mathrm {d}t\right]^\frac{p-1}{m}
\\&=:\gamma(\xi\tilde\omega) 2^{(n+sp)j}r^{n(1+\kappa)}Z_j^{1+\kappa}T_{3,2,1}.
\end{split}\end{equation*}
Transforming
back to the original function $\bar v$, we obtain
\begin{equation*}\begin{split}
T_{3,2,1}=\hat \Theta^{-\frac{p-1}{m}}\left[\int_{\bar t-\hat\Theta\tilde r_j^{sp}}^{\bar t}\left(\int_{\mathbb{R}^n\setminus B_\rho(x_0)}
\frac{(\bar v-k_j^\frac{1}{q})_-(y,t)^{p-1}}{|y-x_0|^{n+sp}}\,\mathrm {d}y\right)^\frac{m}{p-1}\mathrm {d}t\right]^\frac{p-1}{m}
\end{split}\end{equation*}
Next, we claim that
\begin{equation}\begin{split}\label{claim2}
(\bar v-k_j^\frac{1}{q})_-(y,t)\leq \gamma\xi\tilde\omega+(\mu_{\bar v}^--\bar v(y,t))_+.
\end{split}\end{equation}
To prove \eqref{claim2}, we infer from $\mu_v^-=(\mu_{\bar v}^-)^q$ that
\begin{equation*}\begin{split}
&(k_j^\frac{1}{q}-\bar v)_+(y,t)=(k_j^\frac{1}{q}-(\mu_v^-)^\frac{1}{q}+\mu_{\bar v}^--\bar v(y,t))_+
\\&\leq (k_j^\frac{1}{q}-(\mu_v^-)^\frac{1}{q})_++(\mu_{\bar v}^--\bar v(y,t))_+.
\end{split}\end{equation*}
Since $1\leq k_j\leq 1+2\gamma \xi_0^{-1}$, we use \eqref{barvv} to obtain $1\leq |k_j|+|\mu_v^-|\leq 2+2\gamma \xi_0^{-1}$ and hence
\begin{equation*}\begin{split}
(k_j^\frac{1}{q}-(\mu_v^-)^\frac{1}{q})_+\leq \gamma(|k_j|+|\mu_v^-|)^{\frac{1}{q}-1}(k_j-\mu_v^-)_+\leq \gamma\xi\tilde\omega,
\end{split}\end{equation*}
which proves \eqref{claim2}. In view of \eqref{claim2}, we deduce
\begin{equation*}\begin{split}
T_{3,2,1}&\leq\hat \Theta^{-\frac{p-1}{m}}\left[\int_{\bar t-\hat\Theta\tilde r_j^{sp}}^{\bar t}\left(\int_{\mathbb{R}^n\setminus B_\rho(x_0)}
\frac{(\bar v-\mu_{\bar v}^-)_-(y,t)^{p-1}}{|y-x_0|^{n+sp}}\,\mathrm {d}y\right)^\frac{m}{p-1}\mathrm {d}t\right]^\frac{p-1}{m}
\\&\quad+\gamma(\xi\tilde\omega)^{p-1}r^{sp\frac{p-1}{m}}\rho^{-sp}
\\&\leq\Theta^{-\frac{p-1}{m}}\widetilde{\mathrm{Tail}}_m((\bar v-\mu_{\bar v}^-)_-;\hat Q)^{p-1}+\gamma(\xi\tilde\omega)^{p-1}r^{-n\kappa}
\\&\leq \gamma(\xi\tilde\omega)^{p-1}r^{-n\kappa},
\end{split}\end{equation*}
where we used \eqref{firstTailv} in the last step.
Combining these inequalities, we arrive at
\begin{equation*}\begin{split}
T_{0,1}+T_{0,2}\leq \gamma 2^{j(n+sp)}r^{-sp}(\xi\tilde\omega)^p|Q_{r_j}|(Y_j+Z_j^{1+\kappa}).
\end{split}
\end{equation*}
At this stage, we now proceed along the lines of the proof of \cite[Lemma 4.1]{BK} to conclude that
\begin{equation}\begin{split}\label{Yj+1v}
Y_{j+1}\leq \gamma 2^{j(2p+n)}\left(Y_j^{1+\frac{sp}{n+sp}}+Y_j^\frac{sp}{n+sp}Z_j^{1+\kappa}\right)
\end{split}
\end{equation}
and
\begin{equation}\begin{split}\label{Zj+1v}
Z_{j+1}&\leq \gamma 2^{(n+2p)j}\left(Y_j+Z_j^{1+\kappa}\right),
\end{split}
\end{equation}
where the constant $\gamma$ depends only on the data and $\delta_0$.
At this point, we set
\begin{equation}\begin{split}\label{nu0v}
\nu_0=\nu_0(\delta_0)=(4\gamma)^{-\frac{m(1+\kappa)}{(m-(p-1))\zeta}}2^{-(n+2p)\frac{m(1+\kappa)}{(m-(p-1))\zeta^2}},
\end{split}
\end{equation}
where $\zeta=\min\left\{\kappa,\frac{sp}{n+sp}\right\}$.
Using the lemma on fast geometric convergence of sequences again,
we infer from \eqref{Yj+1v} and \eqref{Zj+1v} that $Y_n\to0$ as $n\to\infty$. Transforming
back to the original function $v$, we obtain the desired estimate \eqref{DeGiorgi1v}.
This
completes the proof of Lemma \ref{DeGiorgilemmaforv}.
\end{proof}
The proof of Lemma \ref{DeGiorgilemmaforv} also provides a method to address the nonlocal tail estimates for the function $\bar v$.
With the help of this method, we now proceed along the lines of the proofs of \cite[Lemma 4.2-4.4]{BK}
to obtain the following three lemmas.
\begin{lemma}\label{lemmaDeGiorgiforv2}
Let $\xi\in(0,\frac{1}{4})$ be a fixed constant and set $\hat\Theta=(\xi\tilde\omega)^{2-p}$.
Let $t_1\in(-B\Theta r^{sp},0)$ be a fixed time level. Assume that
\begin{equation}\label{2st assumptionv}
v(\cdot,t_1)\geq \mu_v^-+\xi\tilde\omega\quad\text{a.e.}\quad\text{in}\quad B_r(x_0).\end{equation}
 There exists a constant $\nu_1\in(0,1)$, depending only on the data, such that if
  \begin{equation}\begin{split}\label{secondTailv}
\frac{r^\frac{n\kappa}{p-1}}{(\nu\hat\Theta)^\frac{1}{m}}\widetilde{\mathrm{Tail}}_m((\bar v-\mu_{\bar v}^-)_-;\hat Q)\leq\xi\tilde\omega\end{split}\end{equation}
and $B_r(x_0)\times(t_1,t_1+\nu\hat\Theta r^{sp})\subset \hat Q$ hold for a fixed $\nu<\nu_1$,
then
\begin{equation}
\label{DeGiorgi2v}v(x,t)>\mu_v^-+\tfrac{1}{2}\xi\tilde\omega\qquad\text{for}\ \ \text{a.e.}\ \ (x,t)\in B_{\frac{1}{2}r}(x_0)\times(t_1,t_1+\nu\hat\Theta r^{sp}).\end{equation}
\end{lemma}
\begin{lemma}\label{measureforv}
Let $\alpha\in(0,1)$ be given and let $t_*\in(-B\Theta r^{sp},0)$ be a fixed time level.
Assume that
\begin{equation}\label{t*bartvassumption}\large|\large\{v(\cdot,t_*)\leq\mu_v^+-\xi\tilde\omega\large
\}\cap B_r(x_0)\large|\geq\alpha|B_r|\end{equation}
 holds for a fixed $\xi\in(0,\frac{1}{4})$.
Then, there exist constants $\epsilon$ and $\delta\in (0,1)$ depending only upon the data and $\alpha$, such that if $B_r(x_0)\times (t_*,t_*+\delta\hat\Theta r^{sp})\subset\hat Q$ and
 \begin{equation}\begin{split}\label{measureTailforv}
\frac{r^\frac{n\kappa}{p-1}}{(\delta\hat\Theta)^\frac{1}{m}}\widetilde{\mathrm{Tail}}_m((\bar v-\mu_{\bar v}^+)_+;\hat Q)\leq\xi\tilde\omega\end{split}\end{equation}
hold,
then
\begin{equation}\label{t*bartforv}\large|\large\{v(\cdot,t)\leq\mu_v^+-\epsilon\xi\tilde\omega\large
\}\cap B_r(x_0)\large|\geq\tfrac{1}{2}\alpha|B_r|\end{equation}
holds for all $t\in (t_*,t_*+\delta\hat\Theta r^{sp})$. Here, $\hat\Theta=(\xi\tilde\omega)^{2-p}$.
Moreover, the functional dependence of $\delta$ on the measure-theoretical parameter
$\alpha$ is of the form
\begin{equation}\begin{split}\label{deltaalphadenpendence}
\delta=\delta(\alpha)=\frac{1}{c(8n)^{n+p+1}}\alpha^{n+p+1}
\end{split}\end{equation}
for $c>1$ depending only on the data.
\end{lemma}
\begin{lemma}\label{measureshrinkinglemmav}
Let $\xi$ and $\sigma_*\in(0,1)$ be given constants, and let $\tilde\Theta=(\sigma_*\xi\tilde\omega)^{2-p}$.
Assume that $Q_{2r}^{(\tilde\Theta)}(x_0,0)\subset \hat Q$ and
 \begin{equation}\begin{split}\label{measurevforvassumption}
\large|\large\{v(\cdot,t)\leq\mu_v^+-\xi\tilde\omega\large
\}\cap B_r(x_0)\large|\geq \alpha|B_r|\end{split}\end{equation}
holds for $a.e.$ $t\in(-\tilde\Theta r^{sp},0)$.
Moreover, suppose that
 \begin{equation}\begin{split}\label{measureTail1v}
\frac{r^\frac{n\kappa}{p-1}}{
\tilde\Theta^\frac{1}{m}}\widetilde{\mathrm{Tail}}_m((\bar v-\mu_{\bar v}^+)_+;\hat Q)\leq \sigma_*\xi\tilde\omega.\end{split}\end{equation}
Then, there exists a constant $\gamma$ that can be determined a priori only in terms of the data, such that
\begin{equation}\label{measureQv}\large|\large\{v\geq\mu_v^+-\tfrac{1}{4}\xi\sigma_*\tilde\omega\large
\}\cap Q_{r}^{(\tilde\Theta)}(x_0,0)\large|\leq\gamma \sigma_*^{p-1}\alpha^{-1}|Q_{r}^{(\tilde\Theta)}|.\end{equation}
\end{lemma}
The proofs of Lemma \ref{lemmaDeGiorgiforv2}-\ref{measureshrinkinglemmav} are omitted. It is natural to relate the tail conditions for $\bar v$ to the initial condition \eqref{omegaxi0}.
This motivates us to establish the following result.
\begin{lemma}\label{tailsatisfylemma} Let $\tilde \nu_*\in(0,1)$ be a fixed constant. Then, there exists a constant $\sigma^\prime>0$ depending only upon the data
and $\tilde \nu_*$, such that
if $r=\sigma^\prime\rho$, then
\begin{equation}\begin{split}\label{tailsatisfy}
\frac{r^\frac{n\kappa}{p-1}}{
\Theta^\frac{1}{m}}\widetilde{\mathrm{Tail}}_m((\bar v-\mu_{\bar v}^\pm)_\pm;\hat Q)\leq \tilde\nu_*\tilde\omega.\end{split}\end{equation}
\end{lemma}
\begin{proof} We only treat the case of the truncated function $(\bar v-\mu_{\bar v}^-)_-$ as the other case is similar.
In view of \eqref{omegavc1c2}, we have $\tilde\omega\leq 2\omega_v\leq 2C_2$; hence $\Theta^{-\frac{1}{m}}=\left(\frac{1}{4}\tilde \omega\right)^{\frac{p-2}{m}}\leq \left(\frac{1}{2}C_2\right)^{\frac{p-2}{m}}$. Recalling that $\bar v(x,t)=(\mu^-)^{-1}
u(x,t_0+(\mu^-)^{q-p+1}t)$ and $T_0=(\mu^-)^{p-q-1}A\theta\rho^{sp}$, we obtain
\begin{equation*}\begin{split}
&\frac{r^\frac{n\kappa}{p-1}}{\Theta^\frac{1}{m}}\widetilde{\mathrm{Tail}}_m((\bar v-\mu_{\bar v}^\pm)_\pm;\hat Q)
\\&\leq
\left(\tfrac{1}{2}C_2\right)^{\frac{p-2}{m}}r^\frac{n\kappa}{p-1}
\left[\int_{-T_0}^0\left(\int_{\mathbb{R}^n\setminus B_\rho(x_0)}
\frac{(\bar v-\mu_{\bar v}^-)_-(y,t)^{p-1}}{|y-x_0|^{n+sp}}\,\mathrm {d}y\right)^\frac{m}{p-1}\mathrm {d}t\right]^\frac{1}{m}
\\&=
\left(\tfrac{1}{2}C_2\right)^{\frac{p-2}{m}}r^\frac{n\kappa}{p-1}(\mu^-)^{\frac{p-q-1}{m}-1}
\left[\int_{t_0-A\theta\rho^{sp}}^{t_0}\left(\int_{\mathbb{R}^n\setminus B_\rho(x_0)}
\frac{(u-\mu^-)_-(y,t)^{p-1}}{|y-x_0|^{n+sp}}\,\mathrm {d}y\right)^\frac{m}{p-1}\mathrm {d}t\right]^\frac{1}{m}
\\&=:\left(\tfrac{1}{2}C_2\right)^{\frac{p-2}{m}}r^\frac{n\kappa}{p-1}(\mu^-)^{\frac{p-q-1}{m}-1}T_1,
\end{split}\end{equation*}
where the constant $C_2$ depends only upon the data.
Moreover, we observe that $\mu^--\mu_{-1}^-\leq \mu_{-1}^+-\mu_{-1}^-\leq \omega_{-1}=(1-\eta)^{-1}\omega$
and it follows that
\begin{equation*}\begin{split}
T_1&\leq \left[\int_{t_0-A\theta\rho^{sp}}^{t_0}\left(\int_{\mathbb{R}^n\setminus B_{\rho_{-1}}(x_0)}
\frac{(u-\mu^-)_-(y,t)^{p-1}}{|y-x_0|^{n+sp}}\,\mathrm {d}y\right)^\frac{m}{p-1}\mathrm {d}t\right]^\frac{1}{m}
\\&\quad+\left[\int_{t_0-A\theta\rho^{sp}}^{t_0}\left(\int_{B_{\rho_{-1}}(x_0)\setminus B_\rho(x_0)}
\frac{(u-\mu^-)_-(y,t)^{p-1}}{|y-x_0|^{n+sp}}\,\mathrm {d}y\right)^\frac{m}{p-1}\mathrm {d}t\right]^\frac{1}{m}
\\&\leq \gamma(A\theta\rho^{sp})^\frac{1}{m}\rho^{-\frac{sp}{p-1}}\omega+\widetilde{\mathrm{Tail}}_m((u-\mu_{-1}^-)_-;Q_{-1})
\end{split}\end{equation*}
Taking into account \eqref{i0AltTailall}, \eqref{ximuequivalent},
$\theta=\left(\frac{1}{4}\omega\right)^{q+1-p}$, $\tilde\rho=\lambda^{-1}\rho$ and $\omega_{\bar v}\leq \gamma\omega_{v}\leq \gamma\tilde\omega$,
 we deduce
\begin{equation*}\begin{split}
\frac{r^\frac{n\kappa}{p-1}}{\Theta^\frac{1}{m}}&\widetilde{\mathrm{Tail}}_m((\bar v-\mu_{\bar v}^\pm)_\pm;\hat Q)
\\&\leq
\gamma\left( \frac{r}{\rho}\right)^\frac{n\kappa}{p-1}\left(\frac{\mu^-}{\omega}\right)^{\frac{p-q-1}{m}}\frac{\omega}{\mu^-}
\leq \gamma (\sigma^\prime)^\frac{n\kappa}{p-1}\omega_{\bar v}\leq \gamma (\sigma^\prime)^\frac{n\kappa}{p-1}\tilde \omega.
\end{split}\end{equation*}
With the choice of $\sigma^\prime=(\gamma^{-1}\tilde\nu_*)^\frac{p-1}{n\kappa}$, we get \eqref{tailsatisfy},
which completes the proof of the lemma.
\end{proof}
With the help of the proceeding lemmas, we can now establish an oscillation decay estimate for $u$, and the following proposition is our
main result in this subsection.
\begin{proposition}\label{oscvproposition}
There exist constants $\tilde\nu_*<1$ and $\frac{1}{2}<\tilde\eta<1$ depending only upon the data, such that
if
\begin{equation}\begin{split}\label{tailsatisfyvv}
\frac{r^\frac{n\kappa}{p-1}}{
\Theta^\frac{1}{m}}\widetilde{\mathrm{Tail}}_m((\bar v-\mu_{\bar v}^\pm)_\pm;\hat Q)\leq \tilde\nu_*\tilde\omega,\end{split}\end{equation}
then, we have
\begin{equation}\begin{split}\label{oscforv}
\essosc_{\hat Q_0} v\leq (1-\tilde\eta)\tilde\omega,
 \end{split}\end{equation}
 where $\hat Q_0=Q_{\frac{1}{4}r}^{(\Theta)}(x_0,0)$.
\end{proposition}
\begin{proof}
The proof of \eqref{oscforv}
will be divided into three steps.

Step 1: Once again, we consider two complementary cases.
Let $\nu_0>0$ be the constant claimed by Lemma \ref{DeGiorgilemmaforv} with $\delta_0=1$. Recalling that $\omega_v\leq\tilde\omega\leq 2\omega_v$,
we see that either
  \begin{itemize}
 \item[$\bullet$]
 \textbf{The first alternative}. There exists $-(B-1)\Theta r^{sp}\leq \bar t\leq 0$ such that
\begin{equation}\label{1stv}\left|\left\{(x,t)\in Q_r^{(\Theta)}(x_0,\bar t)
:v\leq \mu_v^-+\tfrac{1}{4}\tilde\omega\right\}\right|\leq \nu_0|Q_r^{(\Theta)}|\end{equation}
\end{itemize}
or this does not hold. If \eqref{1stv} does not hold, we find that the
 following second alternative holds.
  \begin{itemize}
 \item[$\bullet$]
 \textbf{The second alternative}. For any $-(B-1)\Theta r^{sp}\leq \bar t\leq 0$, there holds
\begin{equation}\label{2ndv}\left|\left\{(x,t)\in Q_r^{(\Theta)}(x_0,\bar t)
:v\leq\mu_v^-+\tfrac{1}{4}\tilde\omega\right\}\right|> \nu_0|Q_r^{(\Theta)}|.\end{equation}
\end{itemize}

Step 2: \emph{Analysis for the first alternative.}
First, we assume that $\tilde\nu_*<\frac{1}{4}$ and
hence \eqref{firstTailv} holds with $\xi=\frac{1}{4}$ and $\delta_0=1$. We now apply Lemma \ref{DeGiorgilemmaforv} with $\delta_0=1$ to conclude that for any $\xi_1<\frac{1}{8}$,
there holds
\begin{equation}
\label{initial1v}v(x,\tilde t)\geq\mu_v^-+\xi_1\tilde\omega\qquad\text{for}\ \ \text{a.e.}\ \ x\in B_{\frac{1}{2}r}(x_0),\end{equation}
where $\tilde t=\bar t-\Theta (\frac{1}{2}r)^{sp}$. Moreover, we will apply Lemma \ref{lemmaDeGiorgi2} with $\xi=\xi_1$.
To this end, we take $\xi_1=\min\{\frac{1}{8},\frac{1}{4}\left(2^{-sp}B^{-1}\nu_1\right)^\frac{1}{p-2}\}$,
where $\nu_1$ is the constant defined in Lemma \ref{lemmaDeGiorgiforv2}. It is easy to check that
\begin{equation}\label{tildet1v}\tilde t+\nu_1(\xi_1\tilde\omega)^{2-p}\left(\tfrac{1}{2}r\right)^{sp}\geq 0,\end{equation}
since $\tilde t\geq -B\Theta r^{sp}$.
Furthermore, we define $\nu_1^\prime=4^{p-2}\xi_1^{p-2}$ and observe that $\nu_1^\prime<\nu_1$. It follows that
\begin{equation}\label{tildet11v}\tilde t+\nu_1^\prime(\xi_1\tilde\omega)^{2-p}\left(\tfrac{1}{2}r\right)^{sp}\leq 0,\end{equation}
since $\tilde t\leq-(\frac{1}{4}\tilde\omega)^{2-p}(\frac{1}{2}r)^{sp}$. In view of \eqref{tildet1v}
and \eqref{tildet11v}, we find that there exists a constant $\nu\in [\nu_1^\prime,\nu_1]$, such that
$\tilde t+\nu\hat\Theta\left(\frac{1}{2}r\right)^{sp}= 0$, where $\hat\Theta=(\xi_1\tilde\omega)^{2-p}$.
At this point, we assume that $\tilde\nu_1\leq \xi_1$ and
we conclude from \eqref{tailsatisfyvv} that
\begin{equation*}\begin{split}
\frac{(\frac{1}{2}r)^\frac{n\kappa}{p-1}}{\tilde\omega^\frac{2-p}{m}}&\widetilde{\mathrm{Tail}}_m((\bar v-\mu_{\bar v}^-)_-;\hat Q)<4^\frac{p-2}{m}\xi_1\tilde\omega
\\&= (\nu_1^\prime)^\frac{1}{m}\xi_1^{1-\frac{p-2}{m}}
\tilde\omega
\leq\nu^\frac{1}{m}\xi_1^{1-\frac{p-2}{m}}
\tilde\omega,
\end{split}\end{equation*}
which in turn yields that \eqref{secondTailv} holds with $\xi=\xi_1$ and $\hat\Theta=(\xi_1\tilde\omega)^{2-p}$.
We now use Lemma \ref{lemmaDeGiorgiforv2} with $r$ replaced by $\frac{1}{2}r$ to conclude that
\begin{equation}
\label{DeGiorgiprov}
v(x,t)>\mu_v^-+\tfrac{1}{2}\xi_1\tilde\omega\qquad\text{for}\ \ \text{a.e.}\ \ (x,t)\in B_{\frac{1}{4}r}(x_0)\times(\tilde t,0),\end{equation}
which proves the inequality \eqref{oscforv} with $\tilde\eta\leq\frac{1}{2}\xi_1$.

Step 3: \emph{Analysis for the second alternative.} For any $\bar t\in(-(B-1)\Theta r^{ps},0)$, there exists a time level $\hat t\in
[\bar t-\Theta r^{sp},\bar t-\frac{1}{2}\nu_0\Theta r^{sp}]$, such that
\begin{equation*}\large|\large\{v(\cdot,\hat t)\leq\mu_v^+-\xi_2\tilde\omega\large
\}\cap B_r(x_0)\large|>\tfrac{1}{2}\nu_0|B_r|\end{equation*}
holds for any $\xi_2<\frac{1}{4}$. At this point,
we will apply Lemma \ref{measureforv} with $\alpha=\frac{1}{2}\nu_0$.
Let $\delta_1=\delta(\frac{1}{2}\nu_0)\in(0,1)$ be the constant defined in \eqref{deltaalphadenpendence}. Since
$\frac{1}{2}\nu_0\Theta r^{sp}\leq \bar t-\hat t\leq \Theta r^{sp}$, there exists a constant $\xi_2\in[\frac{1}{4}\delta_1^\frac{1}{p-2},
\frac{1}{4}(2\nu_0^{-1}\delta_1)^\frac{1}{p-2}]$ such that $\delta_1(\xi_2\tilde\omega)^{2-p} r^{sp}=\bar t-\hat t$.
Next. we claim that the upper bound of this interval is less than $\frac{1}{4}$. In view of \eqref{deltaalphadenpendence}, we find that
\begin{equation*}\frac{1}{4}(2\nu_0^{-1}\delta_1)^\frac{1}{p-2}=\frac{1}{4}\left(\frac{2^{-n-p}}{c(8n)^{n+p+1}}
\nu_0^{n+p}
\right)^\frac{1}{p-2}<\frac{1}{4},\end{equation*}
which proves the claim and the choice of $\xi_2$ is justified. To proceed further, we assume that $\tilde\nu_*\leq \delta_1^\frac{1}{p-2}$.
According to \eqref{tailsatisfyvv}, we find that \eqref{measureTailforv} holds with $\hat\Theta=(\xi_2\tilde\omega)^{2-p}$ and Lemma
\ref{measureforv} is at our disposal. We infer from \eqref{t*bartforv} and $\xi_2\geq \frac{1}{4}\delta_1^\frac{1}{p-2}$ that
\begin{equation}\label{t*bartforvv}\large|\large\{v(\cdot,t)\leq\mu_v^+-\tfrac{1}{4}\epsilon\delta_1^\frac{1}{p-2}\tilde\omega\large
\}\cap B_r(x_0)\large|\geq\tfrac{1}{4}\nu_0|B_r|\end{equation}
for all $t\in[\hat t,\bar t]$.
Here, $\epsilon\in(0,1)$
is the constant claimed by Lemma \ref{measureforv}
and depends only upon the data. Recalling that $\bar t$ is arbitrary in $[-(B-1)\Theta r^{sp}, 0]$, we infer that \eqref{t*bartforvv}
holds for any $t\in [-(B-1)\Theta r^{sp}, 0]$.

At this stage,
we will apply Lemma \ref{measureshrinkinglemmav} with $\alpha=\frac{1}{4}\nu_0$,
$\xi=\frac{1}{4}\epsilon\delta_1^\frac{1}{p-2}$
and $\tilde\Theta=(\frac{1}{4}\epsilon\delta_1^\frac{1}{p-2}\sigma_*\tilde\omega)^{2-p}$.
Let $\sigma_*>0$ be a constant that will be determined later. Furthermore, we assume that $\tilde\nu_*\leq (\frac{1}{4}\sigma_*\epsilon\delta_1^\frac{1}{p-2})^\frac{m-(p-2)}{m}$. Then, we obtain \eqref{measureTail1v}
for such a choice of $\tilde\nu_*$. We are now in a position to use Lemma \ref{measureshrinkinglemmav}.
For the fixed $\sigma_*>0$, we infer from \eqref{measureQv} that
\begin{equation}\label{measureQvv}\large|\large\{v\geq\mu_v^+-\tfrac{1}{16}\epsilon\delta_1^\frac{1}{p-2}\sigma_*\tilde\omega\large
\}\cap Q_{r}^{(\tilde\Theta)}(x_0,0)\large|\leq\gamma \sigma_*^{p-1}\nu_0^{-1}|Q_{r}^{(\tilde\Theta)}|,\end{equation}
provided that $Q_{2r}^{(\tilde\Theta)}(x_0,0)\subseteq B_{2r}(x_0)\times [-(B-1)\Theta r^{sp}, 0]$. For this inclusion to hold, it is enough to assume $B=(\epsilon\delta_1^\frac{1}{p-2}\sigma_*)^{2-p}+1$.

Let $\tilde\nu_0=\nu_0(4^{2-p})$ be the constant claimed by Lemma \ref{DeGiorgilemmaforv} with $\delta_0=4^{2-p}$ and $\bar t=0$.
At this point, we set $\sigma_*=(\gamma^{-1}\nu_0\tilde\nu_0)^\frac{1}{p-1}$, where $\gamma$ is the constant in \eqref{measureQvv}.
This also fixes the parameter
\begin{equation}\label{defB}
B=(\epsilon\delta_1^\frac{1}{p-2}(\gamma^{-1}\nu_0\tilde\nu_0)^\frac{1}{p-1})^{2-p}+1.
\end{equation}
Moreover, we assume that $\tilde \nu_*\leq 4^{(2-p)\frac{1}{m}}(\frac{1}{16}\epsilon \delta_1^\frac{1}{p-2}\sigma_*)^{1-\frac{p-2}{m}}$ and
this assumption guarantees that \eqref{firstTailv} holds with $\xi=\frac{1}{16}\epsilon \delta_1^\frac{1}{p-2}\sigma_*$.
An application of Lemma \ref{DeGiorgilemmaforv} gives us that
\begin{equation*}
v(x,t)<\mu_v^+-\tfrac{1}{32}\epsilon \delta_1^\frac{1}{p-2}\sigma_*\tilde\omega\qquad\text{for}\ \ \text{a.e.}\ \ (x,t)\in Q_{\frac{1}{2}r}
^{(\tilde\Theta)}(x_0,0).\end{equation*}
This proves the inequality \eqref{oscforv}. Finally,
we summarize the precise values of the parameters $\tilde\nu_*$ and $\tilde \eta$.
To this end, we choose
\begin{equation}\label{defnu*}
\tilde\nu_*=\min\left\{\xi_1,\delta_1^\frac{1}{p-2},(\tfrac{1}{4}\sigma_*\epsilon\delta_1^\frac{1}{p-2})^\frac{m-(p-2)}{m},
4^{(2-p)\frac{1}{m}}(\tfrac{1}{16}\epsilon \delta_1^\frac{1}{p-2}\sigma_*)^{1-\frac{p-2}{m}}\right\}
\end{equation}
and
\begin{equation}\label{defetatilde}
\tilde\eta=\min\left\{\tfrac{1}{2}\xi_1,\tfrac{1}{32}\epsilon \delta_1^\frac{1}{p-2}\sigma_*\right\}.
\end{equation}
With these choices of $\tilde\nu_*$ and $\tilde \eta$, we
conclude that the lemma holds.
\end{proof}
\subsection{H\"older regularity for the function $v$ on $\hat Q$.}
The aim of this subsection is to prove H\"older continuity for $v$ at the point $(x_0,0)$, and
the proof is divided into three steps.

Step 1: Let $\tilde\nu_*$ be the constant in \eqref{defnu*}.
According to Lemma \ref{tailsatisfylemma}, there exists a constant $\sigma^\prime<1$ such that if $r<\sigma^\prime\rho$, then \eqref{tailsatisfyvv}
holds and hence the hypotheses of Proposition \ref{oscvproposition} is fulfilled. An application of Proposition \ref{oscvproposition} gives us that
\eqref{oscforv} holds.
Now, let us denote by $\tilde r_1$ the radius $r$.
From now on, $r_0$, $\tilde\mu_0^\pm$, $\bar\mu_0^\pm$, $\tilde\omega_0$ and $\Theta_0$ stand for $\rho$,
$\mu_v^\pm$, $\mu_{\bar v}^\pm$,
$\tilde\omega$ and $\Theta$, respectively.
Moreover, we define
$\tilde\omega_1=(1-\tilde\eta)\tilde\omega_0$ and $\Theta_1=\left(\frac{1}{4}\tilde\omega_1\right)^{2-p}$, where $\tilde \eta$ is the constant in \eqref{defetatilde}. Let $B>1$ be the constant in \eqref{defB} and we set $\tilde\lambda=\frac{1}{4}(1-\tilde\eta)^\frac{p-2}{sp}B^{-\frac{1}{sp}}$.
At this point, we define $r_1=\tilde\lambda \tilde r_1$, $Q_1^\prime=Q_{r_1}^{(B\Theta_1)}(x_0,0)$, $\tilde\mu_1^-=\essinf_{Q_1^\prime}v$ and $\tilde\mu_1^+=\esssup_{Q_1^\prime}v$. From this definition, we have $Q_1^\prime\subset \hat Q_0$ and the inequality \eqref{oscforv} reads
\begin{equation*}\begin{split}\essosc_{Q_1^\prime}v=\tilde\mu_1^+-\tilde\mu_1^-\leq\essosc_{\hat Q_0}v\leq \tilde\omega_1.\end{split}\end{equation*}
Let $\tilde\sigma<\sigma^\prime$ be a constant that will be determined later.
Next, we set $\tilde\delta=\tilde\lambda\tilde\sigma$.
For $j\geq1$, we set $\tilde r_j=\tilde\sigma r_{j-1}$
and $r_j=\tilde\lambda\tilde r_{j}=\tilde\delta r_{j-1}$. Define $\tilde\omega_j=(1-\tilde\eta)\tilde\omega_{j-1}=(1-\tilde\eta)^j\tilde\omega_0$, $\Theta_j=\left(\frac{1}{4}\tilde\omega_j\right)^{2-p}$,
$\hat Q_{j-1}=Q_{\frac{1}{4}\tilde r_j}^{(\Theta_{j-1})}(x_0,0)$, $Q_j^\prime=Q_{r_j}^{(B\Theta_j)}(x_0,0)$, $\tilde\mu_j^-=\essinf_{Q_j^\prime}v$, $\tilde\mu_j^+
=\esssup_{Q_j^\prime}v$, $\bar\mu_{j}^-=\essinf_{Q_{j}^\prime}\bar v$ and $\bar\mu_{j}^+
=\esssup_{Q_{j}^\prime}\bar v$. It is easily seen that $Q_j^\prime\subseteq \hat Q_{j-1}\subseteq Q_{j-1}^\prime\subseteq \hat Q$
holds for any $j\geq1$.
%Recalling that $v=\bar v^q$ on $\hat Q$, we conclude that for any $j\geq0$, there holds
%\begin{equation}\label{infsupbarvtildev}\begin{split}&\tilde\mu_j^-=\essinf_{Q_j^\prime}v=(\essinf_{Q_j^\prime}\bar v)^q=(\bar\mu_j^-)^q,
%\\ &\tilde\mu_j^+=\esssup_{Q_j^\prime}v=(\esssup_{Q_j^\prime}\bar v)^q=(\bar\mu_j^+)^q
%\end{split}\end{equation}
%and
%\begin{equation}\label{upperlower}1\leq\tilde\mu_j^-\leq \tilde\mu_j^+\leq \left(\frac{1+\xi_0}{\xi_0}\right)^q,\end{equation}
%where we used \eqref{barvv} in the last line.

Step 2: \emph{Let $j_0\geq1$ be a fixed integer. Assume that for $j=1,\cdots,j_0$, there holds
\begin{equation}\begin{split}\label{jAltTailallv}
\frac{\tilde r_j^\frac{n\kappa}{p-1}}{\Theta_{j-1}
^\frac{1}{m}}\widetilde{\mathrm{Tail}}_m((\bar v-\bar \mu_{j-1}^\pm)_\pm;Q_{j-1}^\prime)\leq\tilde\nu_*\tilde\omega_{j-1}.\end{split}\end{equation}
Through repeated application of Proposition \ref{oscvproposition}, we conclude that
\begin{equation}\begin{split}\label{assumptionforoscv}
\essosc_{\hat Q_{j-1}}v\leq\tilde\omega_j\end{split}\end{equation}
holds for any $1\leq j\leq j_0$. Then, we claim that there exists a constant $\tilde\sigma<\sigma^\prime$ depending only upon the data, such that if $\tilde r_j=\tilde\sigma r_{j-1}$ holds for any $j\geq1$,
then
\begin{equation}\begin{split}\label{j0AltTailallv}
\frac{\tilde r_{j_0+1}^\frac{n\kappa}{p-1}}{\Theta_{j_0}
^\frac{1}{m}}\widetilde{\mathrm{Tail}}_m((\bar v-\bar\mu_{j_0}^\pm)_\pm;Q_{j_0}^\prime)\leq\tilde\nu_*\tilde\omega_{j_0}.\end{split}\end{equation}}
To prove \eqref{j0AltTailallv}, we only address the case of the truncated function $(\bar v-\bar\mu_{j_0}^-)_-$ as the other case is similar.
We first consider the case when $j_0\geq2$.
For abbreviation, we write $\tilde U(y,t)=(\bar v-\bar\mu_{j_0}^-)_-(y,t)^{p-1}|y-x_0|^{-n-sp}$.
By Minkowski's inequality, we obtain
\begin{equation*}\begin{split}
&\frac{\tilde r_{j_0+1}^\frac{n\kappa}{p-1}}{\Theta_{j_0}
^\frac{1}{m}}\widetilde{\mathrm{Tail}}_m((\bar v-\bar\mu_{j_0}^-)_-;Q_{j_0}^\prime)
\leq
\sum_{l=2}^{j_0}T_l+T+T^\prime+T^{\prime\prime},
\end{split}\end{equation*}
where
\begin{equation*}\begin{split}
T_l= \frac{\tilde r_{j_0+1}^\frac{n\kappa}{p-1}}{\Theta_{j_0}
^\frac{1}{m}}\left(\int_{-B\Theta_{j_0} r_{j_0}^{sp}}^{0}
\left(\int_{B_{\frac{1}{4}\tilde r_{l-1}}(x_0)\setminus B_{\frac{1}{4}\tilde r_l}
(x_0)}\tilde U(y,t)\,\mathrm{d}y
\right)^\frac{m}{p-1}\,\mathrm{d}t\right)^\frac{1}{m},
\end{split}\end{equation*}
\begin{equation*}\begin{split}
T= \frac{\tilde r_{j_0+1}^\frac{n\kappa}{p-1}}{\Theta_{j_0}
^\frac{1}{m}}\left(\int_{-B\Theta_{j_0} r_{j_0}^{sp}}^{0}
\left(\int_{B_{\rho}(x_0)\setminus B_{\frac{1}{4}\tilde r_1}
(x_0)}\tilde U(y,t)\,\mathrm{d}y
\right)^\frac{m}{p-1}\,\mathrm{d}t\right)^\frac{1}{m},
\end{split}\end{equation*}
\begin{equation*}\begin{split}
T^{\prime}= \frac{\tilde r_{j_0+1}^\frac{n\kappa}{p-1}}{\Theta_{j_0}
^\frac{1}{m}}\left(\int_{-B\Theta_{j_0} r_{j_0}^{sp}}^{0}
\left(\int_{B_{\frac{1}{4}\tilde r_{j_0}
}(x_0)\setminus B_{r_{j_0}}
(x_0)}\tilde U(y,t)\,\mathrm{d}y
\right)^\frac{m}{p-1}\,\mathrm{d}t\right)^\frac{1}{m}
\end{split}\end{equation*}
and
\begin{equation*}\begin{split}
T^{\prime\prime}=\frac{\tilde r_{j_0+1}^\frac{n\kappa}{p-1}}{\Theta_{j_0}
^\frac{1}{m}}\left(\int_{-B\Theta_{j_0}r_{j_0}^{sp}}^{0}
\left(\int_{\mathbb{R}^n\setminus B_{\rho}(x_0)}\tilde U(y,t)\,\mathrm{d}y
\right)^\frac{m}{p-1}\,\mathrm{d}t\right)^\frac{1}{m}.
\end{split}\end{equation*}
For $l=2,\cdots, j_0$, we see that $B_{\frac{1}{4}\tilde r_{l-1}}(x_0)\times (-B\Theta_{j_0} r_{j_0}^{sp},0)\subseteq \hat
Q_{l-2}\subseteq \hat Q$.
In view of \eqref{barvv}, we find that
\begin{equation*}1\leq \essinf_{\hat Q_{l-2}} v\leq\esssup_{\hat Q_{l-2}} v\leq \left(\frac{1+\xi_0}{\xi_0}\right)^q.\end{equation*}
Taking into account
\eqref{assumptionforoscv}, we conclude that for $(y,t)\in \hat Q_{l-2}$, there holds
\begin{equation*}\begin{split}(\bar v(y,t)-\bar\mu_{j_0}^-)_-&\leq \essosc_{\hat Q_{l-2}}\bar v
=\esssup_{\hat Q_{l-2}}\bar v-\essinf_{\hat Q_{l-2}}\bar v
\\&=(\esssup_{\hat Q_{l-2}} v)^\frac{1}{q}-(\essinf_{\hat Q_{l-2}} v)^\frac{1}{q}
\\&\leq\gamma (\essinf_{\hat Q_{l-2}} v+\esssup_{\hat Q_{l-2}} v)^{\frac{1}{q}-1}\essosc_{\hat Q_{l-2}}v
\\&\leq \gamma \essosc_{\hat Q_{l-2}}v \leq \gamma \tilde\omega_{l-1}.\end{split}\end{equation*}
It follows that
\begin{equation*}\begin{split}
T_l\leq \gamma (1-\tilde\eta)^{-1}
B^\frac{1}{m}\tilde\sigma^\frac{n\kappa}{p-1}\left[(\tilde\lambda\tilde\sigma)^\frac{sp}{p-1}(1-\tilde\eta)^{-1}\right]^{j_0-l}\tilde
\omega_{j_0}.
\end{split}\end{equation*}
At this stage, we choose
\begin{equation}\begin{split}\label{tildesigmadefv}
\tilde\sigma=\min\left\{\tau_0,\ \sigma^\prime,\ \left[\tfrac{1}{100}\gamma^{-1}B^{-\frac{1}{m}}\tilde\nu_*(1-\tilde\eta)\right]
^\frac{p-1}{n\kappa}\right\},
\end{split}\end{equation}
where $\tau_0$ is the constant in \eqref{tau0}.
Consequently, we infer that
\begin{equation*}\begin{split}
\sum_{l=2}^{j_0}T_l&\leq \gamma (1-\tilde\eta)^{-1}B^\frac{1}{m}\tilde\sigma^\frac{n\kappa}{p-1}\tilde\omega_{j_0}\sum_{l=0}^\infty2^{-l}\leq\tfrac{1}{5}\tilde\nu_*
\tilde\omega_{j_0}.
\end{split}\end{equation*}
To estimate $T$, we note that
$B_{\rho}(x_0)\times (-B\Theta_{j_0} r_{j_0}^{sp},0)\subseteq \hat Q$. For $(y,t)\in
B_{\rho}(x_0)\times (-B\Theta_{j_0} r_{j_0}^{sp},0)$, we find that
\begin{equation*}\begin{split}(\bar v(y,t)-\bar\mu_{j_0}^-)_-&\leq \essosc_{\hat Q}\bar v=\omega_{\bar v} \leq \gamma \tilde\omega_0.\end{split}\end{equation*}
Taking into account that $(\tilde\lambda\tilde\sigma)^\frac{sp}{p-1}<\frac{1}{2}(1-\tilde
\eta)$, we obtain
\begin{equation*}\begin{split}
T&\leq \gamma B^\frac{1}{m}\tilde\sigma^\frac{n\kappa}{p-1}(\tilde\lambda\tilde\sigma)^{-\frac{sp}{p-1}}
\left(\frac{(\tilde\lambda\tilde\sigma)^\frac{sp}{p-1}}{1-\tilde\eta}\right)^{j_0}\tilde\omega_{j_0}\leq \tfrac{1}{5}\tilde\nu_*\tilde
\omega_{j_0},
\end{split}\end{equation*}
since $j_0>1$.
Next, we consider the estimate for $T^\prime$. In view of $B_{\frac{1}{4}\tilde\rho_{j_0}
}(x_0)\times (-B\Theta_{j_0} r_{j_0}^{sp},0)\subseteq \hat Q_{j_0-1}$, we conclude from \eqref{barvv} and \eqref{assumptionforoscv} that
for $(y,t)\in\hat Q_{j_0-1}$, there holds
\begin{equation*}\begin{split}
(\bar v(y,t)-\bar\mu_{j_0}^-)_-\leq \essosc_{\hat Q_{j_0-1}}\bar v\leq \gamma\essosc_{\hat Q_{j_0-1}} v\leq \gamma\tilde\omega_{j_0}
\end{split}\end{equation*}
and hence
\begin{equation*}\begin{split}
T^\prime&\leq \gamma \frac{\tilde r_{j_0+1}^\frac{n\kappa}{p-1}}{\Theta_{j_0}
^\frac{1}{m}}\tilde\omega_{j_0}r_{j_0}^{-\frac{sp}{p-1}}(B\Theta_{j_0}r_{j_0}
^{sp})^\frac{1}{m}= \gamma B^\frac{1}{m}\tilde\sigma^\frac{n\kappa}{p-1}\tilde
\omega_{j_0}\leq \tfrac{1}{5}\tilde\nu_*\tilde\omega_{j_0}.
\end{split}\end{equation*}
Finally, we come to the estimate of $T^{\prime\prime}$.
To this end, we first note that for $(y,t)\in \mathbb{R}^n\setminus B_{\rho}(x_0)$, there holds
\begin{equation*}\begin{split}(v(y,t)-\bar\mu_{j_0}^-)_-&\leq (\bar v-\bar\mu_{0}^-)_-(y,t)+\bar\mu_{j_0}^--\bar\mu_0^-\leq (\bar v-\bar\mu_{0}^-)_-(y,t)+\omega_{\bar v}
\\ &\leq (\bar v-\bar\mu_{0}^-)_-(y,t)+\gamma\tilde\omega_0.
\end{split}\end{equation*}
According to Step 1, we find that \eqref{tailsatisfyvv} holds, since $\tilde\sigma<\sigma^\prime$. This implies that
\begin{equation*}\begin{split}
T^{\prime\prime}&\leq
\tilde\sigma^{j_0\frac{n\kappa}{p-1}}\frac{\tilde r_1^\frac{n\kappa}{p-1}}{\Theta_0
^\frac{1}{m}}\widetilde{\mathrm{Tail}}_m((\bar v-\mu_{\bar v}^-)_-;\hat Q)
+\gamma \frac{\tilde r_{j_0+1}^\frac{n\kappa}{p-1}}{\Theta_{j_0}
^\frac{1}{m}}\tilde\omega_0\rho^{-\frac{sp}{p-1}}(B\Theta_{j_0}r_{j_0}^{sp})^\frac{1}{m}
\\&\leq \tilde\sigma^{j_0\frac{n\kappa}{p-1}} \tilde\nu_*\tilde\omega_0+\gamma B^\frac{1}{m}\tilde\sigma^\frac{n\kappa}{p-1}
(1-\tilde\eta)^{-j_0}\left(\frac{r_{j_0}}{r_1}\right)^\frac{sp}{p-1}\tilde\omega_{j_0}
,\end{split}\end{equation*}
since $\Theta_{j_0}\geq \Theta_0$ and $\rho>r_1$. At this point,
the estimate of $T^{\prime\prime}$
follows in a similar manner as the argument for $T$. This gives
$T^{\prime\prime}
<\tfrac{2}{5}\tilde\nu_*\tilde\omega_{j_0}$,
which proves \eqref{j0AltTailallv} for $j_0\geq2$. Finally, in the case $j_0=1$, we have
\begin{equation*}\begin{split}
&\frac{\tilde r_{2}^\frac{n\kappa}{p-1}}{\Theta_{1}
^\frac{1}{m}}\widetilde{\mathrm{Tail}}_m((\bar v-\bar\mu_{1}^-)_-;Q_{1}^\prime)
\leq T+T^\prime+T^{\prime\prime}<\tilde\nu_*\tilde\omega_{j_0},
\end{split}\end{equation*}
which proves the claim.

Step 3: \emph{Proof of the H\"older continuity of $v$ at $(x_0,0)$.}
Let $\tilde\sigma>0$ be the constant chosen according to \eqref{tildesigmadefv}.
We conclude from Proposition \ref{oscvproposition} and Step 1 that
\begin{equation*}\begin{split}
\essosc_{Q_{j_0+1}^\prime}v\leq
\essosc_{\hat Q_{j_0}}v\leq\tilde\omega_{j_0+1}.\end{split}\end{equation*}
This inequality means that \eqref{assumptionforoscv} holds for $j=j_{0}+1$.
Moreover, we repeat the arguments from Step 2 to obtain
\begin{equation}\begin{split}\label{vassumptionforoscforallj}
\essosc_{Q_{r_j}^{(\Theta_0)}(x_0,0)} v\leq
\essosc_{Q_j^\prime}v\leq\tilde\omega_j=(1-\tilde\eta)^j\tilde\omega_0\end{split}\end{equation}
for all $j\geq0$. For any fixed $r<\rho$,
there exists an integer $\hat j\geq0$, such that $\tilde\delta^{\hat j+1}\rho\leq r\leq \tilde\delta^{\hat j}\rho$. At this point, we use
\eqref{vassumptionforoscforallj} to get
\begin{equation}\begin{split}\label{Holderforv}
\essosc_{Q_r^{(\Theta_0)}(x_0,0)} v\leq
\essosc_{Q_{\hat j+1}}v\leq \tilde\omega_{\hat j+1}\leq \gamma\left(\frac{r}{\rho}\right)^{\beta_0}\tilde\omega_0,\end{split}\end{equation}
where $\beta_0=\frac{\ln(1-\tilde\eta)}{\ln\tilde\delta}$. This establishes a H\"older estimate for the function $v$ at the point $(x_0,0)$.
\subsection{The proof of Theorem \ref{main1}.}\label{finalsubsection}
This subsection is devoted to providing the final part of the proof of Theorem \ref{main1}.
In this subsection, we recover the index $i_0$ in our notation. For any fixed $r<\rho_{i_0}$, the inequality \eqref{Holderforv} reads
\begin{equation}\begin{split}\label{Holderforv1}
\essosc_{Q_r^{(\Theta)}(x_0,0)} v\leq
 \gamma\left(\frac{r}{\rho_{i_0}}\right)^{\beta_0}\omega_v\leq
 \gamma\left(\frac{r}{\rho_{i_0}}\right)^{\beta_0}\omega_{\bar v},\end{split}\end{equation}
 since $\tilde\omega_0=\omega_v$, $\Theta_0=\Theta$ and $\omega_v\leq \gamma_0\omega_{\bar v}$. Noting that $Q_r^{(\Theta)}(x_0,0)\subseteq \hat Q$, we deduce from \eqref{barvv} and \eqref{Holderforv1} that
 \begin{equation*}\begin{split}
\essosc_{Q_r^{(\Theta)}(x_0,0)} \bar v&=\esssup_{Q_r^{(\Theta)}(x_0,0)} v^\frac{1}{q}-\essinf_{Q_r^{(\Theta)}(x_0,0)} v^\frac{1}{q}
\\&\leq\gamma (\esssup_{Q_r^{(\Theta)}(x_0,0)} v+\essinf_{Q_r^{(\Theta)}(x_0,0)} v)^{\frac{1}{q}-1}\essosc_{Q_r^{(\Theta)}(x_0,0)} v
\\&\leq
 \gamma \essosc_{Q_r^{(\Theta)}(x_0,0)} v\leq
 \gamma\left(\frac{r}{\rho_{i_0}}\right)^{\beta_0}\omega_{\bar v}.\end{split}\end{equation*}
 Note that $\bar v(x,t)=(\mu_{i_0}^-)^{-1}
u(x,t_0+(\mu_{i_0}^-)^{q-p+1}t)$ and $\omega_{\bar v}=(\mu_{i_0}^-)^{-1}\omega_{i_0}$.
Transforming
back to the original function $u$, the above inequality gives
 \begin{equation*}\begin{split}
\essosc_{Q_r^{(\bar\theta)}(z_0)} u\leq
 \gamma\left(\frac{r}{\rho_{i_0}}\right)^{\beta_0}\omega_{i_0},\end{split}\end{equation*}
 where $\bar\theta=(\mu_{i_0}^-)^{q-p+1}\Theta$. In view of \eqref{ximuequivalent} and \eqref{gamma0}, we see that
 \begin{equation*}\begin{split}\bar\theta \geq \theta_{i_0}\left(4\frac{1+\xi_0}{1-\eta}\right)^{q-p+1}\left(\frac{\gamma_0}{4\xi_0}\right)^{2-p}
 \geq \bar\gamma\theta_0,\end{split}\end{equation*}
 where $\bar\gamma=4^{q-1}\left(\frac{1+\xi_0}{1-\eta}\right)^{q-p+1}\left(\frac{\gamma_0}{\xi_0}\right)^{2-p}$. Recalling the definition of $\omega_j$ from subsection \ref{nearzeroproof}, Step 3, we have $\omega_{i_0}=(1-\eta)^{i_0}\omega_0$ and hence
  \begin{equation*}\begin{split}
\essosc_{Q_r^{(\bar\gamma\theta_0)}(z_0)} u&\leq
 \gamma\left(\frac{r}{\rho_{i_0}}\right)^{\beta_0}\left(\frac{\rho_{i_0}}{\rho_1}\right)^{\alpha_0}\omega_{0}
 \\&\leq \gamma\left(\frac{r}{\rho_{i_0}}\right)^{\alpha}\left(\frac{\rho_{i_0}}{\rho_1}\right)^{\alpha}\omega_{0}
 \\&=\gamma\left(\frac{r}{\rho_1}\right)^{\alpha}\omega_{0}\leq \gamma_0\left(\frac{r}{R}\right)^{\alpha},
 \end{split}\end{equation*}
 where $\alpha=\min\{\alpha_0,\beta_0\}$.
 Here, the constant $\gamma_0$ depends on the data, $R$, $\|u\|_\infty$ and $\mathrm{Tail}_m(|u|;Q_{R})$.
 Hence, we conclude that the weak solution $u$ is locally H\"older continuous in $\Omega_T$.
 This completes the proof
of Theorem \ref{main1}.
\bibliographystyle{abbrv}

\end{document}